\pdfoutput=1
\RequirePackage{ifpdf}
\ifpdf 
\documentclass[pdftex]{sigma}
\else
\documentclass{sigma}
\fi

\numberwithin{equation}{section}

\usepackage{microtype}

\usepackage[all]{xy}

\usepackage{mathtools}
\usepackage{stackrel}

\usepackage{csquotes}

\newtheorem{thm}{Theorem}[section]
\newtheorem{lem}[thm]{Lemma}
\newtheorem{prop}[thm]{Proposition}
\newtheorem{cor}[thm]{Corollary}
\newtheorem{question}[thm]{Question}

\theoremstyle{definition}
\newtheorem{defn}[thm]{Definition}
\newtheorem{Example}[thm]{Example}

\newtheorem{rem}[thm]{Remark}

\newcommand{\N}{\mathbb{N}}
\newcommand{\Z}{\mathbb{Z}}

\newcommand{\R}{\mathbb{R}}
\newcommand{\C}{\mathbb{C}}

\newcommand{\cZ}{\mathcal{Z}}
\newcommand{\cO}{\mathcal{O}}
\newcommand{\cP}{\mathcal{P}}

\newcommand{\cX}{\mathcal{X}}
\newcommand{\cY}{\mathcal{Y}}
\newcommand{\cA}{\mathcal{A}}
\newcommand{\cB}{\mathcal{B}}
\newcommand{\cU}{\mathcal{U}}
\newcommand{\cV}{\mathcal{V}}

\newcommand{\cK}{\mathcal{K}}

\newcommand{\cC}{\mathcal{C}}

\newcommand{\Ct}{\mathrm{C}}
\newcommand{\Cz}{\Ct_0}
\newcommand{\Cb}{\Ct_b}
\newcommand{\Ch}{\Ct_h}
\newcommand{\CzRed}{\Ct_0^{\mathrm{red}}}
\newcommand{\CRed}{\Ct^{\mathrm{red}}}

\newcommand{\fP}{\mathfrak{P}}

\newcommand{\categoryfont}[1]{\mathsf{#1}}
\newcommand{\sLCH}[1][\Gamma]{\boldsymbol\sigma\categoryfont{LCH}_{#1}}
\newcommand{\sLCHz}[1][\Gamma]{\boldsymbol\sigma\categoryfont{LCH}_{#1}^2}
\newcommand{\sLCHp}[1][\Gamma]{\boldsymbol\sigma\categoryfont{LCH}_{#1}^+}
\newcommand{\sLCHzp}[1][\Gamma]{\boldsymbol\sigma\categoryfont{LCH}_{#1}^{+,2}}
\newcommand{\sCH}[1][\Gamma]{\boldsymbol\sigma\categoryfont{CH}_{#1}}
\newcommand{\sCHz}[1][\Gamma]{\boldsymbol\sigma\categoryfont{CH}_{#1}^2}

\newcommand{\LCHz}[1][\Gamma]{\categoryfont{LCH}_{#1}^2}
\newcommand{\LCHp}[1][\Gamma]{\categoryfont{LCH}_{#1}^+}
\newcommand{\LCHzp}[1][\Gamma]{\categoryfont{LCH}_{#1}^{+,2}}

\newcommand{\CHz}[1][\Gamma]{\categoryfont{CH}_{#1}^2}

\newcommand{\sCs}[1][\Gamma]{\boldsymbol\sigma\categoryfont{C}^*_{#1}}
\newcommand{\csCs}[1][\Gamma]{\categoryfont{c}\boldsymbol\sigma\categoryfont{C}^*_{#1}}

\newcommand{\CGCBG}[1][\Gamma]{\categoryfont{CGBG}_{#1}}
\newcommand{\CGCBGz}[1][\Gamma]{\categoryfont{CGBG}^2_{#1}}

\newcommand{\coarseStr}{\mathcal{E}}
\newcommand{\topology}{\mathcal{T}}
\newcommand{\Diag}[1][]{{\Delta_{#1}}}
\newcommand{\swapcross}{\mathbin{\bar\times}}
\newcommand{\HilbertMod}{\mathfrak{H}}

\DeclareMathOperator{\im}{im}
\DeclareMathOperator{\id}{id}

\DeclareMathOperator{\dist}{dist}
\DeclareMathOperator{\Var}{Var}

\DeclareMathOperator{\supp}{supp}
\DeclareMathOperator{\Pen}{Pen}
\DeclareMathOperator{\Cone}{Cone}

\DeclareMathOperator{\Prop}{prop}

\newcommand{\ev}[1]{\mathop{\textrm{ev}_{#1}}}

\newcommand{\maxtensor}{\mathbin{{\otimes}_{\mathrm{max}}}}
\newcommand{\maxcrossed}{\mathbin{{\rtimes}_{\mathrm{max}}}}
\newcommand{\redcrossed}{\mathbin{{\rtimes}_{\mathrm{red}}}}

\newcommand{\Homol}{\mathrm{E}}
\newcommand{\Cohom}{\mathrm{E}}
\newcommand{\HomolX}{\mathrm{EX}}
\newcommand{\CohomX}{\mathrm{EX}}

\newcommand{\HX}{\mathrm{HX}}
\newcommand{\CX}{\mathrm{CX}}
\newcommand{\KX}{\mathrm{KX}}

\newcommand{\K}{\mathrm{K}}
\newcommand{\Ktop}{\mathrm{K}^{\mathrm{top}}}
\newcommand{\KK}{\mathrm{KK}}

\newcommand{\EE}{\mathrm{E}}
\newcommand{\HS}{\mathrm{HS}}
\newcommand{\CS}{\mathrm{CS}}
\newcommand{\lf}{\mathrm{lf}}
\newcommand{\cs}{\mathrm{c}}

\newcommand{\textCstar}{\ensuremath{\mathrm{C}^*\!}}

\newcommand{\Roe}{\mathrm{C}^*}
\newcommand{\PseudoLoc}{\mathrm{D}^*}
\newcommand{\Loc}{\mathrm{C}^*_{\mathrm{L},\Gamma}}
\newcommand{\PseudoLocLoc}{\mathrm{D}^*_{\mathrm{L},\Gamma}}
\newcommand{\sHigCom}{\overline{\mathfrak{c}}}
\newcommand{\sHigCor}{\mathfrak{c}}
\newcommand{\sHigComRed}{\overline{\mathfrak{c}}^{\mathrm{red}}}
\newcommand{\sHigCorRed}{\mathfrak{c}^{\mathrm{red}}}
\newcommand{\sHigFuAlg}{\overline{\mathfrak{b}}}
\newcommand{\sHigFuAlgRed}{\overline{\mathfrak{b}}^{\mathrm{red}}}
\newcommand{\VanInfFuAlg}{\mathrm{B}_0}
\newcommand{\HigFuAlg}{\mathrm{B}_h}
\newcommand{\Kom}{{\mathfrak{K}}}
\newcommand{\Lin}{{\mathfrak{B}}}
\newcommand{\multiplier}{{\mathcal{M}}}
\newcommand{\doubleplier}{{\mathcal{D}}}
\newcommand{\Asymp}{\mathfrak{A}}
\newcommand{\admissible}{{\mathfrak{C}}}

\newcommand{\sfP}{\mathsf{P}}
\newcommand{\sfD}{\mathsf{D}}

\newcommand{\blank}{-}
\newcommand{\coarsecomp}{\Pi_0}
\newcommand{\myvdots}{{\vphantom{(}\raisebox{-0.3ex}[0pt][0pt]{\vdots}}}

\begin{document}
\allowdisplaybreaks

\newcommand{\arXivNumber}{2006.02053}

\renewcommand{\PaperNumber}{057}

\FirstPageHeading

\ShortArticleName{Equivariant Coarse (Co-)Homology Theories}

\ArticleName{Equivariant Coarse (Co-)Homology Theories}

\Author{Christopher WULFF}

\AuthorNameForHeading{C.~Wulff}

\Address{Mathematisches Institut, Georg--August--Universit\"at G\"ottingen,\\ Bunsenstr. 3-5, D-37073 G\"ottingen, Germany}
\Email{\href{mailto:christopher.wulff@mathematik.uni-goettingen.de}{christopher.wulff@mathematik.uni-goettingen.de}}
\URLaddress{\url{https://www.uni-math.gwdg.de/cwulff/eng.html}}

\ArticleDates{Received October 03, 2021, in final form July 15, 2022; Published online July 26, 2022}

\Abstract{We present an Eilenberg--Steenrod-like axiomatic framework for equivariant coarse homology and cohomology theories. We also discuss a general construction of such coarse theories from topological ones and the associated transgression maps. A large part of this paper is devoted to showing how some well-established coarse \mbox{(co-)}homology theories, whose equivariant versions are either already known or will be introduced in this paper, fit into this setup. Furthermore, a new and more flexible notion of coarse homotopy is given which is more in the spirit of topological homotopies. Some, but not all, coarse \mbox{(co-)}homology theories are even invariant under these new homotopies. They also led us to a meaningful concept of topological actions of locally compact groups on coarse spaces.}


\Classification{51F30; 55N35; 46L85}

\vspace{-3mm}

{\small \tableofcontents}

\section{Introduction}

Coarse \mbox{(co-)}homology theories have already been discussed before in various forms, see for example the axiomatic approach by Mitchener \cite{MitchenerCoarse} (only formulated for homology) or the spectral approach by Bunke and Engel \cite{BunkeEngel_CoarseCohomologyTheories, BunkeEngel_homotopy}, see also \cite{BunkeEngelKasprowskiWinges} for the equivariant version.

The work by Bunke and Engel et al.\ is undoubtedly the most ``modern'' manifestation of this topic and also the mightiest from a conceptual point of view.
However, it appears to be unnecessarily inconvenient for most practitioners of coarse geometry and coarse index theory \`a~la~John Roe, because it utilizes the language of infinity categories, which does not belong to the standard repertoire of that community and is usually not required in their research.\footnote{Note that the review of \cite{BunkeEngel_homotopy} in zbMATH independently expresses this opinion.}
Instead, coarse index invariants are often constructed directly in very specific models for certain coarse (co-)homology groups and reformulating them in the world of infinity categories would be rather futile.

Another basic idea of the Bunke--Engel-approach is that coarse (co-)homology theories should always be defined on the category of all bornological coarse spaces. While this is theoretically desirable, the plain unadulterated definitions of many examples work a priori only on certain subcategories, e.g., the one of all proper metric spaces. Additional work has to be invested to reformulate these definitions in order to squeeze the theories into this rather rigid set-up.

Therefore, it is imperative to not completely dismiss the complementary, ``service-oriented'' approach: Instead of defining coarse (co-)homology in a way which forces the users to adapt their examples to the theory, one can also formulate the theory in a flexible way which adapts itself to the concrete examples.
An Eilenberg--Steenrod like axiomatic set-up involving a suitable notion of admissible categories of coarse spaces seems to be perfect in this regard, and the purpose of the present exposition is to provide suchlike.

The present paper should somewhat be understood as an advancement of parts of Mitchener's aforementioned work \cite{MitchenerCoarse}, but we also set the focus a bit differently and take some newer developments into account. First of all, we generalize the axioms to the equivariant case and also consider cohomology theories. Second, we elaborate on how equivariant coarse \mbox{(co-)}homology theories can be constructed from topological \mbox{(co-)}homology theories for so-called $\sigma$-locally compact spaces by a coarsification process involving Rips complexes, which is an idea that goes back to the definition of coarse $\K$-theory in \cite{EmeMeyDualizing} and was discussed more generally in the non-equivariant case in \cite[Section~4]{EngelWulff}.
The advantage of coarse theories obtained as coarsifications is that they are related to the topological \mbox{(co-)}homology groups of Higson dominated coronas via so-called transgression maps, whose equivariant and relative versions are new in this paper.
Their importance stems from their close connection to \mbox{(co-)}assembly maps.
And third, a strong focus is laid on two example sections which illustrate how already established coarse \mbox{(co-)}homology theories, in particular the $\K$-theories of stable Higson coronas and Roe algebras, fit into this setup.

Our primary motivation for expanding further on the theory~-- and also for our specific choice of contents~-- was to lay the foundations for the author's subsequent publication \cite{wulff2020secondary} about secondary cup and cap products on coarse \mbox{(co-)}homology theories, which was written completely in the framework presented here, albeit without equivariance.
For this reason, we focus mainly on the material relevant therein instead of attempting to give a complete encyclopedic account on all the aspects of coarse (co-)homology which have been discussed in the literature so far. Nevertheless, we believe that our way of treating this topic can be highly suitable for other researchers in the area as well, so that we decided to write this separate discourse about it.

In this regard it should be mentioned that~-- although a pretty large part of this paper is devoted to giving examples~-- the equivariant versions of some of them are not being elaborated in a very sophisticated manner. In particular, the equivariance is implemented into ordinary coarse cohomology and the singular cohomology only very naively, which is sufficient for the purpose of providing examples, but probably not optimal for applications. Alexander--Spanier \mbox{(co-)}homology will be touched upon only briefly without any group actions at all.
Therefore, there is probably still room for improvement.

Two interesting by-products of our research should also be mentioned.
One aspect of coarse theory which always used to be a bit awkward to work with is the notion of coarse homotopy, because its appearance is somewhat different from that of a topological homotopy: A coarse homotopy between two coarse maps $X\to Y$ is a map $H\colon X\times \Z\to Y$ with certain properties (cf.~Definition~\ref{def:originalcoarsehomotopy}) such that the two maps can be identified with two restrictions of $H$ to certain subspaces of the form $X_\rho^\rho\coloneqq\{(x,\rho(x))\mid x\in X\}$,\footnote{The reason for this strange notation will become apparent in Sections~\ref{sec:coarsehomotopies} and~\ref{sec:flasquenessimplieshomotopyinvariance}.} where $\rho$ is a bornological map.
It is not only inconvenient that these boundaries $X_\rho^\rho$ of the homotopy are somewhat hidden inside of the definition, but also that the canonical coarse maps $X_\rho^\rho\to X$ do not even have to be coarse equivalences for general bornological $\rho$, so the inverse $X\to X_\rho^\rho\subset X\times \Z$ may be non-coarse. Furthermore, it is in general also not possible to simply concatenate two homotopies and hence transitivity of the homotopy relation is a priori not given.
These problems can be solved, and many authors do so, by assuming that the maps $\rho$ are not only bornological but even controlled. However, for most applications it is not really necessary to restrict the notion of coarse homotopies in this way.

Our attempt to alter the definition of coarse homotopies in order to resolve these inconveniences and transform its appearance to something resembling topological homotopies more closely led to the discovery of the broader and more flexible notion of generalized coarse homotopies (cf.~Definition~\ref{def:coarsehomotopy}). These are defined not on $X\times\Z$ but on $X\times I$, $I$ an closed interval, equipped with a coarse structure in which all slices $X\times\{t\}$ are coarsely embedded copies of $X$ and the intervals $\{x\}\times I$ can widen arbitrarily outside of bounded subsets.
Many coarse \mbox{(co-)}homology theories, in particular those arising from topological \mbox{(co-)}homology theories by a coarsification procedure but also the $\K$-theory of the Roe algebra, are invariant under generalized coarse homotopies.

A nice side-effect of our generalization of coarse homotopies is that it also leads us directly to a meaningful notion of topological actions of locally groups on coarse spaces. We give its definition, because we believe that such topological actions might be useful in future research, but immediately afterwards we restrict our attention to discrete groups for the development of our equivariant \mbox{(co-)}homology theories.

This paper is organized as follows. We recall the relevant definitions from coarse geometry and introduce our new notions of generalized coarse homotopies and topological group actions on coarse spaces in Section~\ref{sec:CoarseGeometry}. Afterwards we state our axioms of coarse \mbox{(co-)}homology theories and summarize elementary properties in Section~\ref{sec:CoarseCoHomologyTheories}. Section~\ref{sec:Examples} gives basic examples, which are defined by direct construction, before we exhibit a general procedure to obtain coarse \mbox{(co-)}homology theories from topologial ones via Rips complexes and introduce the so-called transgression maps in Section~\ref{sec:Sigmastuff}. The final Section~\ref{sec:FurtherExamples} continues the discussion of examples and explains how transgression maps are related to \mbox{(co-)}assembly maps.

\section{Coarse spaces}
\label{sec:CoarseGeometry}

In this section we recall the basics of coarse geometry, most of which can be found in \cite{RoeCoarseGeometry}.
\begin{defn}\label{def:coarsestructure}
A \emph{coarse structure} on a set $X$ is a collection $\coarseStr$ of subsets of $X\times X$, called the \emph{controlled sets} or \emph{entourages}, which contains the diagonal and is closed under the formation of subsets, inverses, products and finite unions. That is:
\begin{enumerate}\itemsep=0pt
\item[$1)$] $\Diag[X]\coloneqq\{(x,x)\mid x\in X\}\in\coarseStr$,
\item[$2)$] if $E\in\coarseStr$ and $E'\subset E$, then $E'\in\coarseStr$,
\item[$3)$] if $E\in\coarseStr$, then $E^{-1}\coloneqq\{(y,x)\mid(x,y)\in E\}\in\coarseStr$ (transposition),
\item[$4)$] if $E_1,E_2\in\coarseStr$, then
\[
E_1\circ E_2\coloneqq\{(x,z)\mid\exists y\in X\colon (x,y)\in E_1\wedge (y,z)\in E_2\}\in \coarseStr
\]
(composition),
\item[$5)$] if $E_1,E_2\in\coarseStr$, then $E_1\cup E_2\in\coarseStr$.
\end{enumerate}
The coarse structure $\coarseStr$ or the coarse space $(X,\coarseStr)$
\begin{enumerate}\itemsep=0pt
\item[$\bullet$] is called \emph{connected} if each element of $X\times X$ is contained in some entourage,
\item[$\bullet$] is said to be \emph{generated by a subset $S$} of the power set $\fP(X\times X)$ of $X\times X$ if it is the smallest coarse structure containing $S$.

Any subset $S$ of $\fP(X\times X)$ generates a unique coarse structure as follows: It is contained in at least one coarse structure, namely $\fP(X\times X)$, and the intersection of coarse structures is again a coarse structure. Thus, the unique coarse structure generated by $S$ is the intersection of all coarse structures in which it is contained.
Equivalently, it consists of all subsets of sets which are produced by taking finitely many unions, compositions and transpositions of sets in $S\cup\{\Delta\}$,

\item[$\bullet$] is called \emph{countably generated} if it is generated by a countable subset $S\subset\fP(X\times X)$.

This is equivalent to the existence of an increasing sequence $E_1\subset E_2\subset E_3\subset\cdots$ of entourages such that each entourage of $X$ is contained in one of them. To see this, let $S=\{F_1,F_2,F_3,\dots\}$ and then define $E_n$ as the union of the finitely many entourages $G_1\circ \dots\circ G_m$ with $m\leq n$ and $G_1,\dots,G_m\in\big\{\Delta_X,F_1,\dots, F_n,F_1^{-1},\dots,F_n^{-1}\big\}$.
\end{enumerate}
The pair $(X,\coarseStr)$ is then called a (countably generated/connected) \emph{coarse space}.
If there is no ambiguity, we will usually call $X$ a coarse space, the coarse structure being understood implicitly.

If $\coarseStr$ is a coarse structure on $X$ and $A\subset X$ is a subset, then
\(\coarseStr_A\coloneqq\{E \cap (A\times A) \mid E\in \coarseStr\}\)
is a coarse structure on $A$, called the \emph{restricted coarse structure}. A pair $(X,A)$ is called a \emph{pair of coarse spaces} if $X$ is a coarse space and $A\subset X$, where we consider $A$ as equipped with the restricted coarse structure.
\end{defn}

Important examples are of course metric spaces $(X,d)$: an entourage in the metric coarse structure on $X$ is a subset of $X\times X$ that is contained in $E_R\coloneqq\{(x,y)\in X\times X\mid d(x,y)\leq R\}$ for some $R\geq 0$.
It is countably generated by the entourages $E_n$ with $n\in\N$.
This construction works even if we allow the metric to take on the value $\infty$.
Then the metric coarse structure is connected if and only if the metric takes on only finite values.
If a coarse structure comes from a metric in this way, we call it \emph{metrizable}.
Roe proved the following metrizability lemma under the additional assumption of coarse connectedness, but its proof works equally well for all coarse spaces.
\begin{lem}[{\cite[Theorem 2.55]{RoeCoarseGeometry}}]\label{lem:metrizability}
A coarse space is metrizable if and only if it is countably generated.
\end{lem}

If $A\subset X$, then the restricted coarse structure on $A$ is the same as the metric coarse structure of the restricted metric.

\begin{defn}For $A\subset X$ and $E\in\coarseStr$ we call the set
\[
\Pen_E(A)\coloneqq E\circ A\coloneqq\{x\in X\mid \exists y\in A\colon (x,y)\in E\}
\]
the \emph{$E$-penumbra} of $A$. We will call a set $B\subset X$ a \emph{penumbra} of $A$, if it is contained in $\Pen_E(A)$ for some entourage $E$, or equivalently, if it is equal to $\Pen_E(A)$ for some entourage~$E$. If~$\Pen_E(A)=X$, then we say that $A$ is \emph{$E$-dense} in $X$, and $A$ is \emph{coarsely dense} in $X$ if it is $E$-dense for some entourage $E$.

For $A=\{y\}$ we call $\Pen_E(\{y\})$ the \emph{$E$-ball} around $y\in X$.
A subset $K\subset X$ is called \emph{bounded} if it is contained in some $E$-ball, or equivalently, if it is itself some $E$-ball.
A family of subsets $\{K_i\}_{i\in I}$, $K_i\subset X$, is called \emph{uniformly bounded} if the union $\bigcup_{i\in I}K_i\times K_i\subset X\times X$ is an entourage; in particular, each $K_i$ is bounded in this case.
\end{defn}

In the example of metric spaces, the $E_R^\circ\coloneqq\{(x,y)\in X\times X\mid d(x,y)<R\}$-penumbras of arbitrary subsets are exactly their open $R$-neigh\-bor\-hoods for each $R>0$.
A subset $K\subset X$ is bounded in the coarse geometric sense iff $K\subset \Pen_{E_R^\circ}(\{y\})=B_R(y)$ for some $R>0$ and some $y\in X$, i.e., if it is bounded in the metric sense.
A family of subsets of $X$ is uniformly bounded iff their diameters are uniformly bounded.

\begin{defn}\label{def:discreteboundedgeometry}
A coarse space $X$ is called \emph{locally finite} or \emph{discrete} if each bounded set is finite and it is called \emph{uniformly locally finite} or \emph{uniformly discrete} if for each uniformly bounded family $\{K_i\}_{i\in I}$ of subsets of $X$ the number of points in $K_i$ is uniformly bounded over all $i\in I$.

The space is said to have \emph{bornologically bounded geometry} if it posseses a coarsely dense and locally finite subspace, called a \emph{discretization}, and it is said to have \emph{coarsely bounded geometry} if it posseses a coarsely dense and uniformly locally finite subspace, called a \emph{uniform discretization}.
\end{defn}

The word ``discrete'' usually refers to a topological property, and indeed it is motivated by the relation between coarse structure and topology which is usually required if both of them are present at the same time.
\begin{defn}\label{defn:topcoarsespace}
A \emph{topological coarse space} is a~set $X$ equipped with both a coarse structure~$\coarseStr$ and a locally compact Hausdorff topology $\topology$ such that there is an entourage $E_0$ which is a~neighborhood of the diagonal in $X\times X$ and every bounded subset with respect to $\coarseStr$ is relatively compact with respect to $\topology$.
\end{defn}

The two properties in the definition ensure that in this case our coarse geometric notion of discreteness agrees with the topological notion of discreteness.
Furthermore, using the first property and Zorn's lemma one sees that each topological coarse space $X$ posseses a maximal subset $X'\subset X$ with the property that $E_0\cap (X'\times X')\subset \Diag[X]$, and this maximal subset is $\big(E_0\cup E_0^{-1}\big)$-dense and discrete in the topological and hence also the coarse geometric sense. Thus:
\begin{lem}\label{lem:topbbg}All topological coarse spaces have bornologically bounded geometry.
\end{lem}

An example of a coarse space without bornologically bounded geometry is the space of all bounded functions $\N\to\N$ equipped with the supremum metric.
It is, of course, not a topological coarse space, because sufficiently large balls in it are not relatively compact.
Although it is a topologically discrete metric space,
it does not contain any coarsely dense coarsely discrete subset. More precisely, it is easy to see that if $X'$ is an $R$-dense subset of this space, then any $3R$-ball must contain infinitely many points of~$X'$.

Let us use this opportunity to clarify the relation between metrics and coarse structures a bit more. The following lemma is in fact obvious from the definitions and Lemmas~\ref{lem:metrizability} and~\ref{lem:topbbg}.

\begin{lem}
A coarse structure and a topology on a set coming from a metric constitute a~topological coarse structure if and only if the metric is proper. Hence, all proper metric spaces have bornologically bounded geometry.

As a partial converse, all countably generated locally finite coarse structures are induced by proper metrics.
\end{lem}

Better still, the question of simultaneous metrizability of topologies and coarse structures has been completely answered by Wright~\cite{Wright_Simul}. His main theorem is as follows.

\begin{thm}[{\cite[Theorem 2.5 and its proof]{Wright_Simul}}]\label{thm:Wrightmetrizability}
Let $d_t$, $d_c$ be two metrics on a set $X$ such that the coarse structure induced by $d_c$ contains an entourage $E_0$ which is open with respect to the topology induced by $d_t$. Then there exists a metric $d$ on $X$ producing the same topology as $d_t$ and quasi-isometric to $d_c$. More precisely,
\[
\forall x,y\in X\colon\ d(x,y)-1\leq d_c(x,y)<(6R+1)(d(x,y)+1)
\]
with $R>\sup\{\lceil d_c(x,y)\rceil\mid(x,y)\in E_0\}$.
\end{thm}

Strictly speaking, Wright did not allow metrics which take the value infinity, but the more general case follows readily: We may always assume that $d_t$ takes only finite values and then apply his results to each coarse component separately. The fact that the same constant $R$ can be chosen for all coarse components implies that the individual metrics can be combined to a new metric $d$ on $X$ which is quasi-isometric to $d_c$.

As a direct consequence, \cite[Corollary 2.6]{Wright_Simul} then gives a necessary and sufficient condition for the simultaneous metrizability of a coarse structure and a topology on a set $X$.
Note that this result also covers cases in which the coarse structure and topology only satisfy one of the compatibility condition, that there is an open entourage, but not the other one, i.e., bounded subsets might not be relatively compact. This corresponds to the metric not being proper.
Instead of recalling Wright's corollary here, we prove the following similar statement about the metrizability of our topological coarse spaces defined in Definition~\ref{defn:topcoarsespace}.

\begin{cor}
A topological coarse space $X$ is metrizable, meaning that there is a necessarily proper metric inducing both the coarse structure and topology, if and only if the coarse structure is countably generated and the topology on each coarse component is second-countable.
\end{cor}

\begin{proof}
If both topology and coarse structure come from a proper metric, then the coarse structure is countably generated.
Furthermore, for each point $x\in X$ and each $n\in\N\setminus\{0\}$ let~$\cU_{x,n}$ be a finite set of $\frac1n$-balls covering the compact subset $\overline{B_n(x)}$. Then $\bigcup_{n=1}^\infty\cU_{x,n}$ is a~countable base for the topology on the coarse component which $x$ is contained in.

For the other direction, we have to find metrics $d_t$, $d_c$ as in Theorem~\ref{thm:Wrightmetrizability}. For $d_c$ we simply take the metric given by Lemma~\ref{lem:metrizability}.

It suffices to construct $d_t$ on each coarse component $C$ separately and then one can choose any common extension to the whole space. Let $C^+=C\cup \{\infty\}$ denote the one-point compactification of the locally compact Hausdorff space $C$ and choose any $x\in C$.
Furthermore, denote by $E_1\subset E_2\subset E_3\subset\cdots$ an increasing sequence of entourages such that each entourage of $X$ is contained in one of them. Therefore each compact $K\subset C$ is contained in one of the relatively compact subsets $E_n\circ \{x\}$ and we conclude that the sets $C^+\setminus \overline{E_n\circ \{x\}}$ are a neighborhood base for $\infty$. Thus, the compact Hausdorff space $C^+$ is second-countable, too, and hence metrizable by Urysohn's metrization theorem. Restricting the so obtained metric to $C$ is what we choose as~$d_t|_C$.
\end{proof}

In a topological coarse space, the relatively compact subsets are exactly the finite unions of bounded ones.
Note that these might not be bounded themselves if the space is not coarsely connected.
We generalize this notion to all coarse spaces.
\begin{defn}\label{def:relativelycompact}
A subset of a coarse space is called \emph{relatively compact} if it is a finite union of bounded subsets.
\end{defn}

Before we can recall and introduce some more properties and constructions of coarse spaces, we will first have to discuss various properties of maps between coarse spaces.

\begin{defn}\label{def:propertiesmapsbetweencoarsespaces}
Two maps $f,g\colon S\to Y$ from an arbitrary set $S$ into a coarse space $Y$ are called \emph{close} if the set $\{(f(x),g(x))\mid x\in S\}$ is an entourage of $Y$.
A map $f\colon X\to Y$ between two coarse spaces $X$, $Y$ is called
\begin{enumerate}\itemsep=0pt
\item[$1)$] \emph{controlled} if $f\times f$ maps entourages of $X$ to entourages of $Y$,
\item[$2)$] \emph{proper} if the preimages of bounded subsets of $Y$ under $f$ are relatively compact in $X$ or, equivalently, if the preimages of relatively compact subsets of $Y$ under $f$ are relatively compact in $X$,
\item[$3)$] \emph{coarse} if it is both controlled and proper,
\item[$4)$] \emph{a coarse equivalence} if it is a coarse map and if there exists another coarse map $g\colon Y\to X$ such that both $g\circ f$ and $f\circ g$ are close to the respective identities,
\item[$5)$] \emph{bornological} if it maps bounded subsets to bounded subsets.
\end{enumerate}
\end{defn}

It is immediate from the definitions that controlled maps and in particular coarse maps are bornological.

If the coarse structure on $X$ is generated by a subset $S\subset \fP(X\times X)$, then in order for the map $f$ to be controlled it is sufficent to check that $f\times f$ maps all entourages in the generating set $S$ to entourages of $Y$.

The composition of two maps between coarse spaces satisfying one of the five properties in the definition again have this property. Furthermore, closeness is an equivalence relation and if $f,g\colon S\to Y$ are close, $h\colon Y\to Z$ is a controlled map and $e\colon T\to S$ is any map, then the maps $h\circ f\circ e$ and $h\circ g\circ e$ are close.
In particular the following categories are well defined.

\begin{defn}\label{def:coarsecategories}
The (\emph{closeness\footnote{The strange expression ``closeness category'' is inspired by the similarity to the well-established notion of ``homotopy category''.}$)$ category of coarse spaces} is the category whose objects are coarse spaces and whose morphisms are (closeness classes of) coarse maps.

The (\emph{closeness$)$ category of pairs of coarse spaces} is the category whose objects are pairs of coarse spaces and whose morphisms are (closeness classes of) coarse maps between the bigger spaces which preserve the subspaces.
\end{defn}

The first of these will be considered a subcategory of the second via the faithful functor $X\mapsto(X,\varnothing)$, and sometimes we also make implicit use of the full functors $(X,A)\to X$ and $(X,A)\to A$ when talking about naturality under coarse maps between pairs of coarse spaces.

Coarse equivalences are exactly those coarse maps whose closeness classes are isomorphisms in the closeness category of coarse spaces. Hence we define a coarse equivalence between pairs of coarse spaces to be a map which represents an isomorphism in the closeness category of pairs of~coarse spaces. These are exactly the coarse equivalences between the big spaces which restrict to~coarse equivalences between the subspaces.

Finally, there are two more definitions that we have to recall.

\begin{defn}
{\samepage The \emph{product} of two coarse spaces $(X,\coarseStr_X)$, $(Y,\coarseStr_Y)$ is the cartesian product of sets $X\times Y$ equipped with the \emph{product coarse structure} $\coarseStr_{X\times Y}$ which is generated by the entourages
\begin{gather*}
E\swapcross F\coloneqq\{((x,y),(x',y'))\in (X\times Y)\times(X\times Y)\mid (x,x')\in E\wedge (y,y')\in F\}
\end{gather*}
for all $E\in \coarseStr_X$ and $F\in\coarseStr_Y$.}

The product of two pairs of coarse spaces $(X,A)$ and $(X,B)$ is $(X,A)\times (Y,B)\coloneqq (X\times Y,A\times Y\cup X\times B)$ equipped with the product coarse structure on $X\times Y$.
\end{defn}

To make it clear, this is not the cartesian product of coarse spaces in the category theoretic sense,
because the projection maps $X\times Y\to X$ and $X\times Y\to Y$ are not coarse maps: they are controlled, but not proper.
It is simply the coarse structure which is usually considered on the set $X\times Y$.

If $(X,d_X)$ and $(Y,d_Y)$ are metric spaces, then the product coarse structure is the one induced by any metric in the quasi-isometry class of $d_X+d_Y$ or, equivalently, $\max\{d_X,d_Y\}$.
Also, if $\coarseStr_X$ and $\coarseStr_Y$ are connected, countably generated, (uniformly) locally finite or of bornologically/coarsely bounded geometry, then their products have the respective property, too. And if $X$ and $Y$ are topological coarse spaces, then so is $X\times Y$ equipped with the product coarse structure and the product topology.

The bounded subsets of the product coarse structure are exactly those which are contained in the cross product of a bounded subset of $X$ and a bounded subset of $Y$.
If $f,f'\colon S\to X$ are close and $g,g'\colon T\to Y$ are close, then the maps $f\times g, f'\times g'\colon S\times T\to X\times Y$ are also close with respect to the product coarse structure.
If $f\colon X\to X'$ and $g\colon Y\to Y'$ are two maps between coarse spaces which satisfy one of the five properties of Definition~\ref{def:propertiesmapsbetweencoarsespaces}, then $f\times g\colon X\times Y\to X'\times Y'$ has the same property, too.
In particular, we can take the cross product of morphisms in all of the four coarse categories of Definition~\ref{def:coarsecategories}.

\begin{defn}\label{def:flasqueness}
A coarse space $X$ is called \emph{flasque} if there exists a coarse map $\phi\colon X\to X$ such that:
\begin{enumerate}\itemsep=0pt
\item[$\bullet$] $\phi$ is close to the identity map,
\item[$\bullet$] for any bounded (and hence also relatively compact) subset $K\subset X$ there exists $N_K\in\N$ such that $\im(\phi^n)\subset X\setminus K$ for all $n\geq N_K$,
\item[$\bullet$] the family $\{\phi^n\}_{n\in\N}$ is equicontrolled, that is, for each entourage $E$ there is an entourage $F$ such that $(\phi^n\times\phi^n)(E)\subset F$ for all $n\in\N$.
\end{enumerate}
A pair of spaces $(X,A)$ is called \emph{flasque} if $X$ is flasque with the map $\phi$ preserving the subspace~$A$.
\end{defn}

The prototypical example of a flasque space are the product coarse spaces $X\times\N$ with $\phi(x,k)=(x,k+1)$, where $\N$ carries the canonical metric coarse structure.

\subsection{Coarse homotopies}
\label{sec:coarsehomotopies}

The best-known notion of coarse homotopy has been introduced in \cite[Definition 11.1]{HigsonPedersenRoe} under the name ``Lipschitz homotopy''. However, a serious deficiency of their original definition has been pointed out in \cite[Remark 3.18]{BartelsSqueezing}, leading essentially to the following corrected definition of coarse homotopy.

\begin{defn}\label{def:originalcoarsehomotopy}
A \emph{coarse homotopy} between two coarse maps between coarse pairs $f,g$: $(X,A)\to (Y,B)$ is a map $H\colon X\times\Z\to Y$ which maps $A\times \Z$ to $B$ such that
\begin{enumerate}\itemsep=0pt
\item[$\bullet$] there are two bornological functions $\rho^\pm\colon X\to\Z$ such that for all $x\in X$, $n\in\Z$ we have
\begin{gather*}
n\leq\rho^-(x)\implies H(x,n)=f(x),
\\n\geq\rho^+(x)\implies H(x,n)=g(x),
\end{gather*}
\item[$\bullet$] and the restriction of $H$ to
\[
X_{\rho^-}^{\rho^+}\coloneqq\{(x,n)\in X\times\Z\mid\rho^-(x)\leq n\leq\rho^+(x)\}
\]
is a coarse map.
\end{enumerate}
\end{defn}
There seems to be some disagreement in the literature whether the functions $\rho^\pm$ are demanded to be only bornological maps or even controlled or coarse maps.
The author is not aware of two maps which are coarsely homotopic in the first version but not in the second or third, although such maps might very well exist.

Working with controlled maps is, of course, more convenient. In particular, in this case it is also possible to reparametrize the coarse homotopy such that $\rho^+\equiv 0$ or $\rho^-\equiv 0$ and hence transitivity of the coarse homotopy relation follows immediately by gluing together two reparametrized coarse homotopies (cf.\ \cite[Theorem 2.4]{MitchenerNorouzizadehSchick}).
However, the author is not aware of any homotopy invariance result which only works in this version but not in the one with only bornological $\rho^\pm$.

This classical notion of coarse homotopy appears to be rather unsatisfactory for two related reasons: First, there are no clear boundaries of the homotopy at which it can be evaluated, but rather strange open ends instead. The boundaries
\[
X_{\rho^\pm}^{\rho^\pm}\coloneqq\big\{(x,n)\in X\times\Z\mid n=\rho^\pm(x)\big\}
\]
are only hidden in the definition in the existence of $\rho^\pm$ and they are not even coarsely equivalent to $X$, if $\rho^\pm$ are not controlled. And second, the relation given by coarse homotopy is not a priori transitive, so one has only an equivalence relation generated by coarse homotopy and not given exactly by it.

These drawbacks motivated the author to the following widening of the notion of coarse homotopy, which brings back the true feeling of homotopy known from topology. As it happens, it closely resembles and also generalizes an even older notion of coarse homotopy between coarse maps between proper metric spaces that was introduced in \cite[Definitions 1.2 and~1.3]{HigsonRoeHomotopy}.

\begin{defn}\label{def:coarsehomotopy}
Let $(X,\coarseStr_X)$ be a coarse space, $I=[a,b]$ an interval and $\cU= \{U_x\}_{x\in X}$ a~collection of neigborhoods of the diagonal $\Diag[I]$ in $I\times I$.
We define the coarse structure $\coarseStr_{\cU}$ on $X\times I$ to be the one generated by the entourages
\begin{gather*}
E\swapcross\Diag[I]=\{((x,s),(y,s))\mid (x,y)\in E\wedge s\in I\}\qquad\text{for}\quad
E\in\coarseStr_X
\end{gather*}
and
\begin{gather*}
E_{\cU}\coloneqq\{((x,s),(x,t))\mid x\in X\wedge (s,t)\in U_x\}.
\end{gather*}

A \emph{generalized coarse homotopy} between two coarse maps $f,g\colon (X,\coarseStr_X)\to (Y,\coarseStr_Y)$ is a coarse map $(X\times I,\coarseStr_\cU)\to (Y,\coarseStr_Y)$ for some $\cU$ as above which restricts to $f$ on $X\times\{a\}$ and to $g$ on $X\times\{b\}$.

If $f,g\colon (X,A)\to (Y,B)$ are coarse maps between pairs of coarse spaces, then a \emph{generalized coarse homotopy} between them is a generalized coarse homotopy between the absolute coarse maps $f,g\colon X\to Y$ which takes $A\times I$ to $B$.

In both cases we call $f$ \emph{generalized coarsely homotopic} to $g$.
\end{defn}

One big advantage of generalized coarse homotopy is that it already defines an equivalence relation. In contrast to the usual coarse homotopy, transitivity is simply obtained by gluing together two of these homotopies.

Note that the coarse structure $\coarseStr_\cU$ is countably generated if and only if the coarse structure~$\coarseStr_X$ is countably generated.

Furthermore, the bounded subsets of $(X\times I,\coarseStr_\cU)$ are exactly those which are contained in $K\times I$ for some bounded $K\subset X$.
To see this, assume that $K\subset X$ is bounded and choose $x\in X$. Then $K\times\{x\}$ is an entourage of $X$ and compactness of $I$ implies that there is $n\in\N$ such that the $n$-fold composition $U_x^n=U_x\circ \dots\circ U_x$ is equal to all of $I\times I$. Thus, $K\times I= ((K\times \{x\})\swapcross \Diag[I])\circ E_\cU^n\circ\{(x,a)\}$ is bounded.
Conversely, if $K\times I$ is bounded, then it is contained in
\[
(E_1\swapcross\Diag[I])\circ E_\cU\circ (E_2\swapcross\Diag[I])\circ E_\cU\circ\dots\circ E_\cU\circ (E_k\swapcross\Diag[I])\circ \{(x,s)\}
\]
for some entourages $E_1,\dots,E_k$ and some $(x,s)\in X\times I$. But then $K\subset E_1\circ\dots\circ E_k\circ\{x\}$ is bounded.

As a direct consequence we note that if $(X,\coarseStr_X,\topology_X)$ is a topological coarse space and $\topology_{X\times I}$ denotes the product topology on $X\times I$, then $(X\times I,\coarseStr_\cU,\topology_{X\times I})$ is again a topological coarse space. Indeed, in this case the above characterization of bounded subsets of $(X\times I,\coarseStr_\cU)$ implies that they are relatively compact with respect to $\topology_{X\times I}$. Furthermore, if $E_0$ denotes an entourage of $X$ which is a neighborhood of the diagonal in $X\times X$, then
\[
(E_0\swapcross\Diag[I])\circ E_\cU\circ E_\cU^{-1}\circ (E_0\swapcross\Diag[I])^{-1}=\bigcup_{(x,t)\in X\times I}(\underbrace{E_0\circ\{x\}\times U_x\circ\{t\}}_{\textnormal{neighborhood of }(x,t)})^2
\]
is an entourage in $\coarseStr_\cU$ which is a neighborhood of the diagonal in $(X\times I)^2$.

Of course, the interval can be reparametrized arbitrarily, but it is convenient to have some flexibility in the notation. For example, it helps us to see how coarse homotopies $H\colon X\times\Z\allowbreak\to X$ give rise to generalized coarse homotopies:
We define the family $\cU= \{U_x\}_{x\in X}$ of open neighborhoods of the diagonal in $[-\infty,\infty]^2$ by
\[
U_x\coloneqq [-\infty,\rho^-(x))^2\cup\big\{(x,y)\in \R^2\mid d(x,y)<1\big\} \cup (\rho^+(x),\infty]^2
 \]
and then a coarse homotopy $\tilde H\colon (X\times [-\infty,\infty],\coarseStr_\cU)\to (Y,\coarseStr_Y)$ between $f$ and $g$ in the sense of Definition~\ref{def:coarsehomotopy} is given by the formula
\[
\tilde H(x,s)\coloneqq\begin{cases}
f(x),&s=-\infty,\\H(x,\lfloor s\rfloor),&s\in\R,\\g(x),&s=+\infty.
\end{cases}
\]
If $H$ maps $A\times \Z$ to $B$, then $\tilde H$ maps $A\times I$ to $B$. Hence this construction also works in the relative case. We have just shown:

\begin{lem}
If two coarse maps between coarse spaces or pairs of coarse spaces are coarsely homotopic, then they are also generalized coarsely homotopic.
\end{lem}

Although generalized coarse homotopies appear to be a lot more flexible than normal coarse homotopies, the author is unaware of an example which disproves the converse of this lemma.

Just like the classical notion of coarse homotopy, generalized coarse homotopy is also compatible with composition of coarse maps. One of the compositions is trivial: If $H\colon(X\times I,\coarseStr_\cU)\to (Y,\coarseStr_Y)$ is a coarse homotopy between two coarse maps $f,g\colon X\to Y$ and if $h\colon Y\to Z$ is another coarse map, then $h\circ H$ clearly is a coarse homotopy between $h\circ f$ and $h\circ g$. For the other composition let $e\colon W\to X$ be a coarse map. Then
\[
e\times\id_I\colon\ (W\times I,\coarseStr_{e^*\cU})\to (X\times I,\coarseStr_\cU),
\]
where $e^*\cU\coloneqq\{U_{e(w)}\}_{w\in W}$, is a coarse map and consequently $H\circ (e\times\id_I)$ is a coarse homotopy between $f\circ e$ and $g\circ e$.

Again, these constructions work equally well for coarse maps between pairs of coarse spaces. We thus obtain categories:

\begin{defn}
The (\emph{generalized$)$ homotopy category of coarse spaces} is the category whose objects are coarse spaces and whose morphisms are homotopy classes of coarse maps.
Analogously, we obtain the (\emph{generalized$)$ homotopy category of pairs of coarse spaces}.
\end{defn}

Note that the generalized homotopy categories are quotients of the homotopy categories and these in turn are quotients of the closeness coarse categories defined in Definition~\ref{def:coarsecategories}, because closeness clearly implies coarse homotopy.

\subsection{Group actions on coarse spaces}
In this subsection we discuss actions of groups $\Gamma$ on coarse spaces $X$. On spaces, we always let the group act on the right.

At the end we will almost exclusively be interested in proper isocoarse (also known as isometric) actions of discrete groups, mostly because our explicit construction of $\Gamma$-equivariant coarse \mbox{(co-)}homology theories via a Rips complex construction in Section~\ref{sec:Ripscomplex} only works easily in this case.
Otherwise the coarse spaces do not have $\Gamma$-invariant discretizations and we would have to work with more complicated versions of the Rips complex instead, compare Emerson--Meyer \cite[Section~2.4]{EmeMeyDescent} for a definition in the case of topological coarse spaces.
Also, the theory of Roe algebras and localization algebras, which we will review in Sections~\ref{sec:RoeAlg} and~\ref{sec:Assembly}, requires proper isometric group actions by discrete groups on proper metric spaces.

However, the idea of our generalized coarse homotopies directly prompts us to propose more generally the following definition of ``topological'' actions of locally compact groups on coarse spaces.
It is totally unclear whether this notion might turn out to be useful in future research or not and we also do not claim that this definition is final.
Possibly, it has to be modified for applications and the reader shall feel free to do so.
Our definition should just serve as a proof of concept that meaningful notions of actions of locally compact topological groups on coarse spaces not carrying a topology exist.

\begin{defn}\label{defn:coarsegroupaction}
A \emph{topological action} of a locally compact group $\Gamma$ on a coarse space $(X,\coarseStr_X)$ is a group action $X\times\Gamma\to X$, $(x,\gamma)\mapsto x\gamma$ which is also a controlled map with respect to some coarse structure on $X\times\Gamma$ defined in the following way:
Given a collection $\cU= \{U_x\}_{x\in X}$ of neigborhoods of the diagonal $\Diag[\Gamma]$ in $\Gamma\times \Gamma$, we define the coarse structure $\coarseStr_{\cU}$ on $X\times\Gamma$ to be the one generated by the entourages
\begin{gather*}
E\swapcross \Diag[L]=\{((x,\gamma),(y,\gamma))\mid \gamma\in L \wedge (x,y)\in E\}\qquad
\text{for}\quad L\subset \Gamma\text{ compact and } E\in\coarseStr_X
\end{gather*}
and
\begin{gather*}
E_{\cU}\coloneqq\{((x,\gamma),(x,\gamma'))\mid x\in X\wedge (\gamma,\gamma')\in U_x\}.
\end{gather*}
\end{defn}

In \cite[Section 2.1]{EmeMeyDescent} and \cite[Section 2.2]{EmeMeyEquivariantCoassembly} Emerson and Meyer also mention continuous and coarse actions of locally compact groups on coarse spaces, but one should keep in mind that their coarse spaces are always topological coarse spaces and continuity is meant in the usual topological sense. Coarseness of their group actions is exactly the statement that the entourages $E\swapcross \Diag[L]$ are mapped to entourages, and continuity of their group action implies that we can pick an entourage $E_0$ which is simultaneously an open neighborhood of the diagonal and then choose a family $\cU$ as in our definition such that $E_{\cU}$ is mapped into $E_0$.
Hence their continuous and coarse group actions are examples of our topological group actions.

Of course, the converse is not true: Our topological group actions on topological coarse spaces are in general not continuous in the topological sense, because the definition makes no use of~the topology on the space.
The word ``topological'' just refers to the fact that the topology of the group has been taken into account. Also, we refrain from calling our group actions ``coarse'', because this property will always be assumed implicitly.

As a more concrete example, if a locally compact group $\Gamma$ acts continuously by quasi-isometries on a proper metric space $(X,d)$ such that the quasi-isometry constants are bounded on each compact subset $L\subset \Gamma$, then we can simply take $U_x\coloneqq \{(\gamma,\gamma')\in \Gamma\times \Gamma\mid d(x\gamma,x\gamma')<1\}$. An~even more particular example is the action of $\R^n\rtimes\mathrm{GL}(\R,n)$ on $\R^n$.

Note that in the definition we have only demanded that the group action is a controlled map and not a coarse map in general. Using the obvious fact that multiplication with the inverse $X\times \Gamma\to X$, $(x,\gamma)\mapsto x\gamma^{-1}$ is also a controlled map, it is straightforward to show that the restrictions of the multiplication maps to the subsets $X\times L$, $L\subset \Gamma$ compact, are coarse maps.

\begin{lem}
The topological actions of a \emph{discrete} group on a coarse space are exactly the group actions by coarse equivalences.
\end{lem}

\begin{proof}
One implication of the lemma is exactly the preceding statement applied to singletons $L=\{\gamma\}\subset \Gamma$.

The other direction follows with the coarse structure on $X\times \Gamma$ induced by the collection of sets $U_x\coloneqq\Delta_\Gamma$, which are neighborhoods of the diagonal due to the discreteness of $\Gamma$. With this choice, the coarse space $X\times \Gamma$ decomposes as a \emph{coarse disjoint union} of the spaces $X\times \{\gamma\}$, that is, these subsets carry the subspace coarse structure and they are closed under taking $E$-penumbras for arbitrary entourages $E$. The claim follows readily.
\end{proof}

The relation of topological group actions with the notion of generalized homotopy is the following:
Having seen that multiplication with each group element is always a coarse equivalence, it is immediate that for any two group elements in the same path component of $\Gamma$ their associated multiplication maps are generalized coarsely homotopic.

After this excursion to locally compact topological groups we specialize to the type of group actions which will appear in our \mbox{(co-)}homology theories: the proper and isocoarse actions of discrete groups by coarse equivalences.

\begin{defn}\label{def:discretegroupactingoncoarsespace}
A discrete group $\Gamma$ is said to \emph{act} on a coarse space $(X,\coarseStr)$ if its acts on the underlying set $X$ such that the action of each $\gamma\in\Gamma$ is a coarse equivalence with respect to the coarse structure $\coarseStr$. In this case we call $X$ a \emph{coarse $\Gamma$-space}, and if in addition $A\subset X$ is $\Gamma$-invariant, then we call $(X,A)$ a \emph{pair of coarse $\Gamma$-spaces}.
Furthermore, the $\Gamma$-action, the coarse $\Gamma$-space or the pair of coarse $\Gamma$-spaces is called
\begin{enumerate}\itemsep=0pt
\item[$\bullet$] \emph{proper} if the set $\{\gamma\in \Gamma\mid K\gamma \cap K\not=\varnothing\}$ is finite for all bounded $K\subset X$,
\item[$\bullet$] \emph{isocoarse} if every entourage is contained in a $\Gamma$-invariant entourage. In this case we will call $X$ concisely an \emph{isocoarse $\Gamma$-space} or $(X,A)$ a \emph{pair of isocoarse $\Gamma$-spaces}.
\end{enumerate}
\end{defn}

If $\Gamma$ itself carries a coarse structure in which bounded sets are finite, e.g., if $\Gamma$ is finitely generated and we equip it with the coarse structure coming from a word metric, then properness of the $\Gamma$-action is equivalent to properness of the map $X\times \Gamma\to X\times X$, $(x,\gamma)\mapsto (x\gamma,x)$.

Isocoarse actions are the coarse geometric analog of isometric actions and are sometimes even called isometric (cf.\ \cite[Definition 1]{EmeMeyDescent}). If $\Gamma$ acts isometrically on a metric space $X$ then every entourage $E$ of $X$ is contained in the $\Gamma$-invariant entourage $E_R$ for some sufficiently large $R\geq 0$
and hence the action is isocoarse. Conversely, if $\Gamma$ acts isocoarsely on a metric space $(X,d)$, then
\[
d'(x,y)\coloneqq\sup_{\gamma\in\Gamma}d(x\gamma,y\gamma)
\]
defines a $\Gamma$-invariant metric on $X$ which is coarsely equivalent to $d$.

The equivariance is also easily implemented into our two notions of coarse homotopy as follows and one readily checks that one immediately obtains corresponding (\emph{generalized$)$ homotopy categories of $($pairs of$)$ $($iso-$)$coarse $\Gamma$-spaces}.

\begin{defn}\label{def:equivcoarsehomotopy}
A \emph{$\Gamma$-equivariant coarse homotopy} between two $\Gamma$-equi\-variant coarse maps $(X,A)\to (Y,B)$ is simply a coarse homotopy between them which is equivariant as a map $X\times\Z\to Y$.
Similarily, a \emph{generalized $\Gamma$-equivariant coarse homotopy} between two $\Gamma$-equi\-variant coarse maps $(X,A)\to (Y,B)$ is a generalized coarse homotopy which is equivariant as a map $X\times I\to Y$ and where the coarse structure $\coarseStr_\cU$ on $X\times I$ is constructed from a $\Gamma$-equivariant family $\cU=\{U_x\}_{x\in X}$, that is, $U_{x\gamma}=U_x$ for all $x\in X$, $\gamma\in \Gamma$.
\end{defn}

Note that this definition is already adapted to isocoarse actions: if $X$ is an isocoarse $\Gamma$-space, then the last condition, $U_{x\gamma}=U_x$, implies that $(X\times I,\coarseStr_\cU)$ is an isocoarse $\Gamma$-space, too. If one is really interested in general non-isocoarse $\Gamma$-actions, then it might seem appropriate to weaken the condition to: for all $\gamma\in \Gamma$ there is $n\in\N$ such that for all $x\in X$ we have $U_{x\gamma}\subset\underbrace{U_x\circ\dots\circ U_x}_{n\text{ times}}$. In this case, $(X\times I,\coarseStr_\cU)$ is a coarse $\Gamma$-space for all coarse $\Gamma$-spaces $X$.
However, even in the few cases, where we do consider general non-isocoarse coarse $\Gamma$-spaces we will need the stronger condition (Lemmas~\ref{lem:ordinarycoarsehomologystronghomotopy} and~\ref{lem:ordinarycoarsecohomologystronghomotopy}), so we decided to use it in all cases.

For the classical notion of coarse homotopy, there is no such distinction to be made. The coarse structure on $X\times\Z$ is independent of choices and isocoarse if and only if $X$ is isocoarse.
The same applies to the following equivariant notion of flasqueness.
\begin{defn}\label{def:equiflasque}
A coarse $\Gamma$-space or a pair of coarse $\Gamma$-spaces is called flasque if it is flasque as a coarse space with $\phi$ being $\Gamma$-equivariant, cf.\ Definition~\ref{def:flasqueness}.
\end{defn}

Finally we also have to discuss the compatibility of group actions with discretizations. This will be important for the construction of $\Gamma$-equivariant coarse \mbox{(co-)}homology theories in Section~\ref{sec:Sigmastuff}.
\begin{lem}
Every proper isocoarse $\Gamma$-space $X$ of bornologically bounded geometry\footnote{Recall that this means that $X$ posseses a discretization, see Definition~\ref{def:discreteboundedgeometry}.} has a~$\Gamma$-invariant discretization.
\end{lem}

\begin{proof}
Let $X'\subset X$ be a locally finite $E$-dense subset for some symmetric entourage $E$ containing the diagonal, which exists because of the coarsely bounded geometry. Due to the isocoarseness we may also assume that $E$ is $\Gamma$-invariant. Now, Zorn's lemma implies that there is a maximal $\Gamma$-invariant subset $X''\subset X$ with the property that for all $(x_1,x_2)\in E^2\cap(X''\times X'')$ we have $x_1\Gamma=x_2\Gamma$. This subset is $E^2$-dense.

We claim that it is also locally finite. Let $K\subset X''$ be bounded. Then there is a map $f\colon K\to \Pen_E(K)\cap X'$ which maps each $x\in K$ to a point $f(x)\in X'$ with $(x,f(x))\in E$. If~$f(x_1)=f(x_2)$, then $(x_1,x_2)=(x_1,f(x_1))\circ(f(x_2),x_2)\in E^2\cap (X''\times X'')\implies x_1\Gamma=x_2\Gamma\implies \exists \gamma\in \Gamma\colon x_2=x_1\gamma$. Because of the properness of the $\Gamma$-action and discreteness of $\Gamma$, there are for each $x_1\in K$ only finitely many $\gamma\in \Gamma$ with $x_1\gamma\in K$. Thus, preimages of points under $f$ are all finite. The target of $f$ is a bounded subset of $X'$, so it is finite. Hence, the domain $K$ is also finite.
\end{proof}

\begin{Example}
Consider the proper but non-isocoarse action of $\Z$ on $\R^2$ given by $n.(x,y)=\linebreak(x+n, 2^ny)$. Although $\R^2$ has coarsely bounded geometry, we cannot find a $\Z$-invariant discretization: Assume that for some $R>0$ there is a $\Z$-invariant $R$-dense subset $X'$. Then for every $n\in\N$ the intersection of $X'$ with $[n,2R+n]\times\big[0,2^{n+1}R\big]$ must contain at least $2^n$ elements (at least one in each square $[n,2R+n]\times[2Rk,2R(k+1)]$ for $k=0,\dots,2^n-1$). Therefore, $[0,2R]^2=(-n).\big([n,2R+n]\times\big[0,2^{n+1}R\big]\big)$ intersects $X'$ also in at least $2^n$ points. As $n\in\N$ was arbitrary, the intersection of $X'$ with $[0,2R]^2$ cannot be finite.

This shows that we cannot dispense with the isocoarseness assumption.
\end{Example}

\begin{lem}\label{lem:goodequivariantdiscretization}
Every proper isocoarse $\Gamma$-space $X$ of bornologically bounded geometry has a~$\Gamma$-invariant discretization $X'$ which allows a $\Gamma$-equivariant coarse equivalence $\pi\colon X\to X'$ which is the identity on $X'$ and hence is a coarse inverse up to closeness to the inclusion $\iota\colon X'\subset X$.
\end{lem}

\begin{proof}
Let $X''\subset X$ be a $\Gamma$-invariant discretization, which exists by the previous lemma.
The problem is that points of $X$ might have stabilizers which don't appear in $X''$ and hence we cannot construct a $\Gamma$-equivariant coarse map $X\to X''$. Instead, $X''$ has to be enlarged first.

Assume that $E$ is a symmetric $\Gamma$-invariant entourage containing the diagonal such that~$X''$ is $E$-dense in~$X$. Due to the properness of the group action, for each $x\in X$ the set $\{\gamma\in\Gamma\mid \Pen_E(x)\gamma\cap\Pen_E(x)\not=\varnothing\}$ is finite. All stabilizers $\Gamma_{y}$ of points $y\in\Pen_E(x)$ are contained in this set and therefore the set $S(x)\coloneqq\{\Gamma_y\mid y\in\Pen_E(x)\}$ of all such stabilizers is finite.

For each $\Gamma$-orbit of $O\subset X''$ we choose one fixed representative $x_O$, i.e., $O=x_O\Gamma$, and afterwards we choose for each stabilizer $\Sigma\in S(x_O)$ a representative $y_\Sigma$, i.e., $\Sigma=\Gamma_{y_\Sigma}$. Let $X'$ be the union of $X''$ with all the orbits $y_\Sigma\Gamma$ for all orbits $O$ in $X''$ and all $\Sigma\in S(x_O)$. The choices have been made in such a way that $X'$ is again locally finite and therefore it is also a $\Gamma$-invariant discretization.

Note that for each $x\in X$ there is now a point $x'\in X'\cap \Pen_{E^2}(x)$ with the same stabilizer. This allows us to choose a $\Gamma$-equivariant coarse map $X\to X'$ which is the identity on $X'$ and such that the composition $X\to X'\subset X$ is $E^2$-close to the identity.
\end{proof}

\section{Coarse (co-)homology theories}
\label{sec:CoarseCoHomologyTheories}

Not all $\Gamma$-equivariant coarse \mbox{(co-)}homology theories will be defined on the whole category of pairs of coarse $\Gamma$-spaces, but only on certain sub-categories.
\begin{defn}\label{def:admissibleCAT}
An \emph{admissible $\Gamma$-coarse category} is a full sub-category $\admissible$ of the category of pairs of coarse $\Gamma$-spaces which contains $\varnothing=(\varnothing,\varnothing)$ and satisfies the following condition. If $(X,A)$ is an object in $\admissible$, then so are
\[
(X,X),\qquad (A,A),\qquad X=(X,\varnothing),\qquad A=(A,\varnothing).
\]
\end{defn}
Examples of admissible $\Gamma$-coarse categories are those consisting of all pairs of coarse $\Gamma$-spaces satisfying any combination of the properties isocoarseness (Definition~\ref{def:discretegroupactingoncoarsespace}), coarse connectedness, countably generatedness (Definition~\ref{def:coarsestructure}), (uniform) discreteness, bornologically/coarsely bounded geometry (Definition~\ref{def:discreteboundedgeometry}). Also, the categories of topological coarse spaces and proper metric spaces map forgetfully onto admissible categories when forgetting the topology or the metric, respectively.

This definition is inspired by Eilenberg--Steenrod's notion of admissibility in \cite[p.~5]{EilenbergSteenrod}.
The biggest difference is that we only consider full sub-categories and therefore do not have to put additional conditions on the morphisms.
Note that this also effects the notion of homotopy: Eilenberg and Steenrod only consider homotopies that are morphisms in the admissible category. In particular the domains must be objects of the admissible category. In contrast to their intrinsic notion of homotopy we do not postulate that the domains $X\times \Z$, $X_{\rho^-}^{\rho^+}$ or $(X\times I,\coarseStr_\cU)$ of our homotopies are objects in the admissible category (although they usually are) and hence we get an extrinsic notion of homotopy.

There are some examples of sub-categories with smaller morphism sets on which \mbox{(co-)}ho\-mo\-logy theories can also be defined meaningfully. For example, one could consider the category of topological coarse spaces together with all continuous coarse maps or the category of all coarse spaces together with all rough maps, that is, coarse maps $f\colon X\to Y$ for which preimages of entourages of $Y$ under $f\times f$ are entourages of $X$.
However, it is the author's opinion that the corresponding \mbox{(co-)}homology theories should then be called ``topological coarse \mbox{(co-)}homology'' or ``rough \mbox{(co-)}homology'', respectively, but not ``coarse \mbox{(co-)}homology''.

Our admissible coarse categories together with coarse homotopy or generalized coarse homotopy and the following notion of excisions fit into the framework of Eilenberg--Steenrod's $h$-categories \cite[Definition~9.1 in Section~IV.9]{EilenbergSteenrod} except that we did not single out certain spaces as ``points''.
\begin{defn}\label{defn:excisionmorphisms}
We call an inclusion map between objects of $\admissible$ of the form $(X\setminus C,A\setminus C)\to (X,A)$ with $C\subset A\subset X$ an \emph{excision}, if for every entourage $E$ of $X$ there is an entourage $F$ of $X$ such that
$\Pen_E(A)\setminus C\subset\Pen_F(A\setminus C)$.
\end{defn}
This is our coarse geometric analogue of the topological excision condition that the closure of $C$ lies within the interior of $A$. It says that $C$ sits sufficiently far inside of $A$ such that the difference $A\setminus C$ does not change from a coarse geometric perspective if we enlarge $A$ to an arbitrarily large penumbra.

If we write $B\coloneqq X\setminus C$, then the condition is equivalent to the coarse excisiveness condition for the covering $X=A\cup B$ known from \cite{HigsonRoeYuMayerVietoris}:
\begin{equation}\label{eq:excisiveness}
\forall E\in\coarseStr_X\exists F\in\coarseStr_X\colon\ \Pen_E(A)\cap \Pen_E(B)\subset \Pen_F(A\cap B).
\end{equation}
One direction is trivial. For the other we observe that
\begin{align*}
\Pen_E(A)\cap\Pen_E(B)&\subset\Pen_E(\Pen_{E^{-1}\circ E}(A)\cap B)
=\Pen_E(\Pen_{E^{-1}\circ E}(A)\setminus C)
\\
&\subset\Pen_E(\Pen_F(A\setminus C))=\Pen_{E\circ F}(A\cap B)
\end{align*}
for a sufficiently large entourage $F$.

If $A,B\subset X$ are arbitrary subspaces which satisfy \eqref{eq:excisiveness}, then we call $(X;A,B)$ an excisive triad in $\admissible$, even if $A$ and $B$ do not cover $X$.

We will now define coarse \mbox{(co-)}homology theories as \mbox{(co-)}homologies on these $h$-categories, but without the dimension and additivity axiom, because they would only be unnecessary constraints.

\begin{defn}\label{def:CoHomologytheories}
Let $\admissible$ be an admissible category of pairs of coarse $\Gamma$-spaces.
\begin{enumerate}\itemsep=0pt
\item[$\bullet$] A \emph{$\Gamma$-equivariant coarse homology theory} on $\admissible$ is a collection of covariant functors $\big\{\HomolX^\Gamma_p\big\}_{p\in\Z}$ from $\admissible$ to the category of Abelian groups together with a collection $\big\{\partial^\Gamma_p\big\}_{p\in\Z}$ of natural transformations, the \emph{connecting homomorphisms} $\partial^\Gamma_p\colon \HomolX^\Gamma_p(X)\to \HomolX^\Gamma_{p-1}(A)$ for all objects $(X,A)$ in $\admissible$,
 satisfying the below-mentioned homotopy, exactness and excision axioms.

\item[$\bullet$] A \emph{$\Gamma$-equivariant coarse cohomology theory} on $\admissible$ is a collection of contravariant functors $\big\{\CohomX_\Gamma^p\big\}_{p\in\Z}$ from $\admissible$ to the category of Abelian groups together with a collection $\big\{\delta_\Gamma^p\big\}_{p\in\Z}$ of natural transformations, the \emph{connecting homomorphisms} $\delta_\Gamma^p\colon \CohomX_\Gamma^p(A)\to \CohomX_\Gamma^{p+1}(X)$ for all objects $(X,A)$ in $\admissible$, satisfying the below-mentioned homotopy, exactness and excision axioms.
\end{enumerate}

\noindent
{\it Homotopy.} If two pairs of $\Gamma$-equivariant coarse maps are $\Gamma$-equivariantly coarsely homotopic, then they induce the same maps on \mbox{(co-)}homology. That is, \mbox{(co-)}homology factors through the homotopy category of $\admissible$.

\medskip\noindent
{\it Exactness.} Each pair of coarse $\Gamma$-spaces $(X,A)$ induces a long exact sequence
\begin{gather*}
\dots\to \HomolX^\Gamma_p(A)\xrightarrow{i_*}\HomolX^\Gamma_p(X)\xrightarrow{j_*}\HomolX^\Gamma_p(X,A) \xrightarrow{\partial_p}\HomolX^\Gamma_{p-1}(A)\to \cdots
\end{gather*}
or
\begin{gather*}
\dots\to \CohomX_\Gamma^{p-1}(A)\xrightarrow{\delta^{p-1}}\CohomX_\Gamma^p(X,A)\xrightarrow{j^*}\CohomX_\Gamma^p(X) \xrightarrow{i^*}\CohomX_\Gamma^{p}(A)\to \cdots,
\end{gather*}
respectively,
where $(A,\varnothing)\xrightarrow{i}(X,\varnothing)\xrightarrow{j}(X,A)$ are the inclusion maps.

\medskip\noindent
{\it Excision.} Excisions $(X\setminus C,A\setminus C)\subset (X,A)$ induce isomorphisms
\(\HomolX^\Gamma_*(X\setminus C,A\setminus C)\cong\HomolX^\Gamma_*(X,A)\) or \(\CohomX_\Gamma^*(X,A)\cong\CohomX_\Gamma^*(X\setminus C,A\setminus C)\), respectively.

If $\Gamma=1$ is the trivial group, then we call $\HomolX_*\coloneqq\HomolX^\Gamma_*$ simply a \emph{coarse homology theory} or $\CohomX^*\coloneqq\CohomX_\Gamma^*$ a \emph{coarse cohomology theory}, respectively.
\end{defn}

Standard constructions from algebraic topology (cf.\ \cite[Chapter~4, Section~8, Theorem 5]{Spanier66} and \cite[Section 2.3]{Hatcher_AT})
also show that there are a long exact sequence for triples and a coarse Mayer--Vietoris sequence.

\begin{prop}
Let $\HomolX^\Gamma_*$ be a $\Gamma$-equivariant coarse homology theory or $\CohomX_\Gamma^*$ a $\Gamma$-equivariant coarse co-homology theory.
\begin{enumerate}\itemsep=0pt
\item[$\bullet$] If $X\supset A\supset B$ are coarse $\Gamma$-space such that $(X,A)$, $(X,B)$ and $(A,B)$ are objects in $\admissible$, then there are long exact sequences
\begin{gather*}
\dots\to \HomolX^\Gamma_p(A,B)\xrightarrow{i_*}\HomolX^\Gamma_p(X,B)\xrightarrow{j_*}\HomolX^\Gamma_p(X,A)\to
\HomolX^\Gamma_{p-1}(A,B)\to \cdots
\end{gather*}
or
\begin{gather*}
\dots\to \CohomX_\Gamma^{p-1}(A,B)\to
\CohomX_\Gamma^p(X,A)\xrightarrow{j^*}\CohomX_\Gamma^p(X,B)\xrightarrow{i^*}\CohomX_\Gamma^{p}(A,B)\to \cdots,
\end{gather*}
respectively,
where $(A,B)\xrightarrow{i}(X,B)\xrightarrow{j}(X,A)$ are the inclusion maps and the connecting homomorphisms are the compositions of the connecting homomorphisms of the long exact sequence associated to the pair $(X,A)$ with the homomorphisms induced by the inclusion $(A,\varnothing)\to (A,B)$.
\item[$\bullet$] Let $X$ be a coarse $\Gamma$-space with $X=A\cup B$ satisfying \eqref{eq:excisiveness} and such that $(X,A)$ and $(B,A\cap B)$ are objects in $\admissible$ or $(X,B)$ and $(A,A\cap B)$ are objects in $\admissible$. Then there are the coarse Mayer--Vietoris sequences
\begin{align*}
\cdots&\to \HomolX^\Gamma_p(A\cap B)\xrightarrow{(i_*,j_*)}\HomolX^\Gamma_p(A)\oplus\HomolX^\Gamma_p(B)\xrightarrow{k_*-l_*} \HomolX^\Gamma_p(X)
\\
&\xrightarrow{\partial^{\mathrm{MV}}
}\HomolX^\Gamma_{p-1}(A\cap B)\to \cdots
\end{align*}
or
\begin{align*}
\cdots&\to \CohomX_\Gamma^{p-1}(A\cap B)\xrightarrow{\delta_{\mathrm{MV}}
}\CohomX_\Gamma^p(X)\xrightarrow{(k^*,l^*)}\CohomX_\Gamma^p(A)\oplus \CohomX_\Gamma^p(B)
\\
&\xrightarrow{i^*-j^*}\CohomX_\Gamma^{p}(A\cap B)\to \cdots,
\end{align*}
respectively,
where $i\colon A\cap B\to A$, $j\colon A\cap B\to B$, $k\colon A\to X$ and $l\colon B\to X$ denote the inclusion maps. The connecting homomorphisms $\partial^{\mathrm{MV}}$, $\delta_{\mathrm{MV}}$ are defined by composing the connecting homomorphisms for the pair $(B,A\cap B)$ with inverse of the isomorphism induced by the excision $(B,A\cap B)\subset (X,A)$ and the homomorphism induced by inclusion $X\mapsto (X,A)$, or analogously with the roles of $A$ and $B$ interchanged, depending on which of the two possible hypotheses is true.
\end{enumerate}
\end{prop}

The axioms of Definition~\ref{def:CoHomologytheories} are the absolute minimum needed for a meaningful notion of \mbox{(co-)}homology, but there are also some additional axioms which one might want to pose ocasionally.

\medskip\noindent
{\it Strong homotopy.} If two pairs of $\Gamma$-equivariant coarse maps are generalized $\Gamma$-equivariantly coarsely homotopic, then they induce the same maps on \mbox{(co-)}ho\-mo\-lo\-gy. That is, the \mbox{(co-)}ho\-mo\-logy theory factors through the generalized homotopy category of $\admissible$.

\medskip\noindent
{\it Flasqueness.} The $\Gamma$-equivariant coarse \mbox{(co-)}homology theory vanishes on all flasque spaces in $\admissible$.

\medskip\noindent
{\it Coronality.}\footnote{This axiom is called the large scale axiom in \cite[Definition 3.1]{MitchenerCoarse}.} The $\Gamma$-equivariant coarse \mbox{(co-)}homology theory vanishes on all relatively compact spaces in $\admissible$.

Mitchener has shown, that his coarse homology theories are completely determined on \emph{coarse CW-complexes} by the values on the ray $\N$ and the one-point space, and the same is true of course for our non-equivariant homology theories, as long as the admissible category $\admissible$ is large enough to accomodate all the constructions needed. Hence, it is not a good idea to pose both the flasqueness and the coronality axiom at the same time.

However, all meaningful examples satisfy exactly one of them. Which one we take depends on how the particular \mbox{(co-)}homology theory should be interpreted:
Some coarse \mbox{(co-)}homology theories arise as coarsifications of locally finite homology or compactly supported cohomology theories (cf.\ Section~\ref{sec:Ripscomplex}). Those topological \mbox{(co-)}homology theories often vanish on spaces of the form $X\times[0,\infty)$ and as a result the flasqueness axiom will be satisfied (see Section~\ref{sec:SteenrodTransgression}, in particular Lemma~\ref{lem:SteenrodFlasqueness}).

Other coarse \mbox{(co-)}homology theories are better interpreted as \mbox{(co-)}ho\-mo\-lo\-gy of some (hypothetical) corona. If these hypothetical \mbox{(co-)}homology groups are unreduced ones, then it is the coronality axiom which is satisfied, because bounded spaces have empty coronas. And if these hypothetical \mbox{(co-)}homology groups are reduced ones, then it is the flasqueness axiom again.

\subsection{Flasqueness implies homotopy invariance}
\label{sec:flasquenessimplieshomotopyinvariance}

It is well-known that vanishing on flasque spaces can be used to show homotopy invariance of a functor (see \cite[Theorem~9.8]{RoeITCGTM}, \cite[Theorem~11.2]{HigsonPedersenRoe}, \cite[Proposition~3.10]{WillettHomological} and \cite[Proposition~4.16]{BunkeEngel_homotopy}; similar is \cite[Proposition~12.4.12]{HigRoe}) and the upcoming Lemma~\ref{lem:flasquenessimplieshomotopyinvariance} is an instance of this principle.
This fact is important, because in some major examples (cf.\ Sections~\ref{sec:sHigCor} and~\ref{sec:RoeAlg}) vanishing on flasque spaces can be shown rather directly by performing certain Eilenberg swindles.
Note that the three notions of coarse homotopy, flasqueness and Eilenberg swindle are closely related in the sense they all feature the natural numbers (or integers) in an essential way. Therefore, the same arguments cannot show strong homotopy invariance, where the homotopy parameter runs through a closed interval.

But first we state an auxiliary lemma. Given any coarse $\Gamma$-space $X$ and a $\Gamma$-equivariant bornological function $\rho\colon X\to\N$ we define the coarse $\Gamma$-spaces
\begin{gather*}
X_{\rho}\coloneqq\{(x,n)\in X\times \Z\mid n\geq\rho(x)\}, \qquad X_{0}\coloneqq X\times \N,
\\
X^{\rho}\coloneqq\{(x,n)\in X\times \Z\mid n\leq\rho(x)\}, \qquad X^{0}\coloneqq X\times -\N,
\\
X_\rho^\rho\coloneqq\{(x,n)\in X\times \Z\mid n=\rho(x)\}, \qquad X_0^{0}\coloneqq X\times \{0\},
\\
X_0^{\rho}\coloneqq\{(x,n)\in X\times \Z\mid 0\leq n\leq\rho(x)\}
\end{gather*}
equipped with the subspace coarse structure of $X\times \Z$.

\begin{lem}\label{lem:homotopydomainisomorphismlemma}
Assume that the admissible category $\admissible$ has the following property: Whenever $(X,A)$ is an object of $\admissible$ and $\rho\colon X\to\N\setminus\{0\}$ is a $\Gamma$-equivariant bornological function, then $\admissible$ also contains objects $\big(\overline X_0^\rho,\overline A_0^\rho\big)$ and $\big(\overline X_\rho^\rho,\overline A_\rho^\rho\big)$
satisfying the following properties:
\begin{enumerate}\itemsep=0pt
\item[$\bullet$] $\overline A_0^\rho\subset \overline X_0^\rho\subset X\times\N$ and $\overline A_\rho^\rho\subset \overline X_\rho^\rho\subset X\times(\N\setminus\{0\})$ carrying the subspace coarse structure of the product coarse structure,
\item[$\bullet$] $\big(\overline X_\rho^\rho,\overline A_\rho^\rho\big)\subset\big(\overline X_0^\rho,\overline A_0^\rho\big)$,
\item[$\bullet$] $\big(\overline X_0^\rho,\overline A_0^\rho\big)$ and $\big(\overline X_\rho^\rho,\overline A_\rho^\rho\big)$ contain $\big(X_0^\rho,A_0^\rho\big)$ and $\big(X_\rho^\rho,A_\rho^\rho\big)$, respectively, as coarsely dense subspaces.
\end{enumerate}
Then any co- or contravariant functor from $\admissible$ to the category of abelian groups that turns the inclusions $i_0\colon (X,A)\cong \big(X_0^0,A_0^0\big)\to \big(\overline X_0^\rho,\overline A_0^\rho\big)$ and $i_\rho\colon \big(\overline X_\rho^\rho,\overline A_\rho^\rho\big)\to \big(\overline X_0^\rho,\overline A_0^\rho\big)$ as well as the projections $p\colon \big(\overline X_0^\rho,\overline A_0^\rho\big)\to (X,A)$ into isomorphisms is homotopy invariant.
\end{lem}
All of the admissible categories mentioned earlier have this property, where in most cases we can take $\big(\overline X_\rho^\rho,\overline A_\rho^\rho\big)=\big(X_\rho^\rho,A_\rho^\rho\big)$ and $\big(\overline X_0^\rho,\overline A_0^\rho\big)=\big(X_0^\rho,A_0^\rho\big)$, except for the categories obtained from proper metric spaces or topological coarse spaces, in which case we have to take the closures of~$X_\rho^\rho$, $A_\rho^\rho$, $X_0^\rho$ and $A_0^\rho$ in $X\times\N$. This explains the notation.
\begin{proof}
Let $(X,A)$ and $(Y,B)$ be two objects in $\admissible$ and let $H\colon (X,A)\times \Z\to(Y,B)$ be a $\Gamma$-equivariant coarse homotopy between $f$ and $g$ with associated $\Gamma$-equivariant bornological maps $\rho^\pm\colon X\to\Z$. We may assume $\rho^-\equiv 0$ and $\rho\coloneqq\rho^+\colon X\to\N$ without loss of generality. Otherwise we simply replace $\rho^-$ by $\min\{\rho^-,-1\}$ and $\rho^+$ by $\max\{\rho^+,1\}$ and cut the homotopy along $X\times\{0\}$ into two homotopies of the above type.

Using the assumptions on $\big(\overline X_0^\rho,\overline A_0^\rho\big)$ and $\big(\overline X_\rho^\rho,\overline A_\rho^\rho\big)$, we can find a coarse map $\overline{H}\colon \big(\overline X_0^\rho,\overline A_0^\rho\big)\to (Y,B)$ which agrees with $H$ on $\big(X_0^\rho\setminus \overline X_\rho^\rho,A_0^\rho\setminus \overline A_\rho^\rho\big)$, in particular on $\big(X_0^0,A_0^0\big)$, and with $g\circ p$ on $\big(\overline X_\rho^\rho,\overline A_\rho^\rho\big)$.
The claim now follows by applying the functor to $f\circ p\circ i_0=\overline H\circ i_0$ and $g\circ p\circ i_\rho=\overline H\circ i_\rho$ and exploiting the isomorphism hypotheses.
\end{proof}

\begin{lem} \label{lem:flasquenessimplieshomotopyinvariance}
Assume that the admissible category $\admissible$ has the following property: Associated to each $X=(X,\varnothing)$ in $\admissible$ and each $\Gamma$-equivariant bornological function $\rho\colon X\to\N\setminus\{0\}$ there are coarse subspaces $\overline X_\rho$, $\overline X^\rho$ of $X\times \Z$ with the following properties:
\begin{enumerate}\itemsep=0pt
\item[$\bullet$] $\overline X_\rho$ and $\overline X^\rho$ contain $X_\rho$ and $X^\rho$, respectively, as coarsely dense subspaces,
\item[$\bullet$] the sets $\overline X_\rho^\rho\coloneqq \overline X_\rho\cap \overline X^\rho$ and $\overline X_0^\rho\coloneqq X_0\cap \overline X^\rho$ together with the corresponding sets for $A\subset X$ satisfy the assumptions of Lemma~$\ref{lem:homotopydomainisomorphismlemma}$,
\item[$\bullet$] all pairs of coarse $\Gamma$-spaces that can be formed from $X\times\Z$, $X_0$, $X^0$, $X_0^0$, $\overline X_\rho$, $\overline X^\rho$, $\overline X_0^\rho$ and~$\overline X_\rho^\rho$ are objects in $\admissible$.
\end{enumerate}
If $\HomolX^\Gamma_*$ or $\CohomX_\Gamma^*$ satisfies the exactness, excision and flasqueness axioms, then it also satisfies the homotopy axiom.
\end{lem}

Again, the hypothesis is satisfied by all admissible the categories mentioned earlier.
The proof is essentially a $\Gamma$-equivariant version of the proofs of the statements cited at the beginning of this subsection.

\begin{proof}
We claim that the functor $\HomolX^\Gamma_*$ or $\CohomX_\Gamma^*$ satisfies the hypotheses of the preceding lemma.
Using the naturality of the long exact sequences under the maps $i_0$, $i_\rho$ and $p$ and the five-lemma, we can furthermore restrict ourselves to the absolute case $A=\varnothing$.

Now, as $\HomolX^\Gamma_*$ or $\CohomX_\Gamma^*$ satisfies the exactness and excision axioms, it also has Mayer--Vietoris sequences which can be applied to the decompositions $X\times\Z=X_0\cup X^0$, $X\times\Z=X_0\cup \overline X^\rho$ and $X\times\Z=\overline X_\rho\cup \overline X^\rho$. It furthermore vanishes on the flasque spaces $X_0$, $X^0$, $\overline X_\rho$, $\overline X^\rho$ and therefore the Mayer--Vietoris sequences simplify to isomorphisms
\begin{gather*}
\HomolX^\Gamma_p\big(X_0^0\big)=\HomolX^\Gamma_p\big(X_0\cap X^0\big)\cong \HomolX^\Gamma_{p+1}(X\times\Z),
\\
\HomolX^\Gamma_p\big(\overline X_0^\rho\big)=\HomolX^\Gamma_p\big(X_0\cap \overline X^\rho\big)\cong \HomolX^\Gamma_{p+1}(X\times\Z),
\\
\HomolX^\Gamma_p\big(\overline X_\rho^\rho\big)=\HomolX^\Gamma_p\big(\overline X_\rho\cap \overline X^\rho\big)\cong \HomolX^\Gamma_{p+1}(X\times\Z)
\end{gather*}
and analogously for $\CohomX_\Gamma^*$.
These isomorphisms are natural and therefore the inclusion maps~$i_0$: $X_0^0\to \overline X_0^\rho$ and $i_\rho\colon \overline X_\rho^\rho\to \overline X_0^\rho$ induce isomorphisms. Up to the canonical isomorphism $X\to X_0^0$ in the category of coarse spaces, the map $p$ is a one sided inverse to $i_0$. Thus, it induces a one sided inverse to the isomorphism induced by $i_0$ and must itself be an isomorphism.
\end{proof}

\section{Examples}\label{sec:Examples}
In Section~\ref{sec:Sigmastuff} we will describe a procedure to construct a wide class of coarse \mbox{(co-)}homology theories out of generalized \mbox{(co-)}homology theories for $\sigma$-locally compact spaces.
Therefore we shall limit ourselves in this section to a few examples which can be defined by different means.

The first two, ordinary coarse homology (Section~\ref{sec:ordinarycoarsehomology}) and cohomology (Section~\ref{sec:ordinarycoarsecohomology}), are defined by constructing rather elementary \mbox{(co-)}chain complexes. Note that we have implemented the $\Gamma$-action in a very naive way by considering $\Gamma$-invariant chains or dividing out the $\Gamma$-action on the cochain complexes, respectively. This is sufficient for the purpose of giving examples, but in real life it would be better to develop a more sophisticated $\Gamma$-equivariant generalization in the cohomological case.

For example, one well-known way of defining equivariant cohomology of a topological space~$X$ with action of a topological group $\Gamma$ is to take a non-equivariant cohomology theory $\Cohom^*$ and evaluate it on the homotopy quotient:
\[
\Cohom_\Gamma^*(X)\coloneqq\Cohom^*(X\times_\Gamma\mathrm{E}\Gamma).
\]

\begin{question}
Does an analogous construction also work for coarse geometry? And if yes, what is the coarse geometric analogue of the classifying space $\mathrm{E}\Gamma$?
\end{question}
An answer to this question would give rise to a wide variety of equivariant cohomology theories.

In Sections~\ref{sec:sHigCor} and~\ref{sec:RoeAlg} we show how the $\K$-theories of stable Higson coronas and Roe algebras fit into our set-up. Here, there are well-established ways of implementing the $\Gamma$-actions.

\subsection{Ordinary coarse homology}
\label{sec:ordinarycoarsehomology}

Among all coarse homology and cohomology theories, the ordinary coarse homology is the most elementary non-trivial one that one could possibly think of. In spirit, it is the obvious dual to Roe's coarse cohomology which we will review in the following subsection. In a special case, its definition has already been given in \cite[p.~453]{YuCyclicCohomology}, and the general version has been developed in \cite[Section~6.3]{BunkeEngel_homotopy}. The equivariant version can be found in \cite[Section~7]{BunkeEngelKasprowskiWinges}. There is actually no need at all to pass to a smaller admissible category than the one of coarse $\Gamma$-spaces.

In the following, given a coarse space $X$ we always consider $X^{p+1}$ as equipped with the $(p+1)$-fold product coarse structure. In particular, the penumbras of the multidiagonal \(\Diag[X,p+1]\coloneqq\big\{(x,\dots,x)\in X^{p+1} \mid x\in X\big\}\) are those subsets $P\subset X^{p+1}$ for which there is an entourage $E$ of~$X$ such that for all $(x_0,\dots,x_p)\in P$ and for all $0\leq i,j\leq p$ we have $(x_i,x_j)\in E$.

\begin{defn}
Let $X$ be a coarse space and $M$ an abelian group. For $p\in\N$ we define $\CX_p(X;M)$ as the group of all infinite formal sums $c=\sum_{\mathbf{x}\in X^{p+1}}m_{\mathbf{x}}\mathbf{x}$ such that
\begin{enumerate}\itemsep=0pt
\item[$\bullet$] the set $\supp(c)\coloneqq\big\{\mathbf{x}\in X^{p+1}\mid m_{\mathbf{x}}\not=0\big\}$ is a penumbra of the multidiagonal,
\item[$\bullet$] the set $\supp(c)\cap K$ is finite for all bounded $K\subset X^{p+1}$, or equivalently, for all relatively compact $K$.\footnote{Recall that relatively compact subsets are finite unions of bouded subsets by Definition~\ref{def:relativelycompact}.}
\end{enumerate}
For $p\in\Z\setminus\N$ we define $\CX_p(X;M)\coloneqq 0$. These groups together with the boundary maps
$\partial\colon\CX_{p}(X;M)\to\CX_{p-1}(X;M)$ which are defined on the single summands by
\[
\partial(x_0,\dots,x_p)\coloneqq \sum_{i=0}^p(-1)^i(x_0,\dots,x_{i-1},x_{i+1},\dots,x_p)
\]
constitute a chain complex.
If $A\subset X$ is a coarse subspace, then $\CX_*(A;M)$ is a subcomplex of $\CX_*(X;M)$ and we define the \emph{coarse chain complex} $\CX_*(X,A;M)\coloneqq \CX_*(X;M)/\CX_*(A;M)$. Its homology is the \emph{ordinary coarse homology} $\HX_*(X,A;M)$ of the pair of coarse spaces $(X,A)$ with coefficients in $M$.

If $\Gamma$ is a discrete group acting coarsely on $(X,A)$ from the right and acting on $M$ from the left, then it also acts on the chain complex $\CX_*(X,A;M)$ from the left by
\[
\gamma\bigg(\sum_{\mathbf{x}\in X^{p+1}}m_{\mathbf{x}}\mathbf{x}\bigg)\coloneqq\sum_{\mathbf{x}\in X^{p+1}}( \gamma m_{\mathbf{x}})\big(\mathbf{x}\gamma^{-1}\big)=\sum_{\mathbf{x}\in X^{p+1}}(\gamma m_{\mathbf{x}\gamma})\mathbf{x}.
\]
The chain $\sum_{\mathbf{x}\in X^{p+1}}m_{\mathbf{x}}\mathbf{x}$ is $\Gamma$-invariant, if $\gamma m_{\mathbf{x}\gamma}=m_{\mathbf{x}}$ for all $\gamma\in\Gamma$ and $x_0,\dots,x_p\in X$.
This property is preserved under the boundary map and hence we obtain the subcomplex $\CX^\Gamma_*(X,A;M)\subset \CX_*(X,A;M)$ of $\Gamma$-invariant chains. Its homology is the \emph{$\Gamma$-equivariant ordinary coarse homology} $\HX^\Gamma_*(X,A;M)$.
\end{defn}

The chain complex $\CX_*(X;\Z)$ is very similar to the one for uniformly finite homology of \cite[Section 2]{BlockWeinberger}, just that we have ommitted the uniform bounds on the coefficients and our bounds on the number of non-zero coefficients in bounded sets are also not uniform in any way.
Due to these uniformity requirements, uniformly finite homology is not functorial under arbitrary coarse maps but only under the rough maps.

In contrast, it is straightforward to see that ordinary $\Gamma$-equivariant coarse homology with arbitrary coefficients is functorial under $\Gamma$-equivariant coarse maps.
We will now show that it satisfies the exactness, strong homotopy, excision and flasqueness axioms and in particular is a~coarse homology theory in the sense of Definition~\ref{def:CoHomologytheories}.

\begin{lem}
The ordinary $\Gamma$-equivariant coarse homology satisfies the exactness axiom.
\end{lem}
\begin{proof}
The short exact sequence
\[
0\to \CX_*(A;M)\to\CX_*(X;M)\to \CX_*(X,A;M)\to 0
\]
has a $\Gamma$-equivariant split $\CX_*(X;M)\to \CX_*(A;M)$ (not compatible with the boundary map, of course) which maps all summands not supported within $A$ to $0$. This readily implies that the short sequence of the subcomplexes
\[
0\to \CX_*^\Gamma(A;M)\to\CX_*^\Gamma(X;M)\to \CX_*^\Gamma(X,A;M)\to 0
\]
is also exact.
\end{proof}

\begin{lem}\label{lem:ordinarycoarsehomologystronghomotopy}
The ordinary $\Gamma$-equivariant coarse homology satisfies the strong homotopy axiom. In particular, it is also functorial under \emph{closeness classes} of $\Gamma$-equivariant coarse maps.
\end{lem}

\begin{proof}
Let $H\colon (X\times [a,b],\coarseStr_\cU)\to (Y,\coarseStr_Y)$ be a $\Gamma$-equivariant generalized coarse homotopy between~$f$ and $g$ which maps $A\times[a,b]$ into $B$.
The outline of the proof is that we are going to construct a ``prism operator'' (inspired by the classical prism operator in algbraic topology, cf.~\cite[proof of Theorem 2.10]{Hatcher_AT}) as a chain homotopy between $f_*$ and $g_*$ which should be interpreted as follows: A~``simplex'' $\mathbf{x}=(x_0,\dots,x_p)$ times the intervall $I$ must be ``sudivided'' into simplices of diameter only depending on the diameter of $\mathbf{x}$, but as the diameter of $I$ is allowed to depend unboundedly on $\mathbf{x}$, the fineness of the subdivision has to be chosen depending on $\mathbf{x}$.
Concretely, this is done as follows.

For each point $x\in X$ we choose $a=s_{x,0}<s_{x,1}<\dots<s_{x,k_x}=b$ such that $(s,s_{x,j})\in U_x$ and $(s,s_{x,j+1})\in U_x$ for all $s\in[s_{x,j},s_{x,j+1}]$, $j=0,\dots,k_x-1$.
Due to $U_{x\gamma}=U_x$, this can be done in such a way that $s_{x,j}=s_{x\gamma,j}$ for all $\gamma\in \Gamma$, $x\in X$ and $j=0,\dots,k_x-1=k_{x\gamma}-1$. Note that although we do not require the action of $\Gamma$ on $X$ to be isocoarse, we cannot weaken the equivariance condition on the $U_x$ as described in the paragraph following Definition~\ref{def:equivcoarsehomotopy}.

The purpose of these choices is that $((x,s),(x,s_{x,j}))$ and $((x,s),(x,s_{x,j+1}))$ are contained in the entourage $E_\cU$ whenever $s\in [s_{x,j},s_{x,j+1}]$. Consequently, if $x_1,\dots,x_n$ lie in the same $E$-ball $\Pen_E(\{x\})$ around some $x\in X$ for some entourage $E\in\coarseStr_X$ and if $s\in\bigcap_{i=0}^p[s_{x_i,j_i},s_{x_i,j_i+1}]$, then the points
\begin{equation}\label{eq:2ppoints}
(x_0,s_{x_0,j_0}),\dots,(x_p,s_{x_p,j_p}),(x_0,s_{x_0,j_0+1}),\dots,(x_p,s_{x_p,j_p+1})
\end{equation}
all lie in the $(E\swapcross\Diag[{[a,b]}])\circ E_\cU$-ball $\Pen_{(E\swapcross\Diag[{[a,b]}])\circ E_\cU}(\{(x,s)\})$ around $(x,s)$.
In other words, if we start with $\mathbf{x}=(x_0,\dots,x_p)$ from a fixed penumbra of the multidiagonal in $X^{p+1}$, then for each $q\in\N$ any $q$-tuple of points with entries from \eqref{eq:2ppoints} lies within a fixed penumbra of the multidiagonal in $(X\times I)^q$.

For each $\mathbf{x}=(x_0,\dots,x_p)\in X^{p+1}$ let $k_\mathbf{x}\coloneqq k_{x_0}+\dots+k_{x_p}$ and we define
\[
\mathbf{j}_\mathbf{x}(l)=(j_{\mathbf{x},0}(l),\dots,j_{\mathbf{x},p}(l))\in\prod_{i=0}^p\{0,\dots,k_{x_i}\}
\]
 for $l=0,\dots,k_\mathbf{x}$ and $i_\mathbf{x}(l)\in\{0,\dots,p\}$ for $l=1,\dots,k_\mathbf{x}$ recursively as follows: For $l=0$ we choose $\mathbf{j}_\mathbf{x}(0)\coloneqq(0,\dots,0)$.
If $l\geq 1$, then we let $i_\mathbf{x}(l)$ be the smallest number in $\{i\in\{0,\dots,p\}\mid j_{\mathbf{x},i}(l-1)+1\leq k_{x_i}\}$ such that
\[
s_{x_{i_\mathbf{x}(l)},j_{\mathbf{x},i_\mathbf{x}(l)}(l-1)+1}=\min\{s_{x_i,j_{\mathbf{x},i}(l-1)+1}\mid i\in\{0,\dots,p\}\wedge j_{\mathbf{x},i}(l-1)+1\leq k_{x_i}\}
\]
 and we choose $j_{\mathbf{x},i_\mathbf{x}(l)}(l)\coloneqq j_{\mathbf{x},i_\mathbf{x}(l)}(l-1)+1$ and $j_{\mathbf{x},i}(l)\coloneqq j_{\mathbf{x},i}(l-1)+1$ for $i\not=i_\mathbf{x}(l)$.
In other words: $\mathbf{j}_\mathbf{x}(l)$ is obtained from $\mathbf{j}_\mathbf{x}(l-1)$ by incrementing the $i_\mathbf{x}(l)$-th entry. It is clear that this procedure stops at $l=k_\mathbf{x}$ with $j_{k_\mathbf{x}}=(k_{x_0},\dots,k_{x_p})$.

Using these choices, we now define for each $l=1,\dots,k_\mathbf{x}$ the tuple (written as a column vector for space reasons)
\[
\mathbf{x}_l\coloneqq
\begin{pmatrix}
(x_0,j_{\mathbf{x},0}(l))
\\\vdots
\\(x_{i_\mathbf{x}(l)-1},j_{\mathbf{x},i_\mathbf{x}(l)-1}(l))
\\(x_{i_\mathbf{x}(l)},j_{\mathbf{x},i_\mathbf{x}(l)}(l-1))
\\(x_{i_\mathbf{x}(l)},j_{\mathbf{x},i_\mathbf{x}(l)}(l))
\\(x_{i_\mathbf{x}(l)+1},j_{\mathbf{x},i_\mathbf{x}(l)+1}(l))
\\\vdots
\\(x_p,j_{\mathbf{x},p}(l))
\end{pmatrix}\in (X\times I)^{p+2},
\]
that is, the entries are exactly the pairs $(x_i,j_{\mathbf{x},i}(l))=(x_i,j_{\mathbf{x},i}(l-1))$ for $i\not=i_\mathbf{x}(l)$ and together with the two pairs
$(x_{i_\mathbf{x}(l)},j_{\mathbf{x},i_\mathbf{x}(l)}(l-1))\not=(x_{i_\mathbf{x}(l)},j_{\mathbf{x},i_\mathbf{x}(l)}(l))$.

The construction was made in such a way that if $\mathbf{x}$ lies in a $E$-penumbra of the multidiagonal in $X^{p+1}$, then $\mathbf{x}_l$ lies in a $E_\cU^{-1}\circ (E\swapcross\Diag[{[a,b]}])\circ E_\cU$-penumbra of the multidiagonal in $(X\times I)^{p+2}$.
Note also, that this construction is clearly $\Gamma$-equivariant thanks to the invariant choice of the $s_{x,j}$: We have $k_{\mathbf{x}\gamma}=k_\mathbf{x}$ and $(\mathbf{x}\gamma)_l=(\mathbf{x}_l)\gamma$ for each $l=1,\dots,k_\mathbf{x}$.

The prism operator for the given $\Gamma$-equivariant generalized coarse homotopy is now defined as
\begin{align*}
P_H\colon\ \CX^\Gamma_p(X,A;M)&\to \CX^\Gamma_{p+1}(Y,B;M),
\\
\sum_{\mathbf{x}\in X^{p+1}}m_{\mathbf{x}}\mathbf{x}&\mapsto \sum_{\mathbf{x}\in X^{p+1}}m_{\mathbf{x}}\sum_{l=1}^{k_\mathbf{x}}(-1)^{i_\mathbf{x}(l)}H_*(\mathbf{x}_l).
\end{align*}
Because of the above-mentioned properties, the image of this map indeed consists of $\Gamma$-invariant coarse chains. A calculation similar to the one for the classical prism operator known from basic algebraic topology (cf.\ \cite[proof of Theorem 2.10]{Hatcher_AT}) now shows that $P_H$ is a chain homotopy between $f_*$ and $g_*$.
\end{proof}

\begin{lem}
The ordinary $\Gamma$-equivariant coarse homology satisfies the excision axiom.
\end{lem}

\begin{proof}
We first show surjectivity of the map \(\HX^\Gamma_p(X\setminus C,A\setminus C;M)\to\HX^\Gamma_p(X,A;M)\). An element of the target is represented by a $\Gamma$-equivariant chain $c\in\CX^\Gamma_p(X;M)$ with $\partial c\in \CX^\Gamma_{p-1}(A;M)$. As $\supp(c)$ is a $\Gamma$-invariant penumbra of the multidiagonal, the set
\begin{align*}
E\coloneqq\{(y,z)\in X\times X\mid \exists (x_0,\dots,x_p)\in \supp(c),\,
\exists i,j\in\{0,\dots,p\}\colon x_i=y\wedge x_j=z\}\cup\Diag[X]
\end{align*}
is a $\Gamma$-invariant entourage which contains the diagonal. In particular, all $E$-penumbras $\Pen_E(B)$ of $\Gamma$-invariant subspaces $B$ are also $\Gamma$-invariant, and we have included the diagonal into $E$ to ensure that $\Pen_E(B)\supset B$.

Then the chain $c$ can be decomposed as $c=c_1+c_2$ with $c_2$ being the $\Gamma$-equivariant chain in $\CX^\Gamma_p(\Pen_E(A);M)$ consisting of all those summands of $c$ which are supported in $\Pen_E(A)^{p+1}$, and therefore $c_1\in\CX^\Gamma_p(X\setminus A;M)\subset \CX^\Gamma_p(X\setminus C;M)$.
We have
\begin{align*}
\partial c_1=\partial c-\partial c_2 &\in \CX^\Gamma_{p-1}(X\setminus C;M)\cap\CX^\Gamma_{p-1}(\Pen_E(A);M)
\\&= \CX^\Gamma_{p-1}(\Pen_E(A)\setminus C;M)
\end{align*}
and hence $c_1$ represents a class in
\(\HX^\Gamma_p(X\setminus C,\Pen_E(A)\setminus C;M)\)
which is clearly mapped to the class of $c$ in $\HX^\Gamma_p(X,A;M)$ under the inclusion of pairs of coarse $\Gamma$-spaces.

Now the surjectivity of the excision map follows by noting that the inclusion $(X\setminus C,A\setminus C)\subset (X\setminus C,\Pen_E(A)\setminus C)$ is a $\Gamma$-equivariant coarse equivalence and evoking the preceeding lemma. A coarse inverse up to closeness is obtained as follows: The excisiveness condition gives us an entourage $F$ such that
\(\Pen_E(A)\setminus C\subset \Pen_F(A\setminus C)\).
Note that by replacing $F$ with $\bigcap_{\gamma\in\Gamma} F\gamma$, if necessary, we can assume that $F$ is $\Gamma$-invariant. This works, because $E$ is $\Gamma$-invariant. The coarse inverse can now be defined as being the identity on $X\setminus \Pen_E(A)$ and mapping $\Pen_E(A)\setminus C$ to $A\setminus C$ in a way which is $F$-close to the identity and $\Gamma$-equivariant.

In order to show injectivity, we start with $c\in\CX^\Gamma_p(X\setminus C;M)$ with $\partial c\in \CX^\Gamma_{p-1}(A\setminus C;M)$ for which there is $d\in \CX^\Gamma_{p+1}(X;M)$ such that $\partial d-c\in\CX^\Gamma_p(A;M)$, because this is equivalent to $c$ representing an element of the kernel. Decomposing $d$ into $d_1\in\CX^\Gamma_{p+1}(X\setminus C;M)$ and $d_2\in \CX^\Gamma_{p+1}(\Pen_E(A);M)$ as above, where this time $E$ is obtained from the support of $d$, we see that
\begin{align*}
\partial d_1-c=\partial d-c-\partial d_2&\in \CX^\Gamma_p(X\setminus C;M)\cap\CX^\Gamma_p(\Pen_E(A);M)
\\&=\CX^\Gamma_p(\Pen_E(A)\setminus C;M).
\end{align*}
This means that the class of $c$ is mapped to zero under
\[
\HX^\Gamma_p(X\setminus C,A\setminus C;M)\to\HX^\Gamma_p(X\setminus C,\Pen_E(A)\setminus C;M),
\]
but we have already seen above that the latter map is an isomorphism. The claim follows.
\end{proof}

\begin{lem}\label{lem:ordinarystronghomologyflasqueness}
The ordinary $\Gamma$-equivariant coarse homology satisfies the flasqueness axiom.
\end{lem}
\begin{proof}
If $(X,A)$ is a flasque coarse $\Gamma$-space with $\Gamma$-equivariant $\phi$ as in Definitions~\ref{def:flasqueness} and~\ref{def:equiflasque}, then a chain contraction is given by
the $M$-linear extension $\CX^\Gamma_*(X,A;M)\to\CX^\Gamma_{*+1}(X,A;M)$ of the mapping
\[
(x_0,\dots,x_p)\mapsto\sum_{n\in\N}\sum_{i=0}^p(-1)^i\big(\phi^{n+1}(x_0),\dots,\phi^{n+1}(x_i), \phi^{n}(x_i),\dots,\phi^{n}(x_i)\big).
\]
The sums are finite on all relatively compact $K^{p+1}$, because the image of $\phi^n$ does not intersect $K$ for $n$ large enough and $\phi$ is close to the identity. Furthermore, equicontrolledness of the $\phi^n$ together with closeness of $\phi$ to $\id$ ensure that the support of $c$ is a penumbra of the multidiagonal.
\end{proof}

\subsection{Ordinary coarse cohomology}
\label{sec:ordinarycoarsecohomology}

Roe defined coarse cohomology in \cite{RoeCoarseCohomIndexTheory} only for metric spaces without group action and with coefficients in $\R$.
We generalize the definition below in the obvious way.

It is a bit surprising and disappointing that, although ordinary coarse cohomology is arguably the most elementary coarse cohomology theory whose groups can be defined on all coarse $\Gamma$-spaces, we nevertheless have to restrict ourselves to countably generated coarsely connected isocoarse $\Gamma$-spaces in order to be able to prove the excision axiom.

\begin{defn}
Let $X$ be a coarse space and $M$ an abelian group. For $p\in\N$ we define $\CX^p(X;M)$ as the group of all functions $\varphi\colon X^{p+1}\to M$ whose support \(\supp(\varphi)\) intersects each penumbra of the multidiagonal in a relatively compact subset.
For $p\in\Z\setminus\N$ we define $\CX^p(X;M)\coloneqq 0$.
These groups together with the Alexander--Spanier coboundary maps $\delta\colon\CX^p(X;M)\to\CX^{p+1}(X;M)$, that is
\begin{equation*}
\delta\varphi(x_0,\dots,x_{p+1})\coloneqq\sum_{i=0}^{p+1}(-1)^i\varphi(x_0,\dots,x_{i-1},x_{i+1},\dots,x_{p+1}),
\end{equation*}
constitute a cochain complex.
If $A\subset X$ is a coarse subspace, then we define the cochain complex $\CX^*(X,A;M)$ as the kernel of the surjective cochain map $\CX^*(X;M)\to \CX^*(A;M)$ and its cohomology is the \emph{ordinary coarse cohomology} $\HX^*(X,A;M)$ of the pair $(X,A)$ with coefficients in $M$.

If $\Gamma$ is a discrete group acting isocoarsely on $(X,A)$ from the right and acting on $M$ from the~left, then it also acts on the cochain complex $\CX^*(X,A;M)$ from the left via $(\gamma\varphi)(x_0,\dots,x_p)\allowbreak\coloneqq \gamma(\varphi( x_0\gamma,\dots,x_p\gamma))$. Hence we have a quotient cochain complex
$\CX_\Gamma^*(X,A;M)\coloneqq\linebreak\Gamma\backslash\CX^*(X,A;M)$,
i.e., we divide out the subcomplexes generated by $\gamma\varphi-\varphi$ for all $\gamma$, $\varphi$, and we call its cohomology the $\Gamma$-equivariant coarse cohomology $\HX_\Gamma^*(X,A;M)$.
\end{defn}

Ordinary $\Gamma$-equivariant coarse cohomology is clearly contravariantly functorial under $\Gamma$-equivariant maps between pairs of coarse $\Gamma$-spaces.

\begin{lem}
The ordinary $\Gamma$-equivariant coarse cohomology satisfies the exactness axiom.
\end{lem}

\begin{proof}
Again, this is implied by the fact that the short exact sequence
\[
0\to\CX^*(X,A;M)\to\CX^*(X;M)\to\CX^*(A;M)\to 0
\]
admits a $\Gamma$-equivariant split $\CX^*(A;M)\to \CX^*(X;M)$, defined by extending cocycles by $0$ on all $(p+1)$-tuples not supported within $A^{p+1}$, and hence the quotient complexes also form a short exact sequence.
\end{proof}

\begin{lem}\label{lem:ordinarycoarsecohomologystronghomotopy}
The ordinary $\Gamma$-equivariant coarse cohomology satisfies the strong homotopy axiom. In particular, it is also contravariantly functorial under \emph{closeness classes} of coarse maps.
\end{lem}

\begin{proof}
The proof is essentially dual to the proof of Lemma~\ref{lem:ordinarycoarsehomologystronghomotopy}. Using exactly the same $k_{\mathbf{x}}$, $i_\mathbf{x}(l)$ and $\mathbf{x}_l$ constructed from an arbitrary $\mathbf{x}\in X^{p+1}$ we can define the cochain homotopy
\begin{align*}
P_H\colon\ \CX_\Gamma^{p+1}(X,A;M)&\to \CX_\Gamma^p(X,A;M),
\\
P_H(\varphi)(\mathbf{x})&\coloneqq\sum_{l=0}^{k_\mathbf{x}}(-1)^{i_\mathbf{x}(l)} \varphi(H_*(\mathbf{x}_l)).
\end{align*}
It does the job.
\end{proof}

\begin{lem}
On the admissible category of pairs of countably generated coarsely connected isocoarse $\Gamma$-spaces the ordinary $\Gamma$-equivariant coarse homology satisfies the excision axiom.
\end{lem}
\begin{proof}
Assume for a moment that there is a $\Gamma$-equivariant map $r\colon X\to X\setminus C$ which restricts to the identity on $X\setminus C$, maps $C$ (and hence also $A$) into $A\setminus C$ and restricts to a controlled map on $\Pen_E(X\setminus C)$ for every penumbra $E$ of $X$.
We claim that $r$ then induces an inverse to the excision homomorphism on the cohomology groups.

Let $\varphi\in\CX^p(X\setminus C,A\setminus C;M)$ and $P$ be a penumbra of the multidiagonal in $X^{p+1}$. Note that the latter is equivalent to $P$ being contained in $\bigcup_{x\in X}(\Pen_E(\{x\}))^{p+1}$ for some entourage $E$ of~$X$.
Hence, if one component $x_i$ of $\mathbf{x}=(x_0,\dots,x_p)\in P$ is not contained in $\Pen_{E\circ E^{-1}}(X\setminus C)$, then none of the components is contained in $X\setminus C$. In particular, all components of $\mathbf{x}$ are mapped to $A\setminus C$ by $r$, implying $r^{p+1}(\mathbf{x})\in(A\setminus C)^{p+1}$ and therefore $r^*\varphi(\mathbf{x})=0$.
This shows $\supp(r^*\varphi)\cap P\subset (\Pen_{E\circ E^{-1}}(X\setminus C))^{p+1}$. By assumption, the restriction of $r$ to $\Pen_{E\circ E^{-1}}(X\setminus C)$ is a controlled map. Together with the other assumption on $r$ we see that it is close to the identity on $\Pen_{E\circ E^{-1}}(X\setminus C)$ and therefore must also be proper, that is, it is even a coarse map. Now, controlledness of $r$ implies that $r^{p+1}(\supp(r^*\varphi)\cap P)$ is a penumbra of the multidiagonal in $X\setminus C$. In addition, it is also contained in $\supp(\varphi)$ and hence it is relatively compact due to the support condition on $\varphi$. Exploiting that $r$ is also proper we conclude that $\supp(r^*\varphi)\cap P$ is relatively compact.

We have just shown that $r^*$ maps $\CX^*(X\setminus C,A\setminus C;M)$ to $\CX^*(X,A;M)$. The map is clearly compatible with the group action and the coboundary map. Therefore it induces a map $r^*\colon \HX_\Gamma^*(X\setminus C,A\setminus C;M)\to\HX_\Gamma^*(X,A;M)$ on cohomology and we claim that it is an inverse to the excision map. One of the two compositions is trivially the identity, because $r$ restricts to the identity on $X\setminus C$. The other composition is homotopic to the identity on cochain level via the ``prism operator''
\[
P(\varphi)(x_0,\dots,x_{p-1})\coloneqq\sum_{l=0}^{p-1}(-1)^l\varphi(x_0,\dots,x_l,r(x_l),\dots,r(x_{p-1})),
\]
where the proof that this formula really defines a map $\CX_\Gamma^p(X,A;M)\to\CX_\Gamma^{p-1}(X,A;M)$ is completely analogous to the above proof that $r^*$ maps $\CX_\Gamma^p(X\setminus C,A\setminus C;M)$ to $\CX_\Gamma^p(X,A;M)$.

It remains to prove that the map $r$ exists. If $X$ is a countably generated isocoarse $\Gamma$-space, then we can pick a sequence $\Diag[X]=E_0\subset E_1\subset E_2\subset\cdots$ of $\Gamma$-invariant entourages such that every entourage of $X$ is contained in one of the $E_n$. Using the alternative excisiveness condition~\eqref{eq:excisiveness} we find for each $E_n$ an entourage $F_n$ such that $A\cap\Pen_{E_n}(X\setminus C)\subset \Pen_{F_n}(A\setminus C)$. We can then define $r$ by patching together the identity on $X\setminus C=\Pen_{E_0}(X\setminus C)$ with maps
\[
\Pen_{E_{n}}(X\setminus C)\setminus \Pen_{E_{n-1}}(X\setminus C)\to A\setminus C
\]
that are $F_n$-close to the identity. The assumption of coarse connectedness ensures that this defines~$r$ on the whole space $X$. All the above-mentioned assumptions on $r$ are readily verified.
\end{proof}

Finally, the flasqueness axiom is proven by dualizing the proof of Lemma~\ref{lem:ordinarystronghomologyflasqueness}.

\begin{lem}
The ordinary $\Gamma$-equivariant coarse cohomology satisfies the flasqueness axiom.
\end{lem}

\subsection{The stable Higson corona}\label{sec:sHigCor}

The stable Higson corona was introduced in \cite{EmeMeyDualizing}.
A function $f\colon X\to D$ from a coarse space $X$ to a \textCstar-algebra $D$ is said to
\begin{enumerate}\itemsep=0pt
\item[$\bullet$] vanish at infinity if for each $\varepsilon>0$ there is a relatively compact subset $K$ such that $f$ is bounded by $\varepsilon$ outside of $K$,
\item[$\bullet$] have \emph{vanishing variation} if for every entourage $E$ of $X$ the function
\[
\Var_Ef\colon\quad X\to [0,\infty),\qquad x\mapsto \sup\{\|f(x)-f(y)\|\mid(x,y)\in E\}
\]
vanishes at infinity.
\end{enumerate}
Denote by $\sHigFuAlg(X;D)$ the \textCstar-algebra of all bounded functions of vanishing variation from $X$ to $D\otimes\Kom$, where $\Kom$ denotes the \textCstar-algebra of compact operators on a separable infinite dimensional Hilbert space. Furthermore, we define $\sHigFuAlgRed(X;D)$ to be the \textCstar-algebra of all bounded functions of vanishing variation from $X$ into the stable multiplier algebra $\multiplier^s(D)\coloneqq\multiplier(D\otimes\Kom)$ of $D$ such that $f(x)-f(y)\in D\otimes\Kom$ for all $x,y\in X$ which lie in the same coarse component.
Both of them contain the ideal $\VanInfFuAlg(X;D\otimes\Kom)$ of all bounded functions $X\to D\otimes\Kom$ vanishing at infinity. The \textCstar-algebras $\sHigFuAlg(X;D)$, $\sHigFuAlgRed(X;D)$, $\VanInfFuAlg(X;D\otimes\Kom)$ have been introduced in \cite[proof of Proposition~3.7]{EmeMeyDualizing} with the additional condition that the functions should also be Borel, but all statements that are relevant to us still hold without this condition.

The \emph{stable Higson corona of $X$ with coefficients in $D$} is the quotient \textCstar-algebra
\[
\sHigCor(X;D)\coloneqq \sHigFuAlg(X;D)/\VanInfFuAlg(X;D\otimes\Kom)
\]
and the \emph{reduced stable Higson corona of $X$ with coefficients in $D$} is the quotient \textCstar-algebra
\[
\sHigCorRed(X;D)\coloneqq \sHigFuAlgRed(X;D)/\VanInfFuAlg(X;D\otimes\Kom).
\]
All of the above function algebras are clearly contravariantly functorial under closeness classes of coarse maps.

If $X$ is moreover a topological coarse space, then there are the sub-\textCstar-algebras of continuous functions
\[
\Cz(X;D\otimes\Kom)\subset\VanInfFuAlg(X;D\otimes\Kom),\qquad
\sHigCom(X;D)\subset\sHigFuAlg(X;D),\qquad
\sHigComRed(X;D)\subset\sHigFuAlgRed(X;D),
\]
which are contravariantly functorial under continuous coarse maps and the inclusions induce isomorphisms
\begin{gather}
\sHigCor(X;D)\cong \sHigCom(X;D)/\Cz(X;D\otimes\Kom),\nonumber
\\
\sHigCorRed(X;D)\cong \sHigComRed(X;D)/\Cz(X;D\otimes\Kom),\label{eq:sHigCorcontinuousrepresentatives}
\end{gather}
see \cite[proof of Proposition 3.7]{EmeMeyDualizing}.
The right hand sides are the familiar original definitions of the stable Higson coronas \cite[Definitions 3.2 and 5.4]{EmeMeyDualizing}.

For our purposes, we need the relative versions:
\begin{defn}\label{defn:stableHigsoncoronas}
If $(X,A)$ is a pair of coarse spaces, then we define its (\emph{unreduced and reduced$)$ relative stable Higson coronas with coefficients in $D$} as the kernels of the restrictions to $A$,
\begin{gather*}
\sHigCor(X,A;D)\coloneqq\ker(\sHigCor(X;D)\to\sHigCor(A;D)),
\\
\sHigCorRed(X,A;D)\coloneqq\ker\big(\sHigCorRed(X;D)\to\sHigCorRed(A;D)\big).
\end{gather*}
If $(X,A)$ is even a pair of topological coarse spaces, then we also define its (\emph{unreduced and reduced$)$ relative stable Higson compactifications with coefficients in $D$} as the kernels of the restrictions to~$A$,
\begin{gather*}
\sHigCom(X,A;D)\coloneqq\ker(\sHigCom(X;D)\to\sHigCom(A;D)),
\\
\sHigComRed(X,A;D)\coloneqq\ker\big(\sHigComRed(X;D)\to\sHigComRed(A;D)\big).
\end{gather*}
\end{defn}

Note that $\sHigCor(X,A;D)=\sHigCorRed(X,A;D)$ and $\sHigCom(X,A;D)=\sHigComRed(X,A;D)$ whenever $A$ meets all coarse components of $X$, whereas for $A=\varnothing$ we recover $\sHigCor(X;D)$, $\sHigCorRed(X;D)$, $\sHigCom(X;D)$ and~$\sHigComRed(X;D)$. Also, $\sHigCor(X;D)=\sHigCorRed(X,K;D)$ for any non-empty relatively compact $K\subset X$.
The relative stable Higson coronas are clearly contravariantly functorial under coarse maps of pairs and the relative stable Higson compactifications are so under continuous coarse maps of pairs.

\begin{rem}\label{rem:howtoreducestableHigsoncorona}
For coarsely disconnected spaces $X$, it is debatable what the correct definition of the reduced stable Higson corona is. One could also ask for $f(x)-f(y)\in D\otimes\Kom$ for all $x,y\in X$, not just for those in the same connected component, just like the reduced singular (co-)chain complexes are obtained from the unreduced ones by augmentation with one copy of $\Z$ and not with one copy for each connected component of the space. As a result, the $\K$-theory of this reduced version of the stable Higson corona would not satisfy flasqueness (cf.\ Lemma~\ref{lem:sHigCorRedFlasque}), which is in analogy to reduced singular (co-)homology not vanishing on disjoint unions of more than one contractible spaces.

Our choice of definition is due to the fact that flasqueness is an indispensable technical tool and we need a version of this coarse cohomology theory satisfying this axiom.
Later on, in Section~\ref{sec:reducedtransgression}, we will also discuss a more flexible notion of topological reduced coarse (co-)homology, which also allows for componentwise augmentation.
\end{rem}

{\sloppy Note that we would immediately obtain long exact sequences for the $\K$-theory groups $\K_{1-*}(\sHigCor(-,-;D))$ and $\K_{1-*}\big(\sHigCorRed(-,-;D)\big)$ if only the restriction maps to $A$ were surjective. Unfortunately, this is a rather delicate problem which has only been solved in special cases. Willett had shown it for pairs of coarsely connected proper metric spaces for the reduced stable Higson coronas with trivial coefficients $D=\C$ \cite[Lemma 3.4]{WillettHomological}, but the proof also works for arbitrary coefficients and the claim for the unreduced stable Higson corona is an easy consequence. We will show in Lemma~\ref{lem:Willettssurjectivity} below how to deduce the generalization to coarsely disconnected proper metric spaces.
Bunke and Engel have discussed a generalization to a bigger class of coarse spaces, but also could not overcome all dificulties \cite[Section 5]{BunkeEngel_CoarseCohomologyTheories}.

}

\begin{lem}\label{lem:Willettssurjectivity}
If $X$ is a proper metric space and $A\subset X$, then the restriction maps \(\sHigCor(X;D)\to\sHigCor(A;D)\) and \(\sHigCorRed(X;D)\to\sHigCorRed(A;D)\) are surjective.
\end{lem}
\begin{proof}
The reduced and unreduced cases are clearly equivalent to each other, and they follow from surjectivity of the restriction map \(\sHigCom(X;D)\to\sHigCom(A;D)\), so let us show the latter.

First of all we may assume that $A$ meets all coarse components of $X$, because the extensions can simply be chosen to be zero on all coarse components which are disjoint from $A$.

Second, it suffices to consider the case, where $X$ and $A$ have at most countably many coarse components. In order to see this, let $f\in \sHigCom(A;D)$. For each $R\in\N$ and $n\in\N\setminus\{0\}$ we find a~relatively compact subset $K_{R,n}\subset A$ such that $\Var_{E_R}f(x)<\frac1n$ for all $x\in A\setminus K_{R,n}$. As~each~$K_{R,n}$ meets only finitely many coarse components of $A$, their union $\bigcup_{R,n}K_{R,n}$ are contained in a countable union of coarse components of $A$, which we denote by $A'\subset A$, and let $X'\subset X$ be the union of the corresponding coarse components of $X$. For all $x\in A\setminus A'$ we therefore have $\Var_{E_R}f(x)<\frac1n$ for all $R\in\N$ and $n\in\N\setminus\{0\}$, so that $f$ is a constant function on the coarse component containing $x$.
Thus, if we know that $\sHigCom(X';D)\to\sHigCom(A';D)$ is surjective, then we can pick any preimage $F'\in\sHigCom(X';D)$ of $f|_{A'}$ and extend it to a function $F\in \sHigCom(X;D)$ which is constant on each coarse component of $X\setminus X'$ and such that $F|_A=f$.

Thirdly, we can deduce the statement from the one for coarsely connected spaces as follows. Let $S\subset \big\{n^2\mid n\in\N\big\}$ be a set of square numbers which has the same cardinality as the set $\coarsecomp(A)$ of coarse components of $A$ (and hence also $X$) and let $S\to A, s\mapsto x_s$ be a map which induces a bijection $S\to \coarsecomp(A)$. If $d\colon X\times X\to [0,\infty]$ is the original coarsely disconnected proper metric, then we can define the coarsely connected proper metric $d'\colon X\times X\to [0,\infty)$ on $X$ as the largest metric such that $d'\leq d$ and $d(x_s,x_{t})\leq |s-t|$ for all $s,t\in S$. It is then clear that changing the metric from $d$ to $d'$ does not change the \textCstar-algebras \(\sHigCom(X;D)\) and \(\sHigCom(A;D)\).

Finally, for coarsely connected proper metric spaces the claim is proven exactly as \cite[Lemma~3.4]{WillettHomological}.
\end{proof}

\begin{lem}\label{lem:sHigCorExcision}
If $X$ is a proper metric space and $C\subset A\subset X$ satisfy the coarse excisiveness condition, then restriction to $X\setminus C$ yields isomorphism
\(\sHigCor(X,A;D)\cong\sHigCor(X\setminus C,A\setminus C;D)\) and \(\sHigCorRed(X,A;D)\cong\sHigCorRed(X\setminus C,A\setminus C;D)\).
\end{lem}

\begin{proof}
Injectivity is straightforward: If a function on $X$ represents an element of the kernel, then its restrictions to both $A$ and $X\setminus C$ vanish at infinity. These two sets cover $X$ and the relatively compact subsets of $X$ are exactly the union of relatively compact subsets of $A$ with relatively compact subsets of $X\setminus C$. This implies that the function itself vanishes at infinity, and therefore represents the zero element.

For surjectivity we use that the right square of the diagram
\[
\xymatrix{
0\ar[r]&\sHigCorRed(X,A;D)\ar[r]\ar[d]&\sHigCorRed(X;D)\ar[r]\ar[d]&\sHigCorRed(A;D)\ar[r]\ar[d]&0
\\0\ar[r]&\sHigCorRed(X\setminus C,A\setminus C;D)\ar[r]&\sHigCorRed(X\setminus C;D)\ar[r]&\sHigCorRed(A\setminus C;D)\ar[r]&0
}
\]
is a pullback diagram. For coarsely connected spaces and trivial coefficients, this is \cite[Lem\-ma~3.3]{WillettHomological}, but the proof goes through for coarsely disconnected spaces and arbitrary coefficients as well. Furthermore, the same is true for the analogous unreduced version of the diagram.
The claim now follows from an easy diagram chase.
\end{proof}

\begin{lem}\label{lem:sHigCorRedFlasque}
For flasque proper metric spaces $X$, the group $\K_{1-*}\big(\sHigCorRed(X;D)\big)$ vanishes.
\end{lem}

\begin{proof}
This is the straightforward adaption of \cite[Proposition 3.7]{WillettHomological} (see also \cite[Theorem 5.2]{EmeMeyDualizing}) to the case of not necessarily coarsely connected spaces and with coefficients.
Let us briefly recall Willett's sketch of proof, to see why our definition of the reduced stable Higson corona for coarsely disconnected spaces is the one fulfilling this property.

Denote by $H$ the infinite dimensional separable Hilbert space on which $\Kom\coloneqq\Kom(H)$ is modelled, such that we also have $\multiplier^s(D)=\Lin(D\otimes H)$. Choose an isometric isomorphism $H\cong\oplus_\N H$ and let $V_n\colon H\to H$ be the isometric inclusion as the $n$-th summand with respect to this decomposition. Furthermore, let $\phi\colon X\to X$ be the coarse map from the definition of flasqueness. For every $f\in\sHigComRed(X;D)$ one now defines
\[
\nu f\colon\quad X\to \multiplier^s(D),\qquad x\mapsto \sum_{n\in\N}(\id_D\otimes V_n)\circ f(\phi^n(x))\circ (\id_D\otimes V_n)^*.
\]

For each two $x$, $y$ in the same coarse component of $X$, the sequences $(\phi^n(x))_{n\in\N}$ and $(\phi^n(y))_{n\in\N}$ stay close to each other and go to infinity. Hence the vanishing variation of $f$ implies that the differences $f(\phi^n(x))-f(\phi^n(y))$ converge in norm to zero for $n\to\infty$. As they are also compact, the difference
\[
\nu f(x)-\nu f(y)= \sum_{n\in\N}(\id_D\otimes V_n)\circ \big(f(\phi^n(x))-f(\phi^n(y))\big)\circ (\id_D\otimes V_n)^*
\]
is compact, too. Note that the part of this argument using vanishing variation does not work for~$x$,~$y$ in different coarse components, and that is the reason why this lemma only holds for our definition of the reduced stable Higson corona and not the alternative mentioned in Remark~\ref{rem:howtoreducestableHigsoncorona}.

Furthermore, it is straightforward to check that $\nu f$ has vanishing variation and that the resulting $*$-homomorphism $\nu\colon \sHigComRed(X;D)\to \sHigComRed(X;D)$ descends to a $*$-homomorphism $\nu\colon \sHigCorRed(X;D)\allowbreak\to \sHigCorRed(X;D)$ on the corona. The proof is then completed by the following Eilenberg swindle: one shows that on $\K$-theory we have $\nu_*(x)=x+\nu_*(x)$ for all $x\in \K_{1-*}\big(\sHigCorRed(X;D)\big)$, so the latter group vanishes.
\end{proof}

\begin{thm}
The contravariant functors $\K_{1-*}\big(\sHigCorRed(-,-;D)\big)$ and $\K_{1-*}(\sHigCor(-,-;D))$ are co\-arse cohomology theories on the category of proper metric spaces and coarse maps. The first one satisfies the flasqueness axiom and the second one the coronality axiom.
\end{thm}

\begin{proof}
The long exact sequences and excision are immediate from Lemmas~\ref{lem:Willettssurjectivity} and~\ref{lem:sHigCorExcision}. The coronality axiom for the second one is clear, because $\sHigCor(K;D)=0$ for all relatively compact spaces $K$.

The flasqueness axiom for the first one follows from \ref{lem:sHigCorRedFlasque} for the absolute groups together with the long exact sequence for the relative groups.
Flasqueness also implies the homotopy invariance of the first functor by Lemma~\ref{lem:flasquenessimplieshomotopyinvariance}.

The homotopy invariance of the second functor is a bit tricky, but it can also be deduced from the flasqueness of the first one as follows.
The proof of Lemma~\ref{lem:flasquenessimplieshomotopyinvariance} worked by showing that the hypothesis of Lemma~\ref{lem:homotopydomainisomorphismlemma} is satisfied.
Thus, we already know that the first functor $\K_{1-*}\big(\sHigCorRed(-,-;D)\big)$ satisfies this hypothesis and it remains to show that the second functor $\K_{1-*}(\sHigCor(-,-;D))$ does so as well.
Given any pair $(X,A)$ of proper metric spaces, we denote by $\coarsecomp\coloneqq\coarsecomp(X,A)$ the set of those coarse components of $X$ which do not contain any points of $A$.
Then there are canonical short exact sequences
\[
0\to \sHigCor(Y,B;D)\to \sHigCorRed(Y,B;D)\to \prod_{\coarsecomp}\frac{\multiplier^s(D)}{D\otimes\Kom}\to 0
\]
for all those pairs $(Y,B)$ of proper metric spaces derived from $(X,A)$ which appear in Lemma~\ref{lem:homotopydomainisomorphismlemma}.
Consider the maps between the associated long exact $\K$-theory sequences induced by the maps~$i_0$,~$i_\rho$ and~$p$.
For $i_0$ this yields a diagram
\[
\xymatrix{
\K_{2-p}\big(\sHigCorRed\big(\overline{X}_0^\rho;D\big)\big)\ar[r]\ar[d]_{i_0^*}^{\cong}
&\K_{2-p}\Bigl(\prod_{\coarsecomp}\frac{\multiplier^s(D)}{D\otimes\Kom}\Bigr)\ar[r]\ar@{=}[d]
&\K_{1-p}\big(\sHigCor\big(\overline{X}_0^\rho;D\big)\big)\ar[r]\ar[d]_{i_0^*}
&
\\\K_{2-p}\big(\sHigCorRed(X;D)\big)\ar[r]
&\K_{2-p}\Bigl(\prod_{\coarsecomp}\frac{\multiplier^s(D)}{D\otimes\Kom}\Bigr)\ar[r]
&\K_{1-p}(\sHigCor(X;D))\ar[r]
&
\\\qquad\qquad\qquad\ar[r]&\K_{1-p}\big(\sHigCorRed\big(\overline{X}_0^\rho;D\big)\big)\ar[r]\ar[d]_{i_0^*}^{\cong}
&\K_{1-p}\Bigl(\prod_{\coarsecomp}\frac{\multiplier^s(D)}{D\otimes\Kom}\Bigr)\ar@{=}[d]
&
\\\qquad\qquad\qquad\ar[r]&\K_{1-p}\big(\sHigCorRed(X;D)\big)\ar[r]
&\K_{1-p}\Bigl(\prod_{\coarsecomp}\frac{\multiplier^s(D)}{D\otimes\Kom}\Bigr)
&
}
\]
and the five-lemma implies that the middle verticle map is an isomorphism. Similarly, the maps~$i_\rho$ and $p$ induce isomorphisms under the second functor.
\end{proof}

We now briefly discuss two equivariant generalizations for pairs of coarsely connected proper metric $\Gamma$-spaces, where we assume that the coefficient \textCstar-algebra $D$ is a $\Gamma$-\textCstar-algebra. In this case, there are canonical actions of $\Gamma$ on $\sHigCorRed(X,A;D)$ and $\sHigCor(X,A;D)$ in which $\gamma [f]$ is represented by the function $\gamma f\colon x\mapsto \gamma(f(x\gamma))$. Recall that we only consider discrete groups (which even have to be countable, because they act properly on coarsely connected proper metric spaces) and thus do not have to worry about continuity of the group actions.

\begin{thm}If $\rtimes\Gamma$ denotes an exact crossed product functor in the sense of {\rm \cite[Definition~3.1]{BaumGuentnerWillettExpandersExactcrossedproducts}}, then the functors $\K_{1-*}\big(\sHigCorRed(-,-;D)\rtimes \Gamma\big)$ and $\K_{1-*}(\sHigCor(-,-;D)\rtimes\Gamma)$ are $\Gamma$-equivariant coarse cohomology theories satisfying the flasqueness or the coronality axiom, respectively.
\end{thm}
\begin{proof}
Exactness and excision again follow directly from Lemmas~\ref{lem:Willettssurjectivity} and~\ref{lem:sHigCorExcision}, and coronality of the unreduced versions is again clear.

The flasqueness axiom for the reduced version is proven exactly as in the non-equivariant case, because the $*$-homomorphism $\nu$ in the proof of Lemma~\ref{lem:sHigCorRedFlasque} will now be equivariant.

Homotopy invariance of the reduced version follows. For the unreduced version we prove it just as in the proof of the previous theorem, where the induced $\Gamma$-action on $\coarsecomp\coloneqq \coarsecomp(X,A)$ and the action on $D$ turn $\prod_{\coarsecomp}\frac{\multiplier^s(D)}{D\otimes\Kom}$ into a $\Gamma$-\textCstar-algebra.
\end{proof}

The above equivariant version is not the most sophisticated. According to the work of Emerson and Meyer \cite{EmeMeyDescent,EmeMeyEquivariantCoassembly}, better choices are $\Ktop_{1-*}(\Gamma,\sHigCor(-,-;D))$ and $\Ktop_{1-*}\big(\Gamma,\sHigCorRed(-,-;D)\big)$.
We do not need to recall their definiton, because the only thing relevant to us is that there are natural isomorphisms
\begin{equation}\label{eq:KtopDef}
\Ktop_*(\Gamma,C)\cong \K_*((C\maxtensor\sfP)\maxcrossed\Gamma)\cong \K_*((C\maxtensor\sfP)\redcrossed\Gamma),
\end{equation}
which hold for both the reduced and full crossed product functor and for any $\Gamma$-\textCstar-algebra~$C$ (see \cite[Theorem~22]{EmeMeyDescent}, which is based on~\cite[Theorems~5.2 and~10.2]{MeyerNest}\footnote{Separability of the \textCstar-algebras is a standing assumption in the latter reference, but the isomorphisms here also hold for non-separable $C$, because we can take the direct limit over all separable sub-\textCstar-algebras and exploit that all constructions in~\eqref{eq:KtopDef} are co-continuous under such.}). Here, $\sfP$ denotes a~certain $\Gamma$-\textCstar-algebra which supports the so-called Dirac morphism $\sfD\in\KK^\Gamma(\sfP,\C)$.
Exactly the same proof as for the previous theorem, just with $(\blank\maxtensor\sfP)\maxcrossed$ instead of $\rtimes$, shows the following.\looseness=1

\begin{thm}
The functors $\Ktop_{1-*}\big(\Gamma,\sHigCorRed(-,-;D)\big)$ and $\Ktop_{1-*}(\Gamma,\sHigCor(-,-;D))$ are $\Gamma$-equivariant coarse cohomology theories satisfying the flasqueness or the coronality axiom, respectively.
\end{thm}

Let us briefly discuss the relation between all of these theories. First of all, the Baum--Connes assembly map yields natural transformations
\begin{gather*}
\Ktop_{1-*}\big(\Gamma,\sHigCorRed(-,-;D)\big)\to \K_{1-*}\big(\sHigCorRed(-,-;D)\maxcrossed\Gamma\big)\to \K_{1-*}\big(\sHigCorRed(-,-;D)\rtimes\Gamma\big),
\\[.5ex]
\Ktop_{1-*}(\Gamma,\sHigCor(-,-;D))\to \K_{1-*}(\sHigCor(-,-;D)\maxcrossed\Gamma)\to \K_{1-*}(\sHigCor(-,-;D)\rtimes\Gamma),
\end{gather*}
which are compatible with the connecting homomorphisms.

Second, there is a commutative diagram
\[
\xymatrix{
0\ar[r]& \prod_{\coarsecomp}(D\otimes\Kom)\ar[r]\ar[d]&\prod_{\coarsecomp} \multiplier^s(D)\ar[r]\ar[d]&\prod_{\coarsecomp}\frac{\multiplier^s(D)}{D\otimes\Kom}\ar[r]\ar@{=}[d]&0
\\0\ar[r]& \sHigCor(X,A;D)\ar[r]& \sHigCorRed(X,A;D)\ar[r]& \prod_{\coarsecomp}\frac{\multiplier^s(D)}{D\otimes\Kom}\ar[r]& 0
}
\]
with exact rows, where $\coarsecomp\coloneqq\coarsecomp(X,A)$ again denotes the set of those coarse components of $X$ which do not contain any points of $A$. The $\K$-theory of $\prod_{\coarsecomp}\multiplier^s(D)$, $\bigl(\prod_{\coarsecomp}\multiplier^s(D)\bigr)\rtimes\Gamma$ and $\bigl(\prod_{\coarsecomp}(\multiplier^s(D))\maxtensor\sfP\bigr)\maxcrossed\Gamma$, vanishes by an Eilenberg swindle and hence we can easily read off the following long exact sequences in $\K$-theory:
\begin{equation}\label{eq:reducedunreducedLES}
\begin{split}
&\xymatrix@R=2ex{
\myvdots\ar[d]\restore&\myvdots\ar[d]&\myvdots\ar[d]
\\
\K_{1-*}\big(\prod_{\coarsecomp}(D\otimes\Kom)\big)\ar[d]&\K_{1-*}\big(\big(\prod_{\coarsecomp}(D\otimes\Kom)\big) \rtimes\Gamma\big)\ar[d]&\Ktop_{1-*}\big(\Gamma,\prod_{\coarsecomp}(D\otimes\Kom)\big)\ar[d]
\\
\K_{1-*}(\sHigCor(X,A;D))\ar[d]&\K_{1-*}(\sHigCor(X,A;D)\rtimes\Gamma)\ar[d]&\Ktop_{1-*} (\Gamma,\sHigCor(X,A;D))\ar[d]
\\
\K_{1-*}\big(\sHigCorRed(X,A;D)\big)\ar[d]&\K_{1-*}\big(\sHigCorRed(X,A;D)\rtimes\Gamma\big)\ar[d] &\Ktop_{1-*}\big(\Gamma,\sHigCorRed(X,A;D)\big)\ar[d]
\\
\K_{-*}\big(\prod_{\coarsecomp}(D\otimes\Kom)\big)\ar[d]&\K_{-*} \big(\big(\prod_{\coarsecomp}(D\otimes\Kom)\big)\rtimes\Gamma\big)\ar[d]&\Ktop_{-*}\big(\Gamma;\prod_{\coarsecomp} (D\otimes\Kom)\big)\ar[d]
\\
\myvdots&\myvdots&\myvdots
}\end{split}\end{equation}
The first of these three sequences decomposes for an unbounded and coarsely connected $X$ and empty $A$ into the short exact sequence
\[
0\to\K_{1-*}(D)\to\K_{1-*}(\sHigCor(X;D))\to\K_{1-*}\big(\sHigCorRed(X;D)\big)\to 0
\]
(cf.\ \cite[Lemma~3.10]{EmeMeyDualizing}), but the proof does not generalize to the equivariant case, because it involves evaluation $*$-homomorphisms at points of $X$, which are in general not equivariant.
The~closest we come to an equivariant generalization is the following: If $X$ contains a flasque subspace~$B$ which is contained in the union of the components in $\coarsecomp(X,A)$ and contains points from all components in $\coarsecomp(X,A)$ (in other words, the inclusion $B\subset X$ induces a bijection between the set $\coarsecomp(B)$ of coarse components of $B$ and $\coarsecomp(X,A)$), then the above long exact sequences are mapped by restriction to the corresponding ones for $B$. But then the arrows leaving the reduced terms in the sequences for $(X,A)$ factor through the corresponding terms for~$B$, which vanish because of flasqueness. Thus, we get a decomposition of the three long exact sequences into short exact sequences in this case, too.

Third, as the long exact sequences \eqref{eq:reducedunreducedLES} indicate, there is also a way of obtaining the reduced groups from the unreduced stable Higson corona with the help of a mapping cone construction. In fact, it is this construction which completely establishes the analogy to reduced cohomology theories as we will see later in Sections~\ref{sec:reducedtransgression} and~\ref{sec:coassembly} (keeping Remark~\ref{rem:howtoreducestableHigsoncorona} in mind).
Recall that the mapping cone of a $*$-homomorphism $f\colon A \to B$ is defined as the \textCstar-algebra
\[
\Cone(f)\coloneqq \{(a,\beta)\in A\times \Cz((0,1],B)\mid \beta(1)=f(a)\}
\]
and if $f$ is surjective, then the inclusion $\ker(f)\to\Cone(f),a\mapsto (a,0)$ induces an isomorphism on $\K$-theory.
\begin{lem}\label{lem:reducedunreducedstablehigsoncoronacone}
Let $\coarsecomp\coloneqq\coarsecomp(X,A)$, again, denote the set of coarse components of $X$ which are disjoint from $A$.
Then there are canonical $($up to a sign$)$ natural isomorphisms
\begin{gather*}
\K_{1-*}\big(\sHigCorRed(X;D)\big)\cong\K_{-*}\big(\Cone\big(\textstyle\prod_{\coarsecomp}(D\otimes\Kom)\to\sHigCor(X;D)\big)\big),
\\
\K_{1-*}\big(\sHigCorRed(X;D)\rtimes\Gamma\big)\cong\K_{-*} \big(\Cone\big(\textstyle\prod_{\coarsecomp}(D\otimes\Kom)\to\sHigCor(X;D)\big)\rtimes\Gamma\big),
\\
\Ktop_{1-*}\big(\Gamma,\sHigCorRed(X;D)\big)\cong\Ktop_{-*} \big(\Gamma,\Cone\big(\textstyle\prod_{\coarsecomp}(D\otimes\Kom)\to\sHigCor(X;D)\big)\big)
\end{gather*}
under which the long exact sequences \eqref{eq:reducedunreducedLES} are isomorphic to the long exact sequences induced by
\[
0\to\Cz(0,1)\otimes\sHigCor(X;D)\to \Cone(\textstyle\prod_{\coarsecomp}(D\otimes\Kom)\to\sHigCor(X;D))\to \textstyle\prod_{\coarsecomp}(D\otimes\Kom)\to 0.
\]
\end{lem}

\begin{proof}
In order to reuse the exact same proof for different algebras later on (Lemma~\ref{lem:reducedKtheory}), we write $I\coloneqq\sHigCor(X;D)$, $A\coloneqq\sHigCorRed(X;D)$, $B\coloneqq\prod_{\coarsecomp}\multiplier^s(D)$ and $J\coloneqq \prod_{\coarsecomp} (D\otimes \Kom)$. The decisive properties are that $I$ is an ideal in $A$, $J$ is an ideal in $B$, $\K_*(B)=\K_*(B\rtimes \Gamma)=\Ktop_*(\Gamma,B)=0$ (because of an Eilenberg swindle argument) and there is a $*$-homomorphism $\iota\colon B\to A$ which induces an isomorphism $B/J\cong A/I$. In particular, $\iota$ restricts to $\iota'\colon J\to I$.

We furthermore let $\pi\colon B\to B/J$ be the quotient homomorphism and consider the diagram with exact rows
\[
\xymatrix@R=2ex{
0\ar[r]& \Cz(0,1)\otimes I\ar[r]\ar@{=}[d]& \Cz(0,1)\otimes A\ar[r]\ar[d]& \Cz(0,1)\otimes B/J\ar[r]\ar[d]&0
\\0\ar[r]& \Cz(0,1)\otimes I\ar[r]\ar@{=}[d]& \Cone(\iota)\ar[r]&\Cone(\pi)\ar[r]&0
\\0\ar[r]& \Cz(0,1)\otimes I\ar[r]& \Cone(\iota')\ar[u] \ar[r]& J\ar[u]\ar[r]&0.
}
\]
Both vertical arrows on the right are known to induce isomorphisms on $\K$-theory. They are the ones implementing the connecting homomorphism
$\K_{1-*}(B/J)\xrightarrow{\cong}\K_{-*}(J)$.
Therefore, the five-lemma shows that the diagram induces isomorphisms between the long exact sequences in $\K$-theory associated to the rows.
This is exactly the first isomorphism of the claim.

For the other two, we apply the functors $\blank\rtimes\Gamma$ and $(\blank\maxtensor\sfP)\maxcrossed\Gamma$ to the whole diagram and note that all the properties mentioned above still hold, because the crossed products are exact and compatible with tensor products, hence also continuous and compatible with the cone construction.
\end{proof}

And last, there is another relation between the reduced and unreduced groups which is based upon the idea that an unreduced cohomology can be obtained by taking reduced cohomology of the space augmented by a distinct basepoint (cf.\ \cite[Section 6]{WulffCoassemblyRinghomo}). But with our Definition~\ref{defn:stableHigsoncoronas} and in the light of Remark~\ref{rem:howtoreducestableHigsoncorona} we have to add not one distinct point but instead the whole set of coarse components $\coarsecomp(X)$ as isolated points, in some sense.
To this end, assume that there exists a relatively compact subspace $K\subset X$ which contains points from all connected components of $X$. Note that this might not be the case if $\Gamma$ is infinite. We define the space $X^\to\coloneqq X\cup_KK\times\N$ equipped with the obvious metric. The subspace $K\times\N$ is coarsely equivalent to $\coarsecomp(X)\times\N$, considered as the coarsely disjoint union of copies of $\N$, and each of these copies should be interpreted as a newly added point at infinity. We note that $(X,K)\subset (X^\to,K\times\N)$ is an excision, $K\times\N$ is flasque and $\sHigCor(X,K;D)=\sHigCor(X;D)$. Therefore, the long exact sequence for the pair $(X^\to,K\times\N)$ gives rise to isomorphisms
\[
\K_{1-*}(\sHigCor(X;D))=\K_{1-*}(\sHigCor(X,K;D))\cong \K_{1-*}(\sHigCor(X^\to,K\times\N;D))\cong \K_{1-*}\big(\sHigCorRed(X^\to;D)\big)
\]
and analogously for the two equivariant versions.

\subsection[K-theory of the Roe algebra]{$\boldsymbol{\K}$-theory of the Roe algebra}\label{sec:RoeAlg}

The Roe algebra is a \textCstar-algebra $\Roe(X)$ that can be constructed for every coarse space $X$ (cf.\ \cite[Section~4.4]{RoeCoarseGeometry}), but it is usually only dealt with in the case of proper metric spaces. As most of the references only consider the latter special case, we shall stick to it.

Let us briefly recall the broad outlines of its construction as presented, e.g., in \cite[Section~6.3]{HigRoe}. Starting with a representation $\rho$ of $\Cz(X)$ on a separable Hilbert space $H$ which is non-degenerate (meaning that $\rho(\Cz(X))H$ is dense in $H$), the Roe algebra $\Roe(X,\rho)$ is defined as the norm closure of the $*$-subalgebra of $\Lin(H)$ consisting of all operators $T$ satisfying the following two properties:\looseness=1
\begin{enumerate}\itemsep=0pt
\item[$\bullet$] $T$ is called \emph{locally compact}, if $\rho(f)T,T\rho(f)$ are compact operators on $H$ for all $f\in\Cz(X)$.
\item[$\bullet$] $T$ has \emph{finite propagation}, if there is $R\geq 0$ such that $\rho(f)T\rho(g)=0$ whenever the distance between the supports of $f,g\in\Cz(X)$ is at least $R$. The infimum of all $R\geq 0$ with this property is called the \emph{propagation} $\Prop(T)$ of $T$.
\end{enumerate}
This \textCstar-algebra depends on the choice of $H$ and $\rho$, but if the representation is big enough in a~cer\-tain sense (it is \enquote{ample}, i.e., it is non-degenerate and $\rho^{-1}(\Kom(H))=\{0\}$), then its $\K$-theory $\K_*(\Roe(X,\rho))$ is independent from the choices up to canonical isomorphism.
Therefore, the representation $\rho$ is usually omitted from the notation with the understanding that one suitable representation has been chosen for each $X$.

Now, there are two well-known generalizations of the Roe algebra. The first is the equivariant Roe algebra in the case that a countable discrete group $\Gamma$ acts properly and isometrically on the coarsely proper metric space $X$. Here one chooses in addition to the non-degenerate representation~$\rho$ also a representation $u$ of the group $\Gamma$ on $H$ by unitaries such that $u(\gamma)\rho(f)u(\gamma)^*=\rho(\gamma\cdot f)$ for all $\gamma\in\Gamma$ and $f\in\Cz(X)$. The triple $(H,\rho,u)$ is called a $\Gamma$-$X$-module. Then the equivariant Roe algebra $\Roe_\Gamma(X,\rho,u)$ is the norm closure of the \textCstar-algebra of all locally compact operators of finite propagation in $\Lin(H)$ that are equivariant with respect to $u$. Again, the $\K$-theory groups $\K_*(\Roe_\Gamma(X,\rho,u))$ are independent up to canonical isomorphism from the $\Gamma$-$X$-module $(H,\rho,u)$ if the latter has been chosen sufficiently large and then abbreviated by $\K_*(\Roe_\Gamma(X))$.
A detailed exposition about equivariant Roe algebras can be found in \cite{WillettYuHigherIndexTheory}.

The second generalization is to introduce a (possibly graded) coefficient \textCstar-algebra $D$ into the Roe algebra, cf.\ \cite{HankePapeSchick,HigsonPedersenRoe,WulffTwisted}. This is done by replacing the Hilbert space $H$ by a right Hilbert module $\HilbertMod$ over $D$ and the non-degenerate representation $\rho$ is by adjointable operators on $\HilbertMod$.
The $\K$-theory of the resulting \textCstar-algebra $\Roe(X,\rho,\HilbertMod;D)$ is again unique up to canonical isomorphism if $\rho$ and $\HilbertMod$ have been chosen sufficiently large and in this case $\rho$ will be ommited from the notation. Here, sufficient largeness means that $\HilbertMod\cong H\otimes D$ and $\rho=\rho'\otimes \id_D$ for an ample representation~$\rho'$ on a separable Hilbert space $H$ \cite[Definition 4.5]{HigsonPedersenRoe}.

It is possible to unify the two generalizations to equivariant Roe algebras $\Roe_\Gamma(X,\rho,u,\HilbertMod;D)$ with coefficients in a (graded) $\Gamma$-\textCstar-algebra $D$.
In this case, $u$ also has to be a representation by adjointable operators on $\HilbertMod$ which is compatible with the $\Gamma$-action on $D$ in the sense that $u(\gamma)(\xi d)=(u(\gamma)\xi) (\gamma d)$ for all $\xi\in\HilbertMod$, $d\in D$ and $\gamma\in\Gamma$.
Again, we obtain unambiguous $\K$-theory groups $\K_*(\Roe_\Gamma(X;D))$ if we start with sufficiently large $H,\rho,u$ as in the equivariant case and then simply use their tensor product with $D$, i.e., $\HilbertMod\coloneqq H\otimes D$ with obvious induced representations of $\Cz(X)$ and $\Gamma$.

The proofs of the following statements can be found in the above-mentioned sources in the special cases just described, but the proofs generalize to the general case of equivariant Roe-algebras with coefficients.

If $A\subset X$ is a $\Gamma$-invariant closed subspace of the proper metric $\Gamma$-space $X$, then $\Roe_\Gamma(X;D)$ contains an ideal $\Roe_\Gamma(A\subset X;D)$ whose $\K$-theory is canonically isomorphic to the $\K$-theory of $\Roe_\Gamma(A;D)$.
It is the norm closure of all $T\in\Roe_\Gamma(X;D)$ which are supported in an $R$-neighborhood of $A$ for some $R\geq 0$, that is, all those $T$ for which $\rho(f)T=T\rho(f)=0$ for all $f\in\Cz(X)$ with $\dist(\supp(f),A)\geq R$.
The quotient \textCstar-algebra $\Roe_\Gamma(X,A;D)\coloneqq \Roe_\Gamma(X;D)/\Roe_\Gamma(A\subset X;D)$ is called the \emph{$\Gamma$-equivariant relative Roe algebra}. We have $\Roe_\Gamma(\varnothing\subset X;D)=0$ and hence $\Roe_\Gamma(X,\varnothing;D)=\Roe_\Gamma(X;D)$. The long exact sequences
\[
\dots\to\K_p(\Roe_\Gamma(A;D))\to \K_p(\Roe_\Gamma(X;D))\to \K_p(\Roe_\Gamma(X,A;D))\to \K_{p-1}(\Roe_\Gamma(A;D))\to \cdots
\]
are immediate.

The $\K$-theory of the $\Gamma$-equivariant relative Roe algebra with coefficients in $D$ is functorial under $\Gamma$-equivariant coarse maps between pairs of proper metric $\Gamma$-spaces and the non-boundary maps in the long exact sequence above are special cases of this functoriality under the inclusion maps $(A,\varnothing)\to(X,\varnothing)\to (X,A)$.

For excision, we note that for an excision $(A\setminus C,X\setminus C)\to (X,A)$ the left square in
\[
\xymatrix{
0\ar[r]&\Roe_\Gamma(A\setminus C;D)\ar[r]\ar[d]&\Roe_\Gamma(X\setminus C;D)\ar[r]\ar[d]&\Roe_\Gamma(X\setminus C,A\setminus C;D)\ar[r]\ar[d]&0
\\0\ar[r]&\Roe_\Gamma(A;D)\ar[r]&\Roe_\Gamma(X;D)\ar[r]&\Roe_\Gamma(X,A;D)\ar[r]&0
}
\]
is simultaneously a pull-back and a push-out diagram (cf.\ \cite[p.~156]{HigRoe} for the non-equivariant case without coefficients, but the general case is proven in exactly the same way) and therefore the excision isomorphism already holds at the level of $\Gamma$-\textCstar-algebras.

Last but not least, the $\K$-theory of the Roe algebra vanishes on flasque spaces by an Eilenberg swindle (cf.\ \cite[Proposition 9.4]{RoeITCGTM}, which is also readily generalized to the equivariant case with coefficients) and therefore is also satisfies the homotopy axiom.
Even more, the methods developed by Higson and Roe in \cite{HigsonRoeHomotopy} can be adapted to show the following.
\begin{lem}\label{lem:Roestronghomotopy}
The $\K$-theory of the Roe algebra satisfies the strong homotopy axiom.
\end{lem}
\begin{proof}
Let $H\colon (X\times [0,1],\coarseStr_\cU)\to (Y,d_Y)$ be a $\Gamma$-equivariant generalized coarse homotopy which maps $A\times[0,1]$ into $B$.
In order to stay within the world of proper metric $\Gamma$-spaces, we may use the metric defined as follows instead of the coarse structure $\coarseStr_\cU$.
Without loss of generality we can assume that the metric $d_X$ on $X$ is discrete and make a $\Gamma$-invariant choice of $c_x\geq 1$ for all $x\in X$ such that $|s-t|\leq \frac{1}{c_x}\implies (s,t)\in U_x$.
The largest metric $d$ on $X\times [0,1]$ with the property
\[
\forall x,y\in X,s,t\in[0,1]\colon\ d((x,s),(y,t))\leq d_X(x,y)+\min\{c_x,c_y\}\cdot |s-t|
\]
is proper and $\Gamma$-invariant and $H$ is still coarse as a map $(X\times [0,1],d)\to (Y,d_Y)$.

Let $e_t\colon (X,A)\to (X,A)\times[0,1]$, $x\mapsto (x,t)$ denote the canonical inclusions. It is then sufficient to prove that $e_0$ and $e_1$ induce the same map
\[
(e_0)_*=(e_1)_*\colon\ \K_*(\Roe_\Gamma(X,A;D))\to\K_*(\Roe_\Gamma((X,A)\times[0,1];D)),
\]
where the right hand side is defined using the above metric $d$ on $X\times[0,1]$.
In the non-equivariant, non-relative case without coefficients ($\Gamma=1$, $A=\varnothing$, $D=\C$) this has been proved in \cite[Sections~5 and~6]{HigsonRoeHomotopy}, albeit with a slightly different metric on $X\times [0,1]$.
The proof relies heavily on calculations with idempotents, which is not adequate for the relative case $A\not=\varnothing$ and for non-unital coefficient \textCstar-algebras $D$. Therefore, we use the same tools but apply them a bit differently to obtain the general result.

Recall the following technique to define elements of the $\K$-theory of a \textCstar-algebra $C$ from \cite[Section 4]{HigsonRoeHomotopy}:
Let $\multiplier(C)$ denote its multiplier algebra and define $\doubleplier(C)\coloneqq \{(a,b)\in\multiplier(C)\mid a-b\allowbreak\in C\}$.
Then there is a split short exact sequence
\[
\xymatrix{0\ar[r]&\K_*(C)\ar[r]^{j_*}&\K_*(\doubleplier(C))\ar[r]_{\pi_*} &\K_*(\multiplier(C))\ar[r]\ar@/_/@{-->}[l]_{s_*}&0,}
\]
where $j\colon c\mapsto (c,0)$, $\pi\colon (a,b)\mapsto b$ and $s\colon b\mapsto (b,b)$.
An element of $\K_*(C)$ can then be defined as the image of an element of $\K_*(\doubleplier(C))$ under $\Theta\coloneqq(j_*)^{-1}\circ(\id-s_*\circ\pi_*)$.

Furthermore, we recall that the asymptotic algebra of a \textCstar-algebra $C$ is defined as $\Asymp(C)\coloneqq \Cb([1,\infty);C)/\Cz([1,\infty);C)$ and an asymptotic morphism from another \textCstar-algebra $B$ to $C$ is a $*$-homomorphism $\psi\colon B\to\Asymp(C)$. As the \textCstar-algebra $\Cz([1,\infty);C)$ is contractible, such an asymptotic morphism induces a homomorphism
\[
\psi_*\colon\ \K_*(B)\xrightarrow{\Psi_*}\K_*(\Asymp(C))\cong\K_*(\Cb([1,\infty);C))\xrightarrow{(\ev{1})_*}\K_*(C).
\]

The plan is to construct an algebra homomorphism
\begin{equation}\label{eq:AsympDoublealgebrahom}
\Psi\colon\ \Ct\big(S^1\big)\otimes\Roe_\Gamma(X,A;D)\to \Asymp\big(\Ct[0,1]\otimes\Ct\big(S^1\big)\otimes\doubleplier\big(\Roe_\Gamma((X,A)\times[0,1];D))\big)\big).
\end{equation}
It will not be a $*$-homomorphism (and therefore strictly speaking not an asymptotic morphism), but as algebraic and topological $\K$-theory agree in degree $0$ we will still obtain an induced map on $\K_0$.\footnote{Alternatively one can use that topological $\K$-theory of \textCstar-algebras is a special case of topological $\K$-theory of Banach algebras \cite[Chapters 5 and~8]{BlaKK}. Since $\Psi$ will clearly be a homomorphism of Banach algebras, it will induce a homomorphism on $\K$-theory even without the $*$-property.}
Since $\K_0(\Ct(S^1)\otimes C)\cong \K_0(C)\oplus \K_1(C)$ in topological $\K$-theory for any $\textCstar$-algebra $C$, this induced map decomposes into two maps
\[
\Psi_*\colon\ \K_i(\Roe_\Gamma(X,A;D))\to\K_i(\Ct[0,1]\otimes\doubleplier(\Roe_\Gamma((X,A)\times[0,1];D))),\qquad i=0,1.
\]
Now, the evaluation $*$-homomorphisms $\ev{t}\colon\Ct[0,1]\to \C$ for $t\in[0,1]$ are all homotopic to each other and therefore all of the maps
\[
\Theta\circ(\ev{t})_*\circ\Psi_*\colon\ \K_i(\Roe_\Gamma(X,A;D))\to \K_i(\Roe_\Gamma((X,A)\times[0,1];D))
\]
are equal.
For the specific $\Psi$ which we are about to define and $t=0,1$, these maps are exactly $(e_0)_*$ and $(e_1)_*$, respectively. The claim follows.

In order to carry this out, assume that $\Roe_\Gamma(X;D)$ has been constructed on the Hilbert module $\HilbertMod\coloneqq H\otimes D$ for a sufficiently large $\Gamma$-$X$-module $(H,\rho,u)$. Furthermore, let $\rho'\colon\Ct[0,1]\to\Lin\big(L^2(\R)\big)$ be the representation, where $\phi\in\Ct[0,1]$ acts on $L^2(\R)$ by multiplication with
\[
\tilde\phi\colon\ s\mapsto
\begin{cases}\phi(0),&s\leq0,\\\phi(t),&0\leq t\leq 1,\\\phi(1),&t\geq 1.
\end{cases}
\]
Then we choose two copies of the $\Gamma$-$(X\times[0,1])$-module $\big(H\otimes L^2(\R),\rho\otimes\rho',u\otimes 1\big)$ to construct $\Roe_\Gamma(X\times [0,1];D)$, that is, this \textCstar-algebra consists of operators on $\big(H\otimes L^2(\R)\otimes D\big)^2$.

The relevant outcome of \cite[Sections 6]{HigsonRoeHomotopy} can be summarized as follows. There exists a bounded and equicontinuous family of maps $\big\{F_p\colon[0,1]\to\Lin\big(L^2(\R)\big)\big\}_{p\in(0,1]}$ such that
\begin{enumerate}\itemsep=0pt
\item[$\bullet$] the operators $F_p(s)$ are Fredholm of index one and $F_p(s)^*$ is an inverse to $F_p(s)$ modulo compact operators,
\item[$\bullet$] for fixed $s$ and varying $p$, the operators $F_p(s)$ are compact perturbations of one another,
\item[$\bullet$] the propagation of $F_p(s)$ and $F_p(s)^*$ with respect to the representation $\rho'$ is bounded by $p$,
\item[$\bullet$] $F_p(0)$ is independent of $p$, has zero propagation and decomposes as $F_p(0)=W_0\oplus\id_{L^2(0,\infty)}$, where $W_0$ is a coisometry on $L^2(-\infty,0)$ (i.e., $W_0W_0^*=\id_{L^2(-\infty,0)}$) of index one,
\item[$\bullet$] and similarily $F_p(1)$ is independent of $p$, has zero propagation and decomposes as $F_p(1)=W_1\oplus\id_{L^2(-\infty,1)}$ for an index one coisometry $W_1$ on $L^2(1,\infty)$.
\end{enumerate}

We can now adapt the constructions of \cite[Sections 5]{HigsonRoeHomotopy} to our needs. Proofs of our claims can be found in that reference.
Given $\varepsilon>0$ we can find a partition of unity $\{\sigma_{\varepsilon,i}\}_{i\in\N}$ on $X$ such that
\begin{enumerate}\itemsep=0pt
\item[$\bullet$] each $\sigma_{\varepsilon,i}$ is continuous, $\Gamma$-invariant and $\Gamma$-cocompactly supported,
\item[$\bullet$] $\sigma_{\varepsilon,i}\sigma_{\varepsilon,j}=0$ if $|i-j|\geq 2$,
\item[$\bullet$] each $\sigma_{\varepsilon,i}$ is $\varepsilon 2^{-i}$-Lipschitz. \end{enumerate}
It can be constructed, for example, by starting with an analogous but non-equivariant partition of unity $\{s_{\varepsilon,i}\}_{i\in\N}$ on $[0,\infty)$ and defining $\sigma_{\varepsilon,i}(x)\coloneqq s_{\varepsilon,i}(\dist(x,\Gamma\cdot x_0))$ for some fixed $x_0\in X$.

For each $n\in\N\setminus\{0\}$ one can now find $\varepsilon_n>0$ and $p_{n,i}\in(0,1]$ such that the $\Gamma$-equivariant operators
\[
G_n(s)\coloneqq \sum_{i=1}^\infty\rho(\sigma_{\varepsilon_n,i})\otimes F_{p_{n,i}}(s)\in\Lin\big(H\otimes L^2(\R)\big)
\]
have propagation at most $\frac1n$ for all $s\in[0,1]$.
Note that these operators are all bounded by $2\sup_{p\in(0,1]}\|F_p\|$, because the second assumption on the partition of
unity implies
\begin{align*}
\|G(v)\|&\leq\biggl\|\sum_{i\text{ even}}(\rho(\sigma_{\varepsilon_n,i})\otimes F_{p_{n,i}}(s))v\biggr\|+\biggl\|\sum_{i\text{ odd}}(\rho(\sigma_{\varepsilon_n,i})\otimes F_{p_{n,i}}(s))v\biggr\|
\\
&\leq\biggl(\sum_{i\text{ even}}\|\id_H\otimes F_{p_{n,i}}\|^2\cdot\|(\rho(\sigma_{\varepsilon_n,i})\otimes\id_{L^2(\R)}) v\|^2\biggr)^{1/2}
\\
&\quad +\biggl(\sum_{i\text{ odd}}\|\id_H\otimes F_{p_{n,i}}\|^2\cdot\|(\rho(\sigma_{\varepsilon_n,i})\otimes\id_{L^2(\R)}) v\|^2\biggr)^{1/2}
\\
&\leq 2\sup_{p\in(0,1]}\|F_p\|\cdot\|v\|
\end{align*}
for all $v\in H\otimes L^2(\R)$.
Furthermore, these operators are equicontinuous in $s$ because of the equicontinuity of the $F_p$.
By the affine linear interpolation $G_t\coloneqq (t-n)G_{n+1}+(n+1-t)G_n$ for $t\in[n,n+1]$ we extend these maps to a bounded continuous map
\[
G\colon\ [1,\infty)\to C[0,1]\otimes\Lin\big(H\otimes L^2(\R)\big),\qquad t\mapsto G_t
\]
and define the bounded continuous map
\begin{align*}
R\colon\ [1,\infty)&\to C[0,1]\otimes\Lin\big(\big(H\otimes L^2(\R)\big)^2\big),
\\[.5ex]
t&\mapsto R_t\coloneqq\begin{pmatrix}2G_t-G_tG_t^*G_t&G_tG_t^*-\id_{H\otimes L^2(\R)}\\\id_{H\otimes L^2(\R)}-G_t^*G_t&G_t^*\end{pmatrix}\!.
\end{align*}
Each $R_t$ is invertible with inverse
\[
R_t^{-1}=\begin{pmatrix}G_t^*&\id_{H\otimes L^2(\R)}-G_t^*G_t
\\[.5ex]
G_tG_t^*-\id_{H\otimes L^2(\R)}&2G_t-G_tG_t^*G_t\end{pmatrix}
\]
and we define the idempotents $Q\coloneqq\id_{H\otimes L^2(\R)}\oplus 0$ and $P_t\coloneqq R_tQR_t^{-1}$ in $C[0,1]\otimes\Lin\big(\big(H\otimes L^2(\R)\big)^2\big)$.

The proof of \cite[Lemma 5.11]{HigsonRoeHomotopy} now shows the following properties of these idempotents:
\begin{enumerate}\itemsep=0pt
\item[$\bullet$] The operators $Q\otimes\id_D$, $P_t(s)\otimes\id_D$ for $s\in[0,1],t\in[1,\infty)$ and $T\otimes \id_{(L^2(\R))^2}$ for $T\in \Roe_\Gamma(\Roe(X;D))$ are all equivariant and have finite propagation. Thus, they are multipliers of $\Roe_\Gamma(\Roe(X\times[0,1];D))$.

\item[$\bullet$] For all $s\in[0,1],t\in[1,\infty)$ and $T\in \Roe_\Gamma(\Roe(X;D))$ we have
\[
((Q-P_t(s))\otimes\id_D)\circ(T\otimes \id_{(L^2(\R))^2})\in \Roe_\Gamma(\Roe(X\times[0,1];D)).
\]
The above-mentioned proof goes directly through for $t\in\N\setminus\{0\}$. For all other $t$ we additonally need to exploit the second of the above-mentioned properties of the $F_p$, i.e., that for fixed $s$ and varying $p$ the $F_p(s)$ are compact perturbations of one another.

\item[$\bullet$] For all $e\in \Ct\big(S^1\big)\otimes \Roe_\Gamma(\Roe(X;D))$ the commutators
\[
\big[P_t\otimes\id_{\Ct(S^1)\otimes D},e\otimes\id_{(L^2(\R))^2}\big]
\]
converge in norm to zero as $t\to\infty$.
\end{enumerate}

These properties imply that algebra homomorphism
\[
\Psi\colon\ \Ct\big(S^1\big)\otimes\Roe_\Gamma(X;D)\to \Asymp\big(\Ct[0,1]\otimes\Ct\big(S^1\big)\otimes\doubleplier\big(\Roe_\Gamma(X\times[0,1];D)\big)\big),
\]
{\sloppy which is the non-relative version of \eqref{eq:AsympDoublealgebrahom},
can now be defined by mapping $e\in \Ct\big(S^1\big)\otimes\allowbreak \Roe_\Gamma(\Roe(X;D))$ to the class represented by the function
\[
t\mapsto \Biggl(s\mapsto \underbrace{((P_t(s)\otimes\id_D)\circ (e(s)\otimes\id_{(L^2(\R))^2}),(e(s)\otimes\id_{L^2(\R)})\oplus 0)}_{\in\doubleplier\big(\Roe_\Gamma(X\times[0,1];D))\big)}\Biggr).
\]}\noindent
It is not a $*$-homomorphism, because the $R_t$ are not unitaries and hence the $P_t$ are not self-adjoint.

To finish the proof of the absolute case, it remains to analyze what happens at $s=0$ and $s=1$.
At $s=0$ we have $G_t(0)=\id_H\otimes F_p(0)=(\id_H\otimes W_0)\oplus\id_{H\otimes L^2(0,\infty)}$ for all $t\in[1,\infty)$, $p\in(0,1]$ and a calculation yields
\begin{align*}
P_t(0)&=\begin{pmatrix}\id_{H\otimes L^2(-\infty,0)}&0\\0&\id_H\otimes(\id_{L^2(-\infty,0)}-W_0^*W_0)\end{pmatrix}\oplus\id_{(H\otimes L^2(0,\infty))^2}
\\[.5ex]
&=Q\oplus (\id_{H\otimes D}\otimes p_0)
\end{align*}
for a rank one orthogonal projection $p_0\in \Lin(L^2(\R))$ which is supported over $0\in[0,1]$ with respect to $\rho'$. It is now straightforward to verify that $\Theta\circ(\ev{0})_*\circ\Psi_*$ is the same map as the one induced by the $*$-homomorphism $\Roe_\Gamma(X,A;D)\to \Roe_\Gamma((X,A)\times[0,1];D)$, $T\mapsto T\otimes p_0$, and this is exactly $(e_0)_*$. Analogously we obtain $\Theta\circ(\ev{1})_*\circ\Psi_*=(e_1)_*$.

Since all $P_t(s)$ have propagation bounded by one, the constructions in the last two paragraphs are compatible with taking the quotients by the ideals $\Roe_\Gamma(A\subset X;D)$, $\Roe_\Gamma(A\times [0,1]\subset X\times [0,1];D)$ and then the relative case follows immediately.
\end{proof}

We summarize:
\begin{thm}
Let $\Gamma$ be a discrete group.
The $\K$-theory groups of the $\Gamma$-equivariant relative Roe algebras with coefficients in a possibly graded $\Gamma$-\textCstar-algebra constitute a $\Gamma$-equivariant coarse homology theory on the admissible category of pairs of proper metric isometric $\Gamma$-spaces which satisfies the flasqueness and the strong homotopy axiom.
\end{thm}

The Roe algebra can also be modified as follows such that the resulting coarse homology theory satisfies the coronality instead of the flasqueness axiom.
The compact operators in $\Roe_\Gamma(A\subset X;D)$ form an ideal $\Kom_\Gamma(A\subset X;D)\coloneqq\Roe_\Gamma(A\subset X;D)\cap\Kom(\HilbertMod)$ which is $0$ if $A=\varnothing$ and equal to the \textCstar-algebra $\Kom(\HilbertMod)^\Gamma$ of $\Gamma$-equivariant compact operators on $\HilbertMod$ if $A\not=\varnothing$ and $X$ is coarsely connected, but $\Kom_\Gamma(A\subset X;D)\subsetneqq \Kom(\HilbertMod)^\Gamma$ if $X$ is coarsely disconnected. We also define $\Kom_\Gamma(X;D)\coloneqq\Kom_\Gamma(X\subset X;D)=\Roe_\Gamma(X;D)\cap\Kom(\HilbertMod)\subset \Roe_\Gamma(X;D)$.
We can now define the \emph{coronal relative $\Gamma$-equivariant Roe algebras with coefficients in $D$} as
\[
\Roe_{/\Kom;\Gamma}(X,A;D)\coloneqq\frac{\Roe_\Gamma(X;D)/\Kom_\Gamma(X;D)}{\Roe_\Gamma(A\subset X;D)/\Kom_\Gamma(A\subset X;D)}\cong \frac{\Roe_\Gamma(X;D)}{\Roe_\Gamma(A\subset X;D)+\Kom_\Gamma(X;D)}
\]
and we note that $\Roe_{/\Kom;\Gamma}(X,\varnothing;D)=\Roe_\Gamma(X;D)/\Kom_\Gamma(X;D)$.
These \textCstar-algebras vanish if $X$ is relatively compact, because in this case $\Roe_\Gamma(X;D)=\Kom_\Gamma(X;D)$.
Note also that we have $\Roe_{/\Kom;\Gamma}(X,A;D)\cong\Roe_\Gamma(X,A;D)$ if $A$ contains points from all coarse components of $X$.

\begin{thm}
The functor $\K_*(\Roe_{/\Kom;\Gamma}(\blank,\blank;D))$ is a $\Gamma$-equivariant coarse homology theory on the admissible category of pairs of proper metric isometric $\Gamma$-spaces satisfying the coronality and the strong homotopy axiom.
\end{thm}

\begin{proof}
Functoriality and independence of $F(\blank,\blank)\coloneqq \K_*(\Roe_{/\Kom;\Gamma}(\blank,\blank;D))$ from the chosen representation follows just as for $G(\blank,\blank)\coloneqq \K_*(\Roe_\Gamma(\blank,\blank;D))$, because the constructions on the level of \textCstar-algebras pass to the quotients.
We know that the excision isomorphism for $G$ holds already at the level of \textCstar-algebras and from here one readily deduces that the same is true for $F$.
The long exact sequence and coronality are clear from the construction.
Finally, strong homotopy invariance can be proven just like Lemma~\ref{lem:Roestronghomotopy} simply by taking another quotient by the ideal compact operators in the last part of the proof.
\end{proof}

\section{From topological to coarse (co-)homology theories}\label{sec:Sigmastuff}

Let $\Gamma$ be a fixed discrete group and let $\CGCBGz$ be the category of pairs of countably gene\-ra\-ted proper isocoarse $\Gamma$-spaces (see Definition~\ref{def:discretegroupactingoncoarsespace}) of bornologically bounded geometry (see Definition~\ref{def:discreteboundedgeometry}) and $\Gamma$-equivariant coarse maps. It is an admissible category in the sense of Definition~\ref{def:admissibleCAT}.

The purpose of this section is to construct coarse \mbox{(co-)}homology theories on $\CGCBGz$ from generalized \mbox{(co-)}homology theories on so-called $\sigma$-locally compact $\Gamma$-spaces by means of a Rips-complex construction. Coarse \mbox{(co-)}homology theories constructed in this way always satisfy the strong homotopy axiom.
For a very useful special type of \mbox{(co-)}homology theories, the so-called single-space \mbox{(co-)}homology theories which also satisfy the so-called strong excision axiom, the resulting coarse \mbox{(co-)}homology theories also satisfy the flasqueness axiom and can be related to the \mbox{(co-)}homology theories of certain coronas via the so-called transgression maps.

The idea of constructing coarse \mbox{(co-)}homology theories in this way goes back to the definition of coarse $\K$-theory in \cite{EmeMeyDualizing} and was generalized to non-equivariant single-space cohomology theories in \cite[Section 4]{EngelWulff}. New in our present discussion are the group actions and that we take a much deeper look at the details.

A different approach to constructing coarse homology theories is via anti-\v{C}ech systems, see \cite[Section~2]{HigsonRoeOberwolfach}, \cite[Definition~3.5]{MitchenerCoarse} and
\cite[Section 5.5]{RoeCoarseGeometry}. Unfortunately, it seems to be unsuitable for cohomology theories, because there instead of simply taking the direct limits one has to find groups satisfying $\varinjlim^1$-sequences, cf.\ \cite[Chapter~3]{RoeCoarseCohomIndexTheory}.

\subsection[Categories of sigma-spaces]{Categories of $\boldsymbol{\sigma}$-spaces}

In this subsection we also allow arbitrary locally compact groups $\Gamma$, but we will return to discrete ones in the following two subsections.

\begin{defn}[{cf.~\cite[Section~2]{EmeMeyDualizing}, \cite[Definition~3.1]{WulffCoassemblyRinghomo}}] Let $\Gamma$ be a locally compact group.
A~\emph{$\sigma$-locally compact $\Gamma$-space} $\cX$ is an increasing sequence $X_0\subset X_1\subset X_2\subset X_3\subset\cdots$ of locally compact Hausdorff spaces equipped with continuous $\Gamma$-actions such that for all $m \leq n$ the space~$X_m$ is closed in $X_n$ and carries the subspace topology and the restricted $\Gamma$-action.
We~call~$\cX$ a~\emph{$\sigma$-compact $\Gamma$-space}, if in addition all $X_n$ are compact.\footnote{Note that this is not the same as what is usually known as $\sigma$-compact spaces in the literature.}
\end{defn}

By an abuse of notation we will use the symbol $\cX$ for the set $\cX = \bigcup_{n \in \N} X_n$ endowed with the induced action and the final topology, i.e., a subset $\cA \subset \cX$ is open/closed if and only if every intersection $A_n := \cA \cap X_n$ is open/closed.
The final topology has the property that a map from~$\cX$ into another topological space is continuous if and only if it is continuous on every $X_n$.

Note also that the final topology on $\cX$ is Hausdorff itself: If $x,y\in \cX$ are two distinct points, say $x,y\in X_n$, then there is a function $f\in \Cz(X)$ with $f(x)=0$ and $f(y)=1$. This function can be extended inductively to all $X_m$ with $m\geq n$, because the one-point compactification $X_m^+$ is a~closed subspace of the normal Hausdorff space $X_{m+1}^+$. Thus, one obtains a continuous extension $F\colon\cX\to \C$ with $F(x)=0$ and $F(y)=1$.

Any locally compact Hausdorff $\Gamma$-space $X$ can be considered also as a $\sigma$-locally compact $\Gamma$-space by assigning to it the constant sequence $X_n\coloneqq X$.

\begin{defn}\label{def:subspace}
We say that a $\sigma$-locally compact $\Gamma$-space $\cX=\bigcup_{n\in\N}X_n$ is a \emph{subspace} of $\cY=\bigcup_{n\in\N}Y_n$ if $\cX\subset \cY$ is $\Gamma$-invariant, $\cX\cap Y_n = X_n$ and each $X_n$ carries the subspace topology of $Y_n$ and the restricted $\Gamma$-action.
\end{defn}

\begin{defn}
Let $\cX$, $\cY$ be $\sigma$-locally compact $\Gamma$-spaces given by the filtrations $(X_m)_{m\in\N}$ and~$(Y_n)_{n\in\N}$, respectively, and let $f\colon\cX\to\cY$ be a map.
\begin{enumerate}\itemsep=0pt
\item[$\bullet$] We call $f$ a \emph{$\sigma$-map} if for every $m\in\N$ there exists $n\in\N$ such that $f(X_m)\subset Y_n$.
Note that such an $f$ is continuous in the final topologies if and only if all the restricted maps $f|_{X_m}\colon X_m\to Y_n$ are continuous.
\item[$\bullet$] We call a continuous $\sigma$-map $f$ \emph{proper} if the preimages of all $\sigma$-compact subspaces of~$\cY$ under~$f$ are $\sigma$-compact subspaces of $\cX$, or equivalently, if the restricted maps $f|_{X_m}$: $X_m\to Y_n$ are proper continuous maps.
\end{enumerate}
\end{defn}

Using these notions we can introduce the categories of spaces on which we will consider \mbox{(co-)}homology theories.

\begin{defn}\quad
\begin{enumerate}\itemsep=0pt
\item[$\bullet$] Let $\sCHz$ be the category whose objects are pairs $(\cX,\cA)$ of $\sigma$-compact $\Gamma$-spaces with $\cA\subset\cX$ a closed subspace\footnote{Recall that subspaces are always $\Gamma$-invariant by Definition~\ref{def:subspace}.} and whose morphisms between $(\cX,\cA)$ and $(\cY,\cB)$ are continuous $\Gamma$-equivariant $\sigma$-maps $\cX\to\cY$ which restrict to maps $\cA\to\cB$.
\item[$\bullet$] Let $\sLCHz$ be the category whose objects are pairs $(\cX,\cA)$ of $\sigma$-locally compact $\Gamma$-spaces with $\cA\subset\cX$ a closed subspace and whose morphisms between $(\cX,\cA)$ and $(\cY,\cB)$ are proper continuous $\Gamma$-equivariant $\sigma$-maps $\cX\to\cY$ which restrict to maps $\cA\to\cB$.
\item[$\bullet$] The category $\sLCHp$ has objects the $\sigma$-locally compact $\Gamma$-spaces and its morphisms between~$\cX$ and $\cY$ are proper continuous $\Gamma$-equivariant $\sigma$-maps $\cU\to\cY$, where $\cU\subset\cX$ is an open subspace. The composition of two such morphisms $\cX\supset\cU\xrightarrow{f}\cY$ and $\cY\supset\cV\xrightarrow{g}\cZ$ is defined to be the morphism $\cX\supset f^{-1}(\cV)\xrightarrow{g\circ f|_{f^{-1}(\cV)}}\cZ$.
\item[$\bullet$] The category $\sLCHzp$ is a mixture of the last two: The objects $(\cX,\cA)$ are the same as in $\sLCHz$ and the morphisms are of the form $\cX \supset \cU\to\cY$ as in $\sLCHp$ but with the additional property that they map $\cA\cap \cU$ into $\cB$. If this is the case, then the restriction $\cA\cap\cU\to\cB$ is also morphisms in $\sLCHp$.
\item[$\bullet$] The categories $\CHz$, $\LCHz$, $\LCHzp$ and $\LCHp$ are the restrictions of the above categories to pairs of compact Hausdorff spaces and locally compact Hausdorff spaces, respectively. Equivalently, one can define them by removing all $\sigma$'s in the above four parts of the definition. Everything that will be said in the remainder of this section about the above four categories is also true in complete analogy for these non-$\sigma$-counterparts and therefore we shall not mention the corresponding statements explicitly.
\item[$\bullet$] If the index $\Gamma$ is omitted from the notation, this means we do not consider group actions, i.e., $\Gamma=1$.
\end{enumerate}
\end{defn}

The definition of the morphisms in the category $\sLCHp$ may appear a bit strange at first sight, but note that they are exactly in one-to-one correspondence with the basepoint-preserving continuous $\Gamma$-equivariant $\sigma$-maps between the \emph{one-point $\sigma$-compactifications}
\[
\cX^+\coloneqq\bigcup_{n\in\N}X_n^+=\bigcup_{n\in\N}X_n\cup\{\infty\},
\]
i.e., the $\sigma$-compact spaces defined by the corresponding sequence of one-point compactifications.
In other words, the functor $\sLCHp\to\sCHz$ which maps a space $\cX$ to the pair of spaces $(\cX^+,\{\infty\})$ and a morphism $\cX\supset\cU\xrightarrow{f}\cY$ to the continuous $\sigma$-map
\[
f^+\colon\quad(\cX^+,\infty)\to(\cY^+,\infty),\qquad
x\mapsto\begin{cases}f(x),&x\in \cU,\\\infty,&x\in\cX^+\setminus\cU\end{cases}
\]
is a full and faithful functor.

Conversely, a right inverse to the above functor is the composition of the full and faithful inclusion of categories
\(\sCHz\to\sLCHz\)
with the faithful inclusion of categories
\(\sLCHz\to\sLCHzp\)
and the functor
\(\sLCHzp\to\sLCHp\)
which maps pairs of spaces $(\cX,\cA)$ to their difference $\cX\setminus\cA$ and morphisms $(\cX,\cA)\to(\cY,\cB)$ which are given by the data $\cX\supset\cU\xrightarrow{f}\cY$ to the morphism $\cX\setminus\cA\supset f^{-1}(\cY\setminus\cB)\xrightarrow{f|_{f^{-1}(\cY\setminus\cB)}}\cY\setminus\cB$.

\begin{rem}\label{rem_morphconstr}
The most convenient way to construct a morphism $\cU\to\cV$ in the category $\sLCHp$ is by providing a morphism $(\cX,\cA)\to(\cY,\cB)$ with $\cX\setminus\cA=\cU$ and $\cY\setminus\cB=\cV$ in the category $\sLCHz$ and applying the above construction, i.e., applying the last two of those three functors.
\end{rem}

Now that we have explained our categories of spaces, we can go on to product spaces and homotopies.

\begin{defn}
Let $\cX$ be a $\sigma$-locally compact $\Gamma$-space and $\cY$ a $\sigma$-locally compact $\Gamma'$-spaces given by the filtrations $(X_m)_{m\in\N}$ and $(Y_n)_{n\in\N}$, respectively. Their \emph{product} $\cX\times\cY$ is the $\sigma$-locally compact $\Gamma\times\Gamma'$-space defined by the sequence of locally compact Hausdorff $\Gamma\times\Gamma'$-spaces $(X_n\times Y_n)_{n\in\N}$.
\end{defn}
Again, this is not a cartesian product in the category of $\sigma$-locally compact Hausdorff spaces.

Note that the union $\bigcup_{n\in\N}X_n\times Y_n$ equipped with the final topology is homeomorphic to the product of the spaces $\bigcup_{n\in\N}X_n$ and $\bigcup_{n\in\N}Y_n$ and hence the interpretation of the expression $\cX\times\cY$ as a topological space is unambiguous.

The following two lemmas and the corollary are obvious. The first lemma says that the cross product is functorial in both variables, the second treats inclusions of one of the factors as slices and projections onto one of the factors, and the corollary explains adequate notions of homotopy.

\begin{lem}
The assignment
\[
\big((\cX,\cA),\,(\cY,\cB)\big)\mapsto (\cX\times\cY,\cA\times\cY\cup\cX\times\cB)
\]
gives rise to functors
\begin{gather*}
\sCHz\times\sCHz[\Gamma']\to \sCHz[\Gamma\times\Gamma'],
\\
\sLCHz\times\sLCHz[\Gamma']\to \sLCHz[\Gamma\times\Gamma'],
\\
\sLCHzp\times\sLCHzp[\Gamma']\to \sLCHzp[\Gamma\times\Gamma']
\end{gather*}
and the assignment $(\cX,\cY)\mapsto \cX\times \cY$ gives rise to a functor
\begin{gather*}
\sLCHp\times\sLCHp[\Gamma']\to \sLCHp[\Gamma\times\Gamma'].
\end{gather*}
In each case, the morphisms are mapped in the obvious way.
They are compatible with each other in the sense that they commute with the three functors mentioned in the paragraph before Remark~$\ref{rem_morphconstr}$.
\end{lem}

\begin{lem}
Fixing $(\cY,\cB)$ $($or only $\cY)$ in the above lemma gives rise to functors
\begin{gather*}
-\times(\cY,\cB)\colon\ \sCHz\to \sCHz[\Gamma\times\Gamma'],
\\
-\times(\cY,\cB)\colon\ \sLCHz\to \sLCHz[\Gamma\times\Gamma'],
\\
-\times(\cY,\cB)\colon\ \sLCHzp\to \sLCHzp[\Gamma\times\Gamma'],
\\
-\times\cY\colon\qquad\, \sLCHp\to \sLCHp[\Gamma\times\Gamma']
\end{gather*}
and for each $y\in \cY$ the maps $\cX\to \cX\times\{y\}\subset\cX\times\cY$ are proper continuous $\sigma$-maps and constitute a natural transformation from the respective identity functors to each of the above functors.

Conversely, if $\cY$ is $\sigma$-compact and $\cB=\varnothing$, then the projection maps $\cX\times\cY\to \cX$ are proper continuous $\sigma$-maps and constitute a natural transformation from $-\times\cY$ to the identity functors in each of these categories and this natural transformation is left inverse to the abovementioned ones.
\end{lem}

\begin{cor}
For $\cY=I\coloneqq[0,1]$ the unit interval with trivial $\Gamma'=1$-action and $\cB=\varnothing$, the above natural tranformations give rise to an adequate notion of homotopy in each of the four categories $\sCHz$, $\sLCHz$, $\sLCHzp$ and $\sLCHp$.
\end{cor}

\subsection{Generalized (co-)homology theories and the Rips complex}\label{sec:Ripscomplex}

From all of the categories of spaces that we have introduced in the last section, $\CHz$, $\LCHz$, $\sCHz$ and $\sLCHz$ are almost admissible categories in the sense of Eilenberg and Steenrod \cite[Section~I.1]{EilenbergSteenrod}, except for the slight difference that the objects are not just topological spaces but contain a little bit of additional information in the form of the group action and possibly the filtration by closed subspaces.

\begin{defn}
By a \mbox{(co-)}homology theory (or \emph{$\Gamma$-equivariant \mbox{$($co-$)$}homology theory}) on $\CHz$, $\LCHz$, $\sCHz$ or $\sLCHz$ we understand a collection of
 covariant (contravariant) functors from this category to the category of $\Z$-graded abelian groups together with natural transformations, the (co-)boundary maps, just as in the axiomatic set-up by Eilenberg and Steenrod,
which satisfies the homotopy, exactness and excision axioms, but not necessarily the dimension or the additivity axioms, cf.\ \cite[Sections I.3 and~I.3c]{EilenbergSteenrod}.
\end{defn}

In the following, we will always be interested in such a homology theory $\Homol^\Gamma_*$ or such a~coho\-mology theory $\Cohom_\Gamma^*$ on the category $\sLCHz$. The purpose of this section is to show that applying them to the Rips complex yields coarse \mbox{(co-)}homology theories on $\CGCBGz$ which satisfy the strong homotopy axiom. From now on, $\Gamma$ is a discrete group.

We first construct the Rips complex for locally finite countably generated proper isocoarse $\Gamma$-spaces. If $E$ is a $\Gamma$-invariant entourage, then we define $P_E(X)$ to be the geometric realization of the simplicial complex with vertex set $X$ and with $x_0,\dots,x_n\in X$ spanning an $n$-simplex iff $(x_i,x_j)\in E$ for all $0\leq i,j\leq n$.
Note that we could equally well consider only symmetric entourages for this construction, because $P_E(X)=P_{E\cap E^{-1}}(X)$, but there is no need at all for doing so.
The space $P_E(X)$ is locally compact, because $X$ is locally finite, and it is a proper $\Gamma$-space, because $E$ is $\Gamma$-invariant and the action of $\Gamma$ on $X$ is proper.

Let $S\coloneqq\{E_n\mid n\in\N\}$ be a countable generating set of the coarse structure. We may assume $\Diag[X]=E_0\subset E_1\subset\cdots$ and that each entourage $E$ is contained in one of the $E_n$, because otherwise we simply replace $E_n$ by the finite union of all the entourages that can be obtained from $E_1,\dots,E_n$ by up to $n$ operations of the form (3)--(5)\ of Definition~\ref{def:coarsestructure}. Furthermore, we may assume that all $E_n$ are $\Gamma$-equivariant, because the $\Gamma$-action is isocoarse.
\begin{defn}
The \emph{Rips complex} of $X$ with respect to the generating set $S$ is the $\sigma$-locally compact $\Gamma$-space $\cP(X,S)\coloneqq\bigcup_{m\in\N} P_m(X)$, where we have abbreviated $P_m(X)\coloneqq P_{E_m}(X)$.
\end{defn}

Note that as a topological space, the Rips complex is simply the full simplicial complex of $X$, but as a $\sigma$-space it depends on the choice of the generating set $S$, because the sequence of the $P_m(X)$ is an essential part of the data.
Nevertheless, the following lemma justifies to drop the generating set from the notation and simply write $\cP(X)$.

\begin{lem}
For two different choices of generating sets $S$, $S'$ as above, the resulting Rips complexes are canonically isomorphic in the category $\sLCH$ of $\sigma$-locally compact $\Gamma$-spaces and proper continuous $\Gamma$-equivariant maps via the identity map $\id\colon\cP(X,S)\to\cP(X,S')$.
\end{lem}
\begin{proof}
If $S'$ is given by the sequence $\Diag[X]=E'_0\subset E'_1\subset\cdots$, then for every $m\in\N$ there is $n\in\N$ such that $E_m\subset E'_n$ and hence $P_{E_m}(X)\subset P'_{E_n}(X)$. Therefore the identity is a $\sigma$-map.
\end{proof}

Given in addition a $\Gamma$-invariant subspace $A\subset X$, its Rips complex $\cP(A)$ with respect to the restricted generating set $\{E\cap(A\times A)\mid E\in S'\}$ is a closed subspace of $\cP(X)$.
If~$f\colon (X,A)\to (Y,B)$ is a $\Gamma$-equivariant coarse map between two pairs of locally finite countably generated coarse $\Gamma$-spaces, then affine linear extension defines a proper continuous $\Gamma$-invariant $\sigma$-map $\cP(f)\colon (\cP(X),\cP(A))\to(\cP(Y),\cP(B))$, i.e., a morphism in $\sLCHz$, the subset $P_E(X)$ being mapped to $P_F(Y)$ if $(f\times f)(E)\subset F$.
It is readily verified that close maps $X\to Y$ induce homotopic maps $\cP(X)\to\cP(Y)$, where the homotopy is affine linear, but we even have the following much stronger result.

\begin{lem}\label{lem:Ripshomotopy}
Let $f,g\colon (X,A)\to (Y,B)$ be $\Gamma$-equivariant coarse maps between two pairs of locally finite countably generated proper isocoarse $\Gamma$-spaces. If $f$, $g$ are $\Gamma$-equi\-variantly gene\-ralized coarsely homotopic, then their induced maps between the Rips complexes $\cP(f)$, $\cP(g)$: $(\cP(X),\cP(A))\to(\cP(Y),\cP(B))$ are $\Gamma$-equivariantly homotopic.
\end{lem}

\begin{proof}
We recall the choices that we have already made at the beginning of the proof of Lemma~\ref{lem:ordinarycoarsehomologystronghomotopy}.
Let $H\colon (X\times [a,b],\coarseStr_\cU)\to (Y,\coarseStr_Y)$ be a $\Gamma$-equivariant generalized coarse homotopy between $f$ and $g$ which maps $A\times[a,b]$ into $B$.
For each point $x\in X$ we choose $a=s_{x,0}<s_{x,1}<\dots<s_{x,k_x}=b$ such that $(s,s_{x,j})\in U_x$ and $(s,s_{x,j+1})\in U_x$ for all $s\in[s_{x,j},s_{x,j+1}]$, $j=0,\dots,k_x-1$ and we can assume that $s_{x,j}=s_{x\gamma,j}$ for all $\gamma\in \Gamma$, $x\in X$ and $j=0,\dots,k_x-1=k_{x\gamma}-1$.
Now, if $x_0,\dots,x_n$ span an $n$-simplex in $P_E(X)$ for some entourage $E\in\coarseStr_X$ and if $s\in\bigcap_{i=0}^n[s_{x_i,j_i},s_{x_i,j_i+1}]$, then the points
\[
(x_0,s_{x_0,j_0}),\quad\dots,\quad(x_n,s_{x_n,j_n}),\quad(x_0,s_{x_0,j_0+1}),\quad\dots,\quad(x_n,s_{x_n,j_n+1})
\]
span an $2n$-simplex in $P_{(E\swapcross\Diag[{[a,b]}])\circ E_\cU}(X)$.
Due to this property the map $\tilde H\colon\cP(X)\times[a,b]\to\cP(Y)$ we are now about to define will obviously be a $\sigma$-map.

We start the construction of $\tilde H$ by demanding that $\tilde H$ shall agree with $H$ on the subset $\{(x,s_{x,j})\mid x\in X\wedge j=0,\dots,k_x\}$. Then we extend it affine linearly to a map on $X\times[a,b]$, that~is
\[
\tilde H(x,(1-r)s_{x,j}+rs_{x,j+1})\coloneqq (1-r)\cdot H(x,s_{x,j})+r\cdot H(x,s_{x,j+1})
\]
for all $j=0,\dots,k_x-1$ and $r\in[0,1]$.
And finally, we extend it again affine linearly to all of $\cP(X)\times[a,b]$ by
\[
\tilde H\bigg(\sum_{x\in X} \lambda_xx,s\biggr)\coloneqq \sum_{x\in X} \lambda_x\tilde H(x,s).
\]
Continuity of $\tilde H$ is obvious and topological properness of $\tilde H$ follows directly from the coarse \mbox{geometric} properness of $H$.
It is also a $\sigma$-map, because it clearly maps $P_E(X)$ into\linebreak $P_{(H\times H)((E\swapcross\Diag[{[a,b]}])\circ E_\cU)}(Y)$, and it obviously maps $\cP(A)$ into $\cP(B)$.
Finally, $\Gamma$-equivariance follows from the $\Gamma$-invariant choice of the $s_{x,j}$.
\end{proof}

We are now going to apply these constructions to discretizations of pairs $(X,A)$ of countably generated proper isocoarse $\Gamma$-spaces of bornologically bounded geometry, i.e., objects in $\CGCBGz$. According to Lemma~\ref{lem:goodequivariantdiscretization} there are $\Gamma$-invariant discretizations $\iota_{X}\colon X'\subset X$ and $A'\subset A$ such that there are $\Gamma$-equivariant coarse equivalences $\pi_{X}\colon X\to X'$ and $\pi_{A}\colon A\to A'$ which are the identities on $X'$ and $A'$, respectively.
By performing the obvious modifications of $X'$ and $\pi_{X}$, if necessary, we may assume that $A'\subset X'$ and $\pi_{A}=\pi_{X}|_A$. Then the inclusion $\iota_{X}\colon (X',A')\to (X,A)$ is a~$\Gamma$-equivariant coarse equivalence with $\pi_{X}$ being a $\Gamma$-equivariant coarse inverse up to closeness.

\begin{defn}
Let $\Homol_*^\Gamma$ be a generalized $\Gamma$-equivariant homology theory or $\Cohom^*_\Gamma$ a generalized $\Gamma$-equivariant cohomology theory. Then its \emph{coarsification} $\HomolX_*^\Gamma$ or $\CohomX^*_\Gamma$ is defined by the groups
\[
\HomolX_*^\Gamma(X,A)\coloneqq \Homol_*^\Gamma(\cP(X'),\cP(A'))\qquad\text{or}\qquad\CohomX^*_\Gamma(X,A)\coloneqq \Cohom^*_\Gamma(\cP(X'),\cP(A')),
\]
respectively, where $(X',A')$ denotes some fixed discretization of $(X,A)$ of the type described above.
\end{defn}

\begin{thm}\label{thm:coarsifiedtheories}
Up to canonical isomorphism, the coarsifications $\HomolX_*^\Gamma(X,A)$ and $\CohomX^*_\Gamma(X,A)$ are independent of the choice of discretization. Furthermore, they are the groups of a $\Gamma$-equivariant \mbox{$($co-$)$}homology theory that satisfies the strong homotopy axiom.
\end{thm}

\begin{proof}
If there is another pair of spaces with discretization $(Y',B')\xleftrightharpoons[\iota_Y]{\pi_Y}(Y,B)$ of the same type and if $f\colon (X,A)\to(Y,B)$ is a $\Gamma$-equivariant coarse map, then the $\Gamma$-equivariant coarse map $f'\coloneqq \pi_Y\circ f\circ\iota_X\colon(X',A')\to (Y',B')$ is up to closeness independent of the choice of $\pi_Y$ and consequently the induced map on the Rips complexes $\cP(f)\colon (\cP(X'),\cP(A'))\to (\cP(Y'),\cP(B'))$ is up to homotopy in the category $\sLCHz$ independent of the choice of $\pi_Y$ by Lemma~\ref{lem:Ripshomotopy}. Thus, we obtain an induced group homomorphism $f_*\colon\HomolX_*^\Gamma(X,A)\to\HomolX_*^\Gamma(Y,B)$ and $f^*\colon\CohomX^*_\Gamma(Y,B)\to\CohomX^*_\Gamma(X,A)$ which is independent of the choice of $\pi_Y$.

Also, if $f_1$, $f_2$ are generalized coarsely homotopic, then so are $f_1'$, $f_2'$ and therefore $\cP(f_1)$, $\cP(f_2)$ are homotopic in the category $\sLCHz$, so $(f_1)_*=(f_2)_*$ and $(f_1)^*=(f_2)^*$.

If there is a third pair of spaces with discretization $(Z',C')\xleftrightharpoons[\iota_Z]{\pi_Z}(Z,C)$ and $g\colon (Y,B)\to (Z,C)$ is a $\Gamma$-equivariant coarse map, then $g'\circ f'$ is close to $(g\circ f)'$ and hence $\cP(g')\circ \cP(f')$ is homotopic to $\cP(g'\circ f')$, so $g_*\circ f_*=(g\circ f)_*$ and $f^*\circ g^*=(g\circ f)^*$.

Applying all of this to the identity map between the same pair of spaces $(X,A)$ but equipped with different discretizations $(X',A')$, we see that the groups $\HomolX_*^\Gamma(X,A)$ and $\CohomX^*_\Gamma(X,A)$ are independent of the choice of discretization up to canonical isomorphism.

Thus, so far we have shown that the groups form a covariant or a contravariant functor from $\CGCBG^2$ to the category of abelian groups which satisfies the strong homotopy axiom. The exactness axiom for $\HomolX_*^\Gamma$, $\CohomX^*_\Gamma$ follows immediately from the exactness axiom for $\Homol_*^\Gamma$, $\Cohom^*_\Gamma$.

Finally, but most complicatedly, it remains to show the excision axiom.
Let $(X\setminus C,A\setminus C)\subset (X,A)$ be an excision (cf.\ Definition~\ref{defn:excisionmorphisms}). We can choose discretizations $C'\subset C$, $A'\subset A$ and $X'\subset X$ as above such that $C'\subset A'\subset X'$ and then $(X'\setminus C',A'\setminus C')\subset (X',A')$ is an excision, too.
Let $\alpha\colon\cP(X)\to [0,1]$ be the continuous $\Gamma$-invariant function which is $1$ on $A'$, $0$ on $X'\setminus A'$ and then extended over all simplices affine linearily. In the same way we construct $\beta\colon \cP(X)\to [0,1]$ being $1$ on $C'$ and $0$ on $X'\setminus C'$. We define the closed subspace $\cA\coloneqq\alpha^{-1}\bigl[\frac13,1\bigr]$ and the open subspace $\cC\coloneqq\beta^{-1}\bigl(\frac23,1\bigr]$ of $\cP(X')$.
Note that the closure of $\cC$ is contained in the interior of $\cA$, because $\beta\leq \alpha$, and therefore the inclusions $(\cP(X')\setminus\cC,\cA\setminus\cC)\to (\cP(X'),\cA)$ are topological excisions and thus induce isomorphisms on \mbox{(co-)}homology.

The proof can now be finished by showing that the subspaces $\cP(A')\subset\cA$, $\cP(X'\setminus C')\subset\cP(X')\setminus\cC$ and $\cP(A'\setminus C')\subset \cA\setminus \cC$ are deformation retracts, because then homotopy invariance and the long exact sequences will show that the inclusions $(\cP(X'),\cP(A'))\to(\cP(X'),\cA)$ and $(\cP(X'\setminus C'),\cP(A'\setminus C'))\to(\cP(X')\setminus\cC,\cA\setminus\cC)$ induce isomorphisms on \mbox{(co-)}homology.

Constructing a retraction $\cA\to\cP(A')$ is straightforward: A point $x\in\cA$ has a decomposition $x=\lambda a+(1-\lambda) b$ with $a\in\cP(A')$ and $b\in\cP(X'\setminus A')$ in which $\lambda\geq \frac13$. In particular, $\lambda\not=0$ and therefore $a$ is uniquely determined. Due to this uniqueness, mapping $x\mapsto a$ gives us a continuous $\Gamma$-equivariant retraction which is homotopic to the identity via the obvious affine linear homotopy. Furthermore, both the retraction and the homotopy respect the filtration, because if $x$ lies in a simplex, then $a$ and the whole path between $x$ and $a$ lie in the same simplex. Exactly the same procedure provides us with a deformation retraction $\cP(X')\setminus\cC\to\cP(X'\setminus C')$.

The tricky part is to construct the deformation retraction $\cA\setminus \cC\to\cP(A'\setminus C')$ and here we will need the excisiveness condition. Let $\Diag[X']=E_0\subset E_1\subset \cdots$ be a sequence of symmetric $\Gamma$-invariant entourages of $X'$ such that every entourage is contained in some $E_n$, and set $E_{-1}=\varnothing$.
We choose a point $x_O\in O$ in each $\Gamma$-orbit $O$ of $X'$, $O=x_O\Gamma$, and let $\Gamma_{x_O}$ denote the stabilizer of $x_O$. Now, there is a unique $n_O\in\N$ such that $x_O\in \Pen_{E_{n}}(A'\setminus C')\setminus\Pen_{E_{n-1}}(A'\setminus C')$ and therefore in particular $O=x_O\Gamma\subset \Pen_{E_{n}}(A'\setminus C')$.
Choose $y_O\in A'\setminus C'$ such that $(x_O,y_O)\in E_n$. Note that for all $\gamma\in\Gamma,\gamma'\in\Gamma_{x_O}$ we have $(x_O\gamma,y_O\gamma'\gamma)\in E_n$ by $\Gamma$-invariance of $E_n$. Also, for $O\subset X'\setminus C', x_O\in X'\setminus C'$, we have $y_O= x_O$.

We can now define a map $r\colon\cP(X')\to\cP(A'\setminus C')$ as the affine linear extension of the $\Gamma$-equivariant extension of the mapping $x_O\mapsto \sum_{\gamma'\in\Gamma_{x_O}}y_O\gamma'$, that is
\[
r\colon\ \sum_{O\subset X',\,\gamma\in\Gamma}\lambda_{O,\gamma}x_O\gamma\mapsto \sum_{O\subset X',\,\gamma\in\Gamma,\,\gamma'\in\Gamma_{x_O}}\lambda_{O,\gamma}y_O\gamma'\gamma.
\]
This map is clearly well defined, $\Gamma$-equivariant, continuous and restricts to the identity on $\cP(A'\setminus C')$.

It is not a $\sigma$-map itself, but its restriction to $\cA\setminus\cC$ is: If $x\in\cA\setminus\cC$, then each simplex of $\cP(X')$ which contains $x$ must have at least one vertex in $A'$ \big(because $\alpha(x)\geq\frac13$\big) and one vertex in $X'\setminus C$ \big(because $\beta(x)\leq\frac23$\big). Thus, if this particular simplex lies in $P_E(X')$, then all of its vertices lie in $\Pen_E(A')\cup\Pen_E(X'\setminus C')\subset \Pen_{E_n}(A'\setminus C')$, where $n\in\N$ only depends on $E$. This shows $(\cA\setminus\cC)\cap P_E(X')\subset P_E(\Pen_{E_n}(A'\setminus C'))$ and $r$ clearly maps this subset into $P_{E_n\circ E\circ E_n}(A'\setminus C')$.

The same applies to the homotopy between $r$ and the identity on $\cA\setminus\cC$ given by affine linear interpolation. Therefore, $\cP(A'\setminus C')\subset \cA\setminus \cC$ is a deformation retract.
\end{proof}

\subsection{Single space (co-)homology theories and transgression maps}
\label{sec:SteenrodTransgression}

An important class of coarse \mbox{(co-)}homology theories are the coarsifications of so-called single-space \mbox{(co-)}homology theories. We will show in this section that they always satisfy the flasqueness axiom and admit the so-called transgression maps.

\begin{defn}
A \mbox{(co-)}homology theory on $\sCHz$ or $\sLCHz$ is called a \emph{single-space \mbox{$($co-$)$}ho\-mo\-lo\-gy theory} (or \emph{$\Gamma$-equivariant single space $($co-$)$ho\-mo\-lo\-gy theory}, if we want to emphasize the group action), if it factors through $\sLCHp$ up to natural isomorphism. This property is called the \emph{strong excision axiom}.
Equivalently, these \mbox{(co-)}homology theories can be axiomatized using solely the absolute \mbox{(co-)}homology groups\footnote{Therefore the name.} as follows:
\begin{enumerate}\itemsep=0pt
\item[$\bullet$] A \emph{$\Gamma$-equivariant single-space homology theory $\Homol^\Gamma_*$ for $\sigma$-locally compact spaces} is a collection of covariant homotopy functors $\Homol^\Gamma_n$ indexed by $n\in\Z$ from the category $\sLCHp$ to the category of abelian groups together with \emph{boundary maps} $\partial_n\colon \Homol^\Gamma_n(\cX\setminus\cA)\to\Homol^\Gamma_{n-1}(\cA)$ such that the long sequences
\begin{equation}\label{eq:ssHomolLES}
\dots\to\Homol^\Gamma_n(\cA)\to\Homol^\Gamma_n(\cX)\to\Homol^\Gamma_n (\cX\setminus\cA)\xrightarrow{\partial_n}\Homol^\Gamma_{n-1}(\cA)\to\cdots
\end{equation}
are exact and natural in $(\cX,\cA)\in\sLCHzp$.
\item[$\bullet$] Dually, a \emph{$\Gamma$-equivariant single-space cohomology theory $\Homol^\Gamma_*$ for $\sigma$-locally compact spaces} is a collection of contravariant homotopy functors $\Cohom_\Gamma^n$ indexed by $n\in\Z$ from the category $\sLCHp$ to the category of abelian groups together with \emph{coboundary maps} $\delta^n\colon \Cohom_\Gamma^n(\cA)\to\Cohom_\Gamma^{n+1}(\cX\setminus\cA)$ such that the long sequences
\begin{equation}\label{eq:ssCohomLES}
\dots\to\Cohom_\Gamma^n(\cX\setminus\cA)\to\Cohom_\Gamma^n(\cX)\to\Cohom_\Gamma^n(\cA)\xrightarrow{\delta_n}\Cohom_\Gamma^{n+1}(\cX\setminus\cA)\to\cdots
\end{equation}
are exact and natural in $(\cX,\cA)\in\sLCHzp$.
\end{enumerate}
In the same way, $\Gamma$-equivariant single-space \mbox{(co-)}homology theories can be introduced on $\CHz$, $\LCHz$ and $\LCHp$.
\end{defn}

This definition has already been well-established for a long time on the categories $\CHz[]$, $\LCHz[]$, $\LCHp[]$. In \cite[Definition 4.13]{EngelWulff} we gave $\sigma$-versions for $\sCHz[],\sLCHp[]$ under the name ``generalized Steenrod \mbox{(co-)}homology theories'', where ``Steenrod'' was meant to express that the strong excision axiom holds. However, we drop this naming in the present paper to avoid confusion, because other authors use it as a qualifier for the cluster axiom, which says that homology turns disjoint unions of spaces into products of homology groups and co-homology turns disjoint unions into direct sums (cf.\ \cite[Definition B.2.2]{WillettYuHigherIndexTheory}).
We will not use the cluster axiom here, but it should be noted that it might be very useful for the calculation of \mbox{(co-)}homology of the Rips complexes by cellular methods.

For the remainder of this section, $\Homol^\Gamma_*$ will always denote a $\Gamma$-equivariant single-space homology theory and $\Cohom_\Gamma^*$ will always denote a generalized $\Gamma$-equivariant single-space cohomology theory for $\sigma$-locally compact spaces and their coarsifications will be denoted by $\HomolX^\Gamma_*$ and $\CohomX_\Gamma^*$, respectively, as before.

\begin{lem}
\label{lem:SteenrodFlasqueness}
Coarsifications of single-space \mbox{$($co-$)$}homology theories satisfy the flasqueness axiom.
\end{lem}

\begin{proof}
Let $X$ be a flasque space, witnessed by $\phi\colon X\to X$, and let $\iota\colon X'\subset X$ be a discretization as in Lemma~\ref{lem:goodequivariantdiscretization} with $\Gamma$-equivariant coarse equivalence $\pi\colon X\to X'$. Then we define the map $(\cP(X'))^+\times[0,\infty]\to (\cP(X'))^+$ as being equal to $\cP(\pi\circ\phi^n\circ\iota_X)$ on $\cP(X')\times\{n\}$, interpolating affine linearily inbetween and mapping everything else (i.e., $\{\infty\}\times [0,\infty]\cup\cP(X')\cup\{\infty\}$) to~$\infty$. From the assumptions on $\phi$, it is straightforward to see that this is a $\Gamma$-equivariant proper continous $\sigma$-map. Furthermore, it is readily checked that it descends to a homotopy in $\sLCHp$ which shows that $\cP(X')$ is homotopy equivalent to the empty set. The long exact sequence shows that \mbox{(co-)}homology vanishes on the empty set, and therefore its coarsification vanishes on $X$.
\end{proof}

Given a coarse space $X$, we let $\HigFuAlg(X)$ denote the commutative \textCstar-algebra of all bounded functions $X\to\C$ of vanishing variation and $\VanInfFuAlg(X)\coloneqq \VanInfFuAlg(X;\C)$ the ideal of all functions vanishing at infinity. The \emph{Higson corona} $\partial_hX$ is by definition the maximal ideal space of the quotient \textCstar-algebra $\HigFuAlg(X)/\VanInfFuAlg(X)$, i.e., the compact Hausdorff space with $\Ct(\partial_hX)\cong \HigFuAlg(X)/\VanInfFuAlg(X)$.
Note that $\HigFuAlg(X)$ and $\VanInfFuAlg(X)$ are clearly contravariantly functorial and hence $\partial_hX$ is covariantly functorial under closeness classes of coarse maps.
In particular, coarsely equivalent spaces have homeomorphic Higson coronas.
We call a compact Hausdorff space $\partial X$ a \emph{Higson dominated corona} for the coarse space $X$ if there is a continuous surjective map $\partial_hX\twoheadrightarrow \partial X$.

If $X$ is moreover a topological coarse space, then $\HigFuAlg(X)$ contains the sub-\textCstar-algebra of continuous functions $\Ch(X)=\HigFuAlg(X)\cap\Ct(X)$ which contains $\Cz(X)$. Its maximal ideal space is a compactification $\overline{X}^h$ of~$X$, called the \emph{Higson compactification}. We have $\Ct(\partial_hX)\cong \HigFuAlg(X)/\VanInfFuAlg(X)\cong \Ch(X)/\Cz(X)\cong \Ct\big(\overline{X}^h\setminus X\big)$, where the proof of the second isomorphism works exactly as the proof of~\eqref{eq:sHigCorcontinuousrepresentatives}, see \cite[proof of Proposition~3.7]{EmeMeyDualizing}.
Thus, the Higson corona is the boundary of this compactification. A compactification $\overline{X}$ of $X$ with boundary~$\partial X$ is called \emph{Higson dominated}, if there is a continuous surjection $\overline{X}^h\twoheadrightarrow\overline{X}$ which restricts to the identity on~$X$, or equivalently, if the closure of each entourage $E\subset X\times X$ intersects $\overline{X}\times\partial X\cup\partial X\times\overline{X}$ only within $\partial X\times\partial X$.
In this case $\partial X\coloneqq \overline{X}\setminus X$ is a Higson dominated corona.

If $X$ comes equipped with a coarse $\Gamma$-action, then $\partial_h X$ and $\overline{X}^h$ are $\Gamma$-spaces. In this case we will always assume that Higson dominated coronas and compactifications $\partial X,\overline{X}$ are also $\Gamma$-spaces and that the surjections are $\Gamma$-equivariant.

Now assume that $X$ is a countably generated proper isocoarse $\Gamma$-space of bornologically bounded geometry and $\partial X$ a Higson dominated corona. Let $X'\subset X$ be a $\Gamma$-equivariant discretization as in Lemma~\ref{lem:goodequivariantdiscretization}.
For each $\Gamma$-invariant entourage $E$ of $X'$, the Rips complex $P_E(X')$ becomes a topological isocoarse $\Gamma$-space if we additionally equip it with the canonical coarse structure characterized by the properties that the inclusion $X'\subset P_E(X')$ is a coarse equivalence and the simplices are uniformly bounded.
Then there is a Higson dominated compactification of $P_E(X')$ with the same corona $\partial X$, which can be constructed as follows: Note that the corona $\partial X$ corresponds to a unital sub-$\Gamma$-\textCstar-algebra of $\Ct(\partial_hX)\cong\HigFuAlg(X)/\VanInfFuAlg(X)\cong \HigFuAlg(X')/\VanInfFuAlg(X')\cong \HigFuAlg(P_E(X'))/\VanInfFuAlg(P_E(X'))\cong\Ch(P_E(X'))/\Cz(P_E(X'))$, that is, it is of the form $C/\Cz(P_E(X'))$ for a unital sub-$\Gamma$-\textCstar-algebra $C\subset \Ch(P_E(X'))$ and the maximal ideal space of~$C$ is a Higson-dominated compactification $\overline{P_E(X')}$ of $P_E(X')$ with boundary $\partial X$. Applying it to the sequence of Rips complexes $P_{E_0}(X')\subset P_{E_1}(X')\subset \cdots$ we obtain a sequence of compact Hausdorff $\Gamma$-spaces $\overline{P_{E_0}(X')}\subset\overline{P_{E_1}(X')}\subset \cdots$, where the inclusions are the identity on $\partial X$. Thus, this sequence is a $\sigma$-compactification $\overline{\cP(X')}$ of $\cP(X')$ with boundary $\partial X$.

Assume furthermore that $A\subset X$ is a coarse $\Gamma$-invariant subspace. Then the image of $\Ct(\partial X)$ under the restriction map $\Ct(\partial_h X)\cong \HigFuAlg(X)/\VanInfFuAlg(X)\to \HigFuAlg(A)/\VanInfFuAlg(A)\cong \Ct(\partial_h A)$ corresponds to a Higson dominated corona $\partial A$ of $A$. By construction we have a surjection $\Ct(\partial X)\to\Ct(\partial A)$, so $\partial A$ can be considered a subspace of $\partial X$.
If $X'\subset X$ and $A'\subset A$ are two discretizations as in Lemma~\ref{lem:goodequivariantdiscretization} with $A'\subset X'$, then we obtain corresponding compactifications $\overline{P_E(A')}$ with boundary $\partial A$ for each entourage $E$ of $X$. The function algebras $\Ct(\overline{P_E(A')})$ are exactly the images of $\Ct(\overline{P_E(X')})$ under the restriction maps $\Ch(P_E(X'))\to\Ch(P_E(A'))$, so we can consider $\overline{P_E(A')}$ as a subspace of $\overline{P_E(X')}$. Thus, the $\sigma$-compactification $\overline{\cP(A')}$ is also a subspace of $\overline{\cP(X')}$ and we note that $\overline{\cP(A')}\cap\partial X=\partial A$.

In the proof of Theorem~\ref{thm:coarsifiedtheories} we have seen that different choices of discretizations $(X'',A'')$ lead to homotopy equivalences $(\cP(X'),\cP(A'))\simeq (\cP(X''),\cP(A''))$ in the category $\sLCHz$ which are canonical up to homotopy. They clearly extend to homotopy equivalences $(\overline{\cP(X')},\overline{\cP(A')})\simeq (\overline{\cP(X'')},\overline{\cP(A'')})$ in the category $\sCHz$ which are the identity on the boundaries $(\partial X,\partial A)$ and are also canonical up to homotopy. Therefore, the next definition is independent of the choice of discretizations.

\begin{defn}
Let $(X,A)$ be a pair of countably generated proper isocoarse $\Gamma$-spaces of bornologically bounded geometry, $\partial X$ a Higson dominated corona and $\partial A$ the corresponding corona of $A$.
The \emph{transgression maps} are the connecting homomorphisms
\[
\HomolX_{*+1}^\Gamma(X,A)\to\Homol_*^\Gamma(\partial X,\partial A)\qquad\text{and}\qquad \Cohom^*_\Gamma(\partial X,\partial A)\to \CohomX^{*+1}_\Gamma(X,A)
\]
associated to the pair of $\sigma$-locally compact spaces $\big(\overline{\cP(X')}\setminus\overline{\cP(A')},\partial X\setminus\partial A\big)$.
\end{defn}

Important for applications (see, e.g., Theorems~\ref{thm:transgressionfactorsthroughcoassembly} and~\ref{thm:transgressionfactorsthroughassembly} below) is the question when transgression maps are isomorphisms. The long exact sequences and homotopy invariance immediately imply the following result.

\begin{lem}\label{lem:relativetransgressionisos}
If $\overline{\cP(A')}$ is a deformation retract of $\overline{\cP(X')}$ in the category $\sCH$, then the transgression maps are isomorphisms.
\end{lem}

\subsection{Reduced (co-)homology and reduced transgression}\label{sec:reducedtransgression}

In the absolute case $A=\varnothing$, there is a similar result to Lemma~\ref{lem:relativetransgressionisos} if we consider a reduced version of the transgression maps.
To define them, we first introduce reduced (co-)homology for $\sigma$-locally compact $\Gamma$-spaces $\cX$ equipped with a decomposition $\cX=\coprod_{i\in \pi}\cK_i$ as a coproduct in the category of $\sigma$-locally compact spaces (indexed by an arbitrary set $\pi$) with the following properties: First, we assume that each $\cK_i$ is $\sigma$-compact. And second, the subsets $\cK_i$ do not have to be closed under the $\Gamma$-action, but we do assume that each group element maps each $K_i$ homeomorphically onto some $K_j$.
In case of the trivial decomposition into a single $\sigma$-compact subset $\cK=\cX$ and without group action we will recover the usual definition of reduced (co-)homology.

For each nonempty $\sigma$-compact space $\cK$, let $\cC\cK\coloneqq \cK\times[0,1]/\cK\times\{1\}$ be the closed cone and $\cO\cK\coloneqq \cK\times(0,1]/\cK\times\{1\}$ the open cone over $\cK$.
Furthermore we define $\cC\varnothing$ and $\cO\varnothing$ to be the one-point space.
Thanks to $\sigma$-compactness of $\cK$, the closed cone is again $\sigma$-compact and the open cone is $\sigma$-locally compact.
Then we define the closed and open multicones of $\cX$ as $\cC^+\cX\coloneqq\coprod_{i\in \pi}\cC\cK_i$ and $\cO^+\cX\coloneqq\coprod_{i\in \pi}\cO\cK_i$, respectively. They are canonically $\sigma$-locally compact $\Gamma$-spaces.

\begin{defn}
The \emph{reduced} \mbox{(co-)}homology groups of a decomposed $\sigma$-locally compact $\Gamma$-space~$\cX$ as above are defined as
\[
\widetilde\Homol^\Gamma_*(\cX)\coloneqq \Homol^\Gamma_{*+1}(\cO^+\cX)\qquad \text{and}\qquad \widetilde\Cohom_\Gamma^*(\cX)\coloneqq \Cohom_\Gamma^{*+1}(\cO^+\cX).
\]
\end{defn}

To understand this definition, we furthermore define the suspension of a $\sigma$-locally compact $\Gamma$-space $\cX$ as $\Sigma\cX\coloneqq \cX\times(0,1)$.
The obvious deformation retraction of $\cX\times[0,1]$ to $\cX\times\{1\}$ induces a homotopy equivalence in $\sLCHp$ between $\cX\times[0,1)$ and the empty space and hence all single space \mbox{(co-)}homology theories vanish on them.
\begin{defn}\label{def:suspensionandreduction}
The \emph{suspension isomorphisms} are the connecting homomorphisms
\[
\Homol^\Gamma_{*+1}(\Sigma \cX)\cong\Homol^\Gamma_*(\cX)\qquad \text{and}\qquad\Homol_\Gamma^*(\cX)\cong \Cohom_\Gamma^{*+1}(\Sigma \cX)
\]
in the long exact sequences associated to the pairs of spaces $(\cX\times[0,1),\cX\times\{0\})$.
\end{defn}

Now, we can canonically consider $\pi$ itself as a discrete $\Gamma$-space. As such, it is embedded in the multicones $\cC^+\cX$, $\cO^+\cX$ as the subspace of apices. As $\cC^+\cX$ is contractible onto $\pi$ and by using the suspension isomorphisms, the long exact sequences associated to the pair $(\cC^+\cX,\cX\times\{0\})$ become
\begin{gather}
\cdots\to \Homol^\Gamma_{*+1}(\pi) \to \widetilde\Homol^\Gamma_*(\cX) \to \Homol^\Gamma_*(\cX) \to \Homol^\Gamma_*(\pi) \to\cdots,\nonumber
\\\cdots\to \Cohom_\Gamma^*(\pi) \to \Cohom_\Gamma^*(\cX) \to \widetilde\Cohom_\Gamma^*(\cX) \to \Cohom_\Gamma^{*+1}(\pi) \to\cdots.\label{eq:generalizedreducedunreducedsequence}
\end{gather}
These sequences imply in particular that reduced \mbox{(co-)}homology vanishes if the canonical $\Gamma$-equivariant continuous $\sigma$-map $\cX\to \pi$ is a homotopy equivalence.

The connection to the well-known definition of reduced (co-)homology is now as follows:
If $\cX$ contains a $\Gamma$-invariant subset which meets each $\cK_i$ in exactly one point, then the map $\cX\to \pi$ has a right inverse in $\sLCH$ and the exact sequences yield canonical isomorphisms
\[
\widetilde\Homol^\Gamma_*(\cX)\cong\ker(\Homol^\Gamma_*(\cX)\to\Homol^\Gamma_*(\pi))\qquad \text{and}\qquad \widetilde\Cohom_\Gamma^*(\cX)\cong\operatorname{coker}(\Cohom_\Gamma^*(\pi)\to\Cohom_\Gamma^*(\cX)).
\]
In the non-equivariant case and for the trivial decomposition, such that $\pi$ itself is a one-point space, these isomorphisms are exactly the classical definitions of the reduced groups via augmentation with a point.
Analogously, for non-trivial decompositions, we recognize the idea that an augmentation with more than one point has been performed.
However, we emphasize that such a right inverse $\pi\to \cX$ need not exist in the equivariant case, so our definiton via cones is superior.

Given a closed subspace $\cA\subset \cX$
equipped with the decomposition into the subsets $\cA\cap \cK_i$, we have $\cO^+\cA\subset\cO^+\cX$ and $\cO^+\cX\setminus\cO^+\cA=\Sigma(\cX\setminus\cA)$.
As $\cA\cap \cK_i$ might be empty even if $\cK_i$ is not, we point out that it is important to allow empty sets in the decomposition and to define the cones over the empty space to be the one-point space for this to work.
Then, the long exact sequences associated to the pairs $(\cO^+\cX,\cO^+\cA)$ become
\begin{gather}
\cdots\to \widetilde\Homol^\Gamma_*(\cA) \to \widetilde\Homol^\Gamma_*(\cX) \to \Homol^\Gamma_*(\cX,\cA) \to \widetilde\Homol^\Gamma_{*-1}(\cA) \to\cdots,
\label{eq:multireducedLEShomol}
\\
\cdots\to \widetilde\Cohom_\Gamma^{*-1}(\cA) \to \Cohom_\Gamma^*(\cX,\cA) \to \widetilde\Cohom_\Gamma^*(\cX) \to \widetilde\Cohom_\Gamma^*(\cA) \to\cdots.
\label{eq:multireducedLEScohom}
\end{gather}

\begin{lem}Up to signs, there is a canonical natural transformations from the reduced homological long exact sequence \eqref{eq:multireducedLEShomol} to the corresponding unreduced one \eqref{eq:ssHomolLES} and dually from the unreduced cohomological long exact sequence \eqref{eq:ssCohomLES} to the corresponding reduced one \eqref{eq:multireducedLEScohom}.
\end{lem}

\begin{proof}
{\sloppy The transformations are defined by applying naturality under the morphism $(\cO^+\cX,\cO^+\cA)\allowbreak\supset (\Sigma\cX,\Sigma\cA)$ and the suspension isomorphisms. The only nontrivial claim is that the boundary maps $\Homol^\Gamma_{*+1}(\Sigma\cX,\Sigma\cA) \to \Homol^\Gamma_{*}(\Sigma\cA)$ and $\Homol^\Gamma_*(\cX,\cA) \to \Homol^\Gamma_{*-1}(\cA)$ correspond to each other under suspension (and dually for cohomology), but this is only true up to the sign $-1$. Its proof is a~fairly standard argument from algebraic topology, although concrete references seem to be hard to find: First, one rewrites the boundary map using suspension as
\begin{align*}
\Homol^\Gamma_*(\cX,\cA)&\cong \Homol^\Gamma_*(\cX\times\{1\}\cup A\times[0,1],\cA\times\{0\})
\\
&\to \Homol^\Gamma_*(A\times (0,1],\cA\times\{0\})\cong\Homol^\Gamma_*(\Sigma A)
\cong\Homol^\Gamma_{*-1}(\cA)
\end{align*}}\noindent
and similarily for the boundary map for $(\Sigma\cX,\Sigma\cA)$.
The claim is then equivalent to the fact that the two ways of performing the double suspension $\Homol^\Gamma_{*+1}\big(\Sigma^2 A\big) \cong\Homol^\Gamma_*(\Sigma A)\cong\Homol^\Gamma_{*-1}(\cA)$ (suspending the first copy of~$(0,1)$ first and then the second copy versus the other way around) coincide up to~the sign $-1$, where the sign arises from the reflection of the square~$(0,1)^2$ along the diagonal.
How the proof works can be seen in the follow-up article \cite[Lemma~4.6.3]{wulff2020secondary}, where a more complicated analogous statement is shown for suspension in coarse (co-)homology.
\end{proof}

Now, let $X$ be a countably generated proper isocoarse $\Gamma$-space of bornologically bounded geometry and $X'\subset X$ a discretization.
For each coarse component $K$ of $X$, we furthermore consider a Higson-dominated corona $\partial K$ in a way compatible with the $\Gamma$-action. More precisely, we assume that their coproduct $\partial X\coloneqq\coprod_{K\in\coarsecomp(X)}\partial K$ is a locally compact $\Gamma$-space and moreover a quotient $\Gamma$-space of $\coprod_{K\in\coarsecomp(X)}\partial_hK$. We can attach each $\partial K$ to $\cP(X'\cap K)$ and obtain the $\sigma$-locally compact $\Gamma$-space $\overline{\cP(X')}\coloneqq \coprod_{K\in\coarsecomp(X)}\overline{\cP(X'\cap K)}$, which contains $\cP(X')$ as an open and~$\partial X$ as a closed subspace.
We consider all of them as equipped with the decomposition indexed by $\coarsecomp(X)$.\footnote{Note however, that the notation is slightly misleading: Here, the spaces $\partial X$ and $\overline{\cP(X')}$ are not ($\sigma$-)compact, if $\coarsecomp(X)$ is not finite.}

\begin{defn}\label{defn:reducedtransgression}
We call a locally compact $\Gamma$-space $\partial X\coloneqq\coprod_{K\in\coarsecomp(X)}\partial K$ as above together with its decomposition a \emph{componentwise Higson dominated corona} of $X$.
The associated \emph{reduced transgression maps} are the connecting homomorphisms
\[
\HomolX_{*+1}^\Gamma(X)\to\widetilde\Homol_*^\Gamma(\partial X)\qquad\text{and}\qquad \widetilde\Cohom^*_\Gamma(\partial X)\to \CohomX^{*+1}_\Gamma(X)
\]
in the long exact sequences of reduced (co-)homology associated to the pair $(\overline{\cP(X')},\partial X)$.
\end{defn}

We now obtain the following reduced analog of Lemma~\ref{lem:relativetransgressionisos}.
\begin{lem}If each $\overline{\cP(X'\cap K)}$ is contractible and all the contractions together are compatible with the $\Gamma$-action, then the reduced transgression maps are isomorphisms.
\end{lem}

{\sloppy\begin{proof}The hypothesis means nothing else than that the $\sigma$-locally compact $\Gamma$-space $\coprod_{K\in\coarsecomp(X)}\overline{\cP(X'\cap K)}$ is homotopy equivalent to $\coarsecomp(X)$, canonically considered as a discrete $\Gamma$-space. But its reduced (co-)homology groups vanish, because $\cO^+\coarsecomp(X)=(0,1]\times \coarsecomp(X)$ is homotopy equivalent to the empty set. The claim now follows from the exact sequences \eqref{eq:multireducedLEShomol} and \eqref{eq:multireducedLEScohom}.
\end{proof}}

As an example, the central Theorem~5.7 of~\cite{EngelWulff} says that the hypothesis is satisfied in the non-equivariant case for the combing compactification of a proper metric space equipped with an expanding and coherent combing.

Let us close this section by mentioning the following relation between reduced and unreduced co-assembly.\footnote{Compare to the end of Section~\ref{sec:sHigCor}.} If $X$ is a countably generated proper isocoarse $\Gamma$-space of bornologically bounded geometry and $K\subset X$ is a subspace whose intersection with each coarse component of $X$ is bounded, then we define the countably generated isocoarse $\Gamma$-space $X^\to\coloneqq X\cup_KK\times\N$ of bornologically bounded geometry by equipping it with the obvious coarse structure. Then reduced transgression for $X^\to$ canonically identifies with relative transgression of the pair $(X^\to,K\times\N)$ and with unreduced transgression for~$X$.

\section{Further examples}\label{sec:FurtherExamples}

In the following subsections we take a closer look at coarsifications of some specific directly constructed \mbox{(co-)}homology theories on $\sigma$-spaces and partly also their transgression maps.

\subsection{Singular (co-)homology}
Locally finite singular homology $\HS^{\lf}_*(\blank,\blank;M)$ and compactly supported singular cohomology $\HS_{\cs}^*(\blank,\blank;M)$ with coefficients in an abelian group $M$ are known to be \mbox{(co-)}homology theories on $\LCHz[]$. Let $\CS^{\lf}_*(\blank,\blank,M)$ and $\CS_{\cs}^*(\blank,\blank,M)$ denote their respective \mbox{(co-)}chain complexes. If $\Gamma$ is in addition a discrete group which acts on $M$ from the left, then we can perform the following modifications to the \mbox{(co-)}chain complexes to obtain \mbox{(co-)}homology theories on $\sLCHz$.
First we pass to the subcomplexes $\CS^{\lf,\Gamma}_*(\blank,\blank;M)\subset\CS^{\lf}_*(\blank,\blank;M)$ of $\Gamma$-equivariant chains and the quotient complex $\CS_{\cs,\Gamma}^*(\blank,\blank;M)\coloneqq\Gamma\backslash\CS_{\cs}^*(\blank,\blank;M)$, just as we did in Sections~\ref{sec:ordinarycoarsehomology} and~\ref{sec:ordinarycoarsecohomology}, to obtain functors on $\LCHz$. And second, if $\cX=\bigcup_{n\in\N}X_n$ is a $\sigma$-locally compact $\Gamma$-space with closed subspace $\cA=\bigcup_{n\in\N}A_n$, then we define
\begin{gather*}
\CS^{\lf,\Gamma}_*(\cX,\cA;M)\coloneqq\varinjlim_{n\in\N}\CS^{\lf,\Gamma}_*(X_n,A_n;M),
\\
\CS_{\cs,\Gamma}^*(\cX,\cA;M)\coloneqq\varprojlim_{n\in\N}\CS_{\cs,\Gamma}^*(X_n,A_n;M),
\end{gather*}
{\sloppy where all the maps in the directed systems are injective or surjective, respectively.
The usual proofs go through to show that these \mbox{(co-)}chain complexes define \mbox{(co-)}homology theories $\HS^{\lf}_*(\blank,\blank;M)$ and $\HS_{\cs}^*(\blank,\blank;M)$ on $\sLCHz$ and for the sake of computations we note that we have $\HS^{\lf,\Gamma}_*(\cX,\cA;M)\cong\varinjlim_{n\in\N}\HS^{\lf,\Gamma}_*(X_n,A_n;M)$ and a Milnor-${\varprojlim}^1$-sequence
\[
0\to {\varprojlim_{n\in\N}}^1\HS_{\cs,\Gamma}^{*-1}(X_n,A_n;M)\to \HS_{\cs,\Gamma}^*(\cX,\cA;M)\to \varprojlim_{n\in\N}\HS_{\cs,\Gamma}^*(X_n,A_n;M)\to 0.
\]}

Now, if $(X',A')$ is a pair of locally finite countably generated proper isocoarse $\Gamma$-spaces, then its ordinary $\Gamma$-equivariant coarse \mbox{(co-)}chain complex is nothing else but the $\Gamma$-equivariant locally finite/compactly supported simplicial \mbox{(co-)}chain complex of the pair of $\sigma$-simplicial complexes $(\cP(X'),\cP(A'))$ (defined in the completely analogous way). The usual proofs using barycentric subdivision can be adapted to the present $\sigma$-situation with $\Gamma$-actions to show that the singular and simplicial \mbox{(co-)}homology of $(\cP(X'),\cP(A'))$ coincide. This shows $\HX^\Gamma_*(X',A';M)\cong \HS^{\lf,\Gamma}_*(\cP(X'),\cP(A');M)$ and $\HX_\Gamma^*(X',A';M)\cong \HS_{\cs,\Gamma}^*(\cP(X'),\cP(A');M)$, that is, the ordinary $\Gamma$-equivariant coarse \mbox{(co-)}homology is the coarsification of $\Gamma$-equivariant locally finite/compactly supported singular \mbox{(co-)}homology.

\subsection{Alexander--Spanier (co-)homology}\label{sec:AlexanderSpanier}

Unfortunately, singular \mbox{(co-)}homology does not satisfy strong excision, so they do not provide transgression maps. Instead one has to use other \mbox{(co-)}homology theories like Alexander--Spanier cohomology and its dual.

Alexander--Spanier cohomology is known to be a single-space cohomology theory on $\LCHzp[]$ (see in particular \cite[Section 6.6]{Spanier66} or \cite[Section~1.7]{MasseyHomologyCohomologyBook} for a proof of the strong excision property).
By taking inverse limits of Alexander--Spanier cochain complexes one can easily generalize this cohomology theory to a single-space cohomology theory on $\sLCHzp[]$ whose coarsification is the (non-equivariant) ordinary coarse cohomology. As a consequence we obtain transgression maps from the Alexander--Spanier cohomology of Higson-dominated coronas to the coarse cohomology of a space, and one can indeed show that they agree with Roe's original transgression map from \cite[Section~5.3]{RoeCoarseCohomIndexTheory}.
All of this was done in \cite[Section~4.4]{EngelWulff}.

Care has to be taken if one tries to dualize Alexander--Spanier cohomology (or \v{C}ech coho\-mo\-logy) in order to obtain a homology theory which satisfies the strong excision axiom, because most attempts of dualization destroy the long exact sequences. Massey pointed out how to do it correctly \cite{MasseyHomologyCohomologyBook,MasseyHowTo}, providing us with a single-space homology theory on $\LCHzp[]$, which we call Alexander--Spanier homology. It can be generalized easily to $\sLCHzp[]$ by taking direct limits.
If $(X',A')$ is a pair of uniformly locally finite countably generated coarse spaces, then all the pairs of Rips complexes $(P_E(X'),P_E(A'))$ are finite dimensional and the Alexander--Spanier homology of these pairs are isomorphic to their locally finite cellular homology by \cite[Section~4.9(4)]{MasseyHomologyCohomologyBook} and hence also to their locally finite simplicial homology. Consequently, for spaces of coarsely bounded geometry the coarsification of Alexander--Spanier homology is (non-equivariant) ordinary coarse homology and we obtain transgression maps from ordinary coarse homology to the Alexander--Spanier homology of Higson-dominated coronas.

The proofs of the basic properties of Alexander--Spanier \mbox{(co-)}homology are somewhat involved, and therefore some non-trivial effort would probably have to be invested to see if $\Gamma$-actions can be implemented into the definitions in the naive ways seen in Sections~\ref{sec:ordinarycoarsehomology} and~\ref{sec:AlexanderSpanier}. We are not going to pursue this question any further.

\subsection[Coarse K-theory and coassembly]{Coarse $\boldsymbol{\K}$-theory and coassembly}\label{sec:coassembly}

Given a $\sigma$-locally compact $\Gamma$-space $\cX=\bigcup_{n\in\N}X_n$ and a $\Gamma$-\textCstar-algebra $D$, the associated function $*$-algebra
\[
\Cz(\cX;D)\coloneqq \{f\colon\cX\to D\mid \forall n\in\N\colon f|_{X_n}\in\Cz(X_n;D)\}
\]
equipped with the countable family of \textCstar-seminorms $\|f\|_n\coloneqq\|f|_{X_n}\|$ and the pull-back action $\gamma\cdot f\colon x\mapsto \gamma\cdot (f(x\gamma))$ is a so-called $\sigma$-$\Gamma$-\textCstar-algebra, i.e., it is a complex topological $*$-algebra whose topology is Hausdorff, generated by a countable family of $\Gamma$-equivariant \textCstar-seminorms and is complete with respect to this family of seminorms. Equivalently, it is the inverse limit in the category of topological complex $*$-algebras of an inverse system of $\Gamma$-\textCstar-algebras indexed over the natural numbers (cf.\ \cite{PhiInv}).

Referring to a previous remark, the $\Gamma$-action on a $\sigma$-$\Gamma$-\textCstar-algebra is isometric with respect to all the seminorms. Instead one could also pose the weaker condition that the $\Gamma$-action is only continuous with respect to the topology on the whole algebra. The latter results in what we would call $\Gamma$-$\sigma$-\textCstar-algebras, i.e., the $\sigma$ and $\Gamma$ are interchanged to account for the different order in which they are implemented into the definition.

Via pullback of functions, we obtain a contravariant functor $\Cz(\blank;D)\colon\sLCHp\to\sCs$
into the category of $\sigma$-$\Gamma$-\textCstar-algebras and $\Gamma$-equivariant $*$-homomorphisms, the latter being automatically continuous. In the case $D=\C$ this functor is actually an equivalence between the categories $(\sLCHp)^{\mathrm{op}}$ and the category of commutative $\sigma$-$\Gamma$-\textCstar-algebras $\csCs$, i.e., a Gelfand--Naimark type duality.

If $\blank\rtimes\Gamma$ is any crossed product functor on the category of $\Gamma$-\textCstar-algebras, then we obtain a~corresponding crossed product functor from $\sCs$ to $\sCs[]$ by applying it to an associated inverse system of $\Gamma$-\textCstar-algebras and taking the inverse limit afterwards. It is a direct consequence of \cite[Proposition~5.3(2)]{PhiInv} that the resulting crossed product functor is exact, if we started with an exact crossed product functor.
We can now extend $\Ktop_*$ to $\sigma$-$\Gamma$-\textCstar-algebras by turning~\eqref{eq:KtopDef} into a definition, as was done in \cite{EmeMeyDescent,EmeMeyEquivariantCoassembly}.

\begin{defn}
Using Phillips' $\K$-theory for $\sigma$-\textCstar-algebras \cite{PhiRep,PhiFre},
for any exact crossed product functor and any $\sigma$-$\Gamma$-\textCstar-algebra $C$ we define
\[
\Ktop_*(\Gamma,C)\coloneqq \K_*((C\maxtensor\sfP)\maxcrossed\Gamma).
\]
Furthermore, the \emph{$\Gamma$-equivariant $\K$-theory with coefficients in a $\Gamma$-\textCstar-algebra $D$} is defined as
\[
\K_\Gamma^*(\blank;D)\coloneqq \Ktop_{-*}(\Gamma,\Cz(\blank;D)).
\]
\end{defn}

\begin{lem}The functors $\K_\Gamma^*(\blank,D)$ and $\K_{-*}(\Cz(\blank;D)\rtimes\Gamma)$ for any exact crossed product functor $\blank\rtimes\Gamma$ are $\Gamma$-equivariant single-space cohomology theories on $\sLCHp$ and they all agree on proper $\sigma$-locally compact $\Gamma$-spaces.
\end{lem}

\begin{proof}The first statement follows directly from the long exact sequences and homotopy invariance of Phillips' $\K$-theory together with continuity and exactness of the maximal tensor product and the crossed product functors in consideration.

If $X$ is a proper locally compact $\Gamma$-space, then $\Cz(X;D)$ is a proper $\Gamma$-\textCstar-algebra and \cite[Remark~3.4.16]{echterhoff_KK_BC_overview} implies that all crossed products $\Cz(X;D)\rtimes\Gamma$, even the non-exact ones, coincide. Thus, the same is true for $\Cz(\cX;D)\rtimes \Gamma$ if $\cX$ is a proper $\sigma$-locally compact $\Gamma$-space.

As observed in \cite[Section~2.7]{EmeMeyDescent}, the Baum--Connes assembly maps yield isomorphisms \linebreak$\Ktop_{-*}(\Gamma,C)=\K_{-*}((C\maxtensor\sfP)\maxcrossed\Gamma)\cong\K_{-*}(C\maxcrossed\Gamma)$ for all proper $\sigma$-$\Gamma$-\textCstar-algebras $C$. Applying this to $C=\Cz(\cX;D)$ for proper $\sigma$-locally compact $\Gamma$-spaces $\cX$ finishes the proof.
\end{proof}

\begin{defn}By the second part of the preceding lemma, the coarsifications of $\K_\Gamma^*(\blank,D)$ and $\K_{-*}(\Cz(\blank;D)\rtimes\Gamma)$ for all exact crossed product functors $\blank\rtimes\Gamma$ agree and we call it the \emph{$\Gamma$-equivariant coarse $\K$-theory} $\KX_\Gamma^*(\blank,D)$.
\end{defn}

We are also interested in the reduced $\K$-theory groups of $\sigma$-locally compact $\Gamma$-spaces $\cX$ equipped with a decomposition $\cX=\coprod_{i\in \pi}\cK_i$ as in Section~\ref{sec:reducedtransgression}. In \cite[Section~2.3]{EmeMeyEquivariantCoassembly}, Emerson and Meyer gave a definition thereof for undecomposed $\sigma$-compact $\Gamma$-spaces $\cK$ as the $\K$-theory of a~certain $\sigma$-$\Gamma$-\textCstar-algebra, which we shall call $\CRed(\cK;D)$. We recall and generalize their construction and show that it agrees with the reduced groups obtained via Definition~\ref{def:suspensionandreduction} if all of the $\cK_i$ are nonempty.

Given a non-empty compact Hausdorff space $K$, Emerson and Meyer define $\CRed(K;D)$ as the sub-$\Gamma$-\textCstar-algebra of $\Ct(K;\multiplier^s(D))$ consisting of those functions for which $f(x)-f(y)\in D\otimes\Kom$ for all $x,y\in K$.
For the empty space, we deviate and simply define $\CRed(\varnothing;D)\coloneqq\frac{\multiplier^s(D)}{D\otimes\Kom}$.
If~$\cK=\bigcup_{n\in\N}K_n$ is a $\sigma$-compact Hausdorff space, then we obtain the $\sigma$-$\Gamma$-\textCstar-algebra
\[
\CRed(\cK;D)\coloneqq\big\{f\colon\cK\to \multiplier^s(D)\mid \forall n\in\N\colon f|_{K_n}\in\CRed(K_n;D)\big\}=\varprojlim_{n}\CRed(K_n;D).
\]

Now we generalize it to locally compact Hausdorff spaces $X$ decomposed into compact subsets $K_i$, which are allowed to be empty. To this end, for each $f\in \CRed(K;D)$ we let $\Var(f)\coloneqq\sup_{x,y\in K}\|f(x)-f(y)\|$ if $K\not=\varnothing$ and $\Var(f)\coloneqq 0$ if $K=\varnothing$.
Then we define
\[
\CzRed(X;D)\coloneqq\bigg\{(f_i)_{i\in \pi}\in \prod_{i\in \pi}\CRed(K_i;D)\bigg|\big(\pi\mapsto [0,\infty),i\mapsto\Var(f_i)\big)\in \Cz(\pi)\bigg\},
\]
recalling that we also consider $\pi$ as a discrete space.
Note that it can be identified as a~sub-\textCstar-algebra of $\Cb(X;\multiplier^s(D))$ if all $K_i$ are non-empty.

For the $\sigma$-locally compact space $\cX=\bigcup_{n\in\N}X_n$ decomposed as above into the $\cK_i$ we then define the $\sigma$-\textCstar-algebra
\[
\CzRed(\cX;D)\coloneqq\varprojlim_{n\in\N}\CzRed(X_n;D),
\]
where the locally compact subspaces $X_n$ of the filtration of $\cX$ are decomposed into the compact subsets $X_n\cap \cK_i$.
Note that it is even a $\sigma$-$\Gamma$-\textCstar-algebra if $\cX$ is a $\sigma$-locally compact $\Gamma$-space and the decomposition is compatible with the $\Gamma$-action as in Section~\ref{sec:reducedtransgression}.

The purpose of this cryptic construction is to have the short exact sequence
\begin{equation}\label{eq:reducedcontinuousfunctionalgebrasSES}
0\to\Cz(\cX;D)\otimes\Kom\to \CzRed(\cX;D)\to\frac{\prod_{\pi}\multiplier^s(D)}{\bigoplus_{\pi}(D\otimes\Kom)}\to 0
\end{equation}
of $\sigma$-$\Gamma$-\textCstar-algebras.

{\sloppy\begin{lem}\label{lem:reducedKtheory}
The reduced groups of the cohomology theories $\K^*(\blank;D)$, $\K_\Gamma^*(\blank;D)$ and $\K_{-*}(\Ct(\blank;D\otimes\Kom)\rtimes\Gamma)$ are canonically naturally isomorphic to $\K_{-*}\big(\CRed(\blank;D)\big)$, $\Ktop_{-*}\big(\Gamma,\CRed(\blank;D)\big)$ and $\K_{-*}\big(\CRed(\blank;D\otimes\Kom)\rtimes\Gamma\big)$, respectively, in such a way that the corresponding sequences \eqref{eq:generalizedreducedunreducedsequence} are isomorphic to the long exact sequences induced by~\eqref{eq:reducedcontinuousfunctionalgebrasSES}.
\end{lem}}

\begin{proof}
We reuse the proof of Lemma~\ref{lem:reducedunreducedstablehigsoncoronacone}, but this time with the $\sigma$-$\Gamma$-\textCstar-algebras $I=\Cz(\cX;D)\allowbreak\otimes\Kom$, $A=\CzRed(\cX;D)$, $B\coloneqq\prod_{\pi}\multiplier^s(D)$ and $J\coloneqq \bigoplus_{\pi}(D\otimes\Kom)$. Note that here we have
\[
\Cone\biggl(\iota'\colon \bigoplus_{\pi}(D\otimes\Kom)\to \Cz(\cX;D)\biggr)\cong \Cz(\cO^+\cX;D)\otimes\Kom
\]
and that the $*$-homomorphism $\Cone(\iota')\to \bigoplus_{\pi}(D\otimes\Kom)= \Cz(\pi;D)\otimes\Kom$ is exactly the evaluation at the apices of the cones. The claim follows.
\end{proof}

The $\Gamma$-equivariant coarse $\K$-theory is the target of the so-called coassembly map, which we recall next. Let $X$ be a proper metric isometric $\Gamma$-space. Then for any discretization $X'$ as in Lemma~\ref{lem:goodequivariantdiscretization}, all the Rips complexes $P_n(X')=P_{E_n}(X')$ carry an obvious $\Gamma$-isocoarse coarse structure compatible with the topology such that all of the inclusions $X\supset X'\subset P_m(X')\subset P_n(X')$ for $m\leq n$ are coarse equivalences. It follows that $\sHigCorRed(\cP(X');D)\coloneqq\varprojlim_{n}\sHigCorRed(P_n(X');D)$
is a $\Gamma$-\textCstar-algebra isomorphic to $\sHigCorRed(X;D)$,
because all of the maps in the inverse system are isomorphisms.
Furthermore we define the $\sigma$-$\Gamma$-\textCstar-algebra
$\sHigComRed(\cP(X');D)\coloneqq\varprojlim_{n}\sHigComRed(P_n(X');D)$.
By \cite{EmeMeyDualizing} they fit into a short exact sequence
\[
0\to\Cz(\cP(X');D\otimes\Kom)\to\sHigComRed(\cP(X');D)\to\sHigCorRed(\cP(X');D)\cong\sHigCorRed(X;D)\to 0.
\]
If $A\subset X$ is a closed subspace and a suitable pair of discretizations $(X',A')$ has been chosen, then we obtain analogously the short exact sequence
\[
0\to\Cz(\cP(X')\setminus\cP(A');D\otimes\Kom)\to\sHigCom(\cP(X'),\cP(A');D)\to\sHigCor(X,A;D)\to 0,
\]
where $\sHigCom(\cP(X'),\cP(A');D)\coloneqq\ker(\sHigCom(\cP(X');D)\to\sHigCom(\cP(A');D))$.

\begin{defn}
The connecting homomorphisms in $\K$-theory
\begin{gather*}
\mu^*\colon\quad\, \K_{1-*}\big(\sHigCorRed(X;D)\big)\to\KX^*(X;D),
\\[.5ex]
\hphantom{\mu^*\colon\quad\,}\K_{1-*}(\sHigCor(X,A;D))\to\KX_{\phantom{\Gamma}}^*(X,A;D),
\\[.5ex]
\mu_{\mathrm{top}}^*\colon\ \Ktop_{1-*}\big(\Gamma;\sHigCorRed(X;D)\big)\to\KX_\Gamma^*(X;D),
\\[.5ex]
\hphantom{\mu_{\mathrm{top}}^*\colon\ }\Ktop_{1-*}(\Gamma;\sHigCor(X,A;D))\to\KX_\Gamma^*(X,A;D),
\\[.5ex]
\mu_\rtimes^*\colon \ \ \,\K_{1-*}\big(\sHigCorRed(X;D)\rtimes\Gamma\big)\to\KX_\Gamma^*(X;D),
\\[.5ex]
\hphantom{\mu_\rtimes^*\colon\ \ \, }\K_{1-*}(\sHigCor(X,A;D)\rtimes\Gamma)\to\KX_\Gamma^*(X,A;D),
\end{gather*}
associated to the above short exact sequences are all called \emph{co-assembly maps}.
\end{defn}
The three co-assembly maps in the absolute case have originally been defined in \cite{EmeMeyDualizing,EmeMeyDescent,EngelWulffZeidler}, respectively.

It is usually $\mu_{\mathrm{top}}^*$ which is of interest, because it is known to be an isomorphism in many important cases, whereas not much can be said about $\mu_\rtimes^*$ \cite[p.~72]{EmeMeyEquivariantCoassembly}.
According to \cite[Lemma~5.2]{EngelWulffZeidler}, the co-assembly map $\mu_{\mathrm{top}}^*$ factors through $\mu_\rtimes^*$ in the absolute case, but the same also holds true for the corresponding relative versions. Therefore, if one is only interested in surjectivity (as was the case in \cite{EngelWulffZeidler}), then $\mu_\rtimes^*$ is good enough.

\begin{thm}\label{thm:transgressionfactorsthroughcoassembly}
The transgression maps associated to the cohomology theories $\K^*(\blank,D)$, $\K_\Gamma^*(\blank,D)$ and $\K_{-*}(\Cz(\blank;D\otimes\Kom)\rtimes\Gamma)$ factor through the assembly maps as follows:
\begin{gather*}
\widetilde\K^{*-1}(\partial X;D)\to\K_{1-*}\big(\sHigCorRed(X;D)\big)\to\KX^*(X;D),
\\[.5ex]
\K^{*-1}(\partial X,\partial A;D)\to\K_{1-*}(\sHigCor(X,A;D))\to\KX^*(X,A;D),
\\[.5ex]
\widetilde\K_\Gamma^{*-1}(\partial X;D)\to\Ktop_{1-*}\big(\Gamma;\sHigCorRed(X;D)\big)\to\KX_\Gamma^*(X;D),
\\[.5ex]
\K_\Gamma^{*-1}(\partial X,\partial A;D)\to\Ktop_{1-*}(\Gamma;\sHigCor(X,A;D))\to\KX_\Gamma^*(X,A;D),
\\[.5ex]
\K_{1-*}\big(\CRed(\partial X;D)\rtimes\Gamma\big)\to\K_{1-*}\big(\sHigCorRed(X;D)\rtimes\Gamma\big)\to\KX_\Gamma^*(X;D),
\\[.5ex]
\K_{1-*}(\Cz(\partial X\setminus\partial A;D\otimes\Kom)\rtimes\Gamma)\to\K_{1-*}(\sHigCor(X,A;D)\rtimes\Gamma)\to\KX_\Gamma^*(X,A;D).
\end{gather*}
Thus, if a transgression map is an isomorphism $($e.g., because $\overline{\cP(X')}$ is contractible or $\overline{\cP(A')}$ is a deformation retract of $\overline{\cP(X')})$, then the corresponding co-assembly map is surjective.
\end{thm}

\begin{proof}
The claims follow directly from the commutative diagrams with exact rows
\[
\xymatrix@C=2ex{
0\ar[r]&\Cz(\cP(X');D\otimes\Kom)\ar[r]\ar@{=}[d]&\CRed(\overline{\cP(X')};D)\ar[r]\ar[d]&\CRed(\partial X;D)\ar[r]\ar[d]& 0
\\0\ar[r]&\Cz(\cP(X');D\otimes\Kom)\ar[r]&\sHigComRed(\cP(X');D)\ar[r]&\sHigCorRed(X;D)\ar[r]& 0
\\0\ar[r]&\Cz(\cP(X')\setminus\cP(A');D)\ar[r]\ar[d]&\Cz(\overline{\cP(X')}\setminus\overline{\cP(A')};D)\ar[r]\ar[d]&\Cz(\partial X\setminus\partial A;D)\ar[r]\ar[d]& 0
\\0\ar[r]&\Cz(\cP(X')\setminus\cP(A');D\otimes\Kom)\ar[r]&\sHigCom(\cP(X'),\cP(A');D)\ar[r]&\sHigCor(X,A;D)\ar[r]& 0,
}
\]
which exist because functions on a Higson-dominated compactification restrict to functions of vanishing variation.
\end{proof}

\subsection[Coarse K-homology and assembly]{Coarse $\boldsymbol{\K}$-homology and assembly}\label{sec:Assembly}

Using $\Gamma$-equivariant $\EE$-theory,\footnote{Of course, $\KK$-theory could also be used. In that case, a different description of the assembly maps would need to be used and hence completely new proofs of our results would be required. However, $\KK$-theory suffers from similar technical shortcomings as $\EE$-theory and therefore no advantage over our approach is expected. In the end, $\EE$-theory was primarily selected because of the author's greater familiarity with this theory.} we define the $\Gamma$-equivariant $\K$-homology with coefficients in a~separable $\Gamma$-\textCstar-algebra $D$ as the functors
\[
\K^\Gamma_*(\blank;D)\coloneqq \EE^\Gamma_*(\Cz(\blank);D).
\]
Note that this definition requires the function algebras $\Cz(\blank)$ to be separable, so we only consider these functors on the full subcategory of $\LCHp$ of all \emph{second countable} locally compact $\Gamma$-spaces.
In practice, this is only a minor restriction to the theory of Section~\ref{sec:Sigmastuff}, because we are mainly interested in two special cases: First, Rips complexes of discretizations of countably generated proper isocoarse $\Gamma$-spaces of bornologically bounded geometry are always second countable and, second, their compactification with a Higson dominated corona is second countable, if the corona is second countable, i.e., a compact metrizable space. The latter is an ubiquitous assumption in applications.

Due to the well-known properties of bivariant $\K$-theory (cf.\ \cite[Chapter~6]{GueHigTro}), the functors $\K^\Gamma_*(\blank;D)$ are $\Gamma$-equivariant single-space homology theories on the above-mentioned category. Furthermore, as direct limits preserve exact sequences, they can be extended to $\Gamma$-equivariant single-space homology theories on the full subcategory of $\sLCHp$ consisting of all second countable $\sigma$-locally compact $\Gamma$-spaces $\cX=\bigcup_{n\in\N}X_n$ by
\(\K^\Gamma_*(\cX;D)\coloneqq\lim_{n\to\infty}\K^\Gamma_*(X_n;D)\).

We want to compare the resulting transgression maps with the coarse assembly maps $\mu$: $\KX^\Gamma_*(X;D)\to\K_*(\Roe_\Gamma(X;D))$. The definition of the latter is based upon other pictures of $\K$-homology.
We choose the one using Yu's localization algebras \cite{YuLocalization} (see also \cite{QiaoRoe} and \cite{EngelWulffZeidler}) for our discussion.
A lot about localization algebras has been proven in the comprehensive book~\cite{WillettYuHigherIndexTheory}, but note that the ``localized Roe algebras'' considered there are slightly bigger than Yu's original localization algebras and hence a bit of care has to be taken when applying their results.

Following up our Section~\ref{sec:RoeAlg} about Roe algebras, let $X$ be a proper metric space equipped with a proper isometric $\Gamma$-action. We assume that sufficiently large $\Gamma$-$X$-module $(H,\rho,u)$ has been chosen and let $\HilbertMod\coloneqq H\otimes D$ be equipped with the induced representations. Recall that we defined $\Roe_\Gamma(X;D)$ as the norm closure in $\Lin(\HilbertMod)$ of the $*$-subalgebra of all $\Gamma$-equivariant locally compact operators of finite propagation, and if $A$ is a subspace of $X$, then $\Roe_\Gamma(A\subset X;D)$ denoted the ideal generated by all operators which are supported in an $R$-neighborhood of $A$ for some $R>0$ and $\Roe_\Gamma(X,A;D)\coloneqq \Roe_\Gamma(X;D)/\Roe_\Gamma(A\subset X;D)$.

\begin{defn}
The \emph{localization algebra} $\Loc(X;D)$ is the norm closure of the sub-$*$-algebra of $\Cb([1,\infty),\Roe_\Gamma(X;D))$ of all bounded uniformly continuous functions $L\colon[1,\infty)\to\Roe_\Gamma(X;D)$ such that the propagation of $L(t)$ is finite for all $t\geq 1$ and converges to zero as $t\to\infty$. Furthermore, we denote by $\Loc(A\subset X;D)\subset \Loc(X;D)$ the ideal generated by all $L$ for which there is a~function $\varepsilon\colon[1,\infty)\to(0,\infty)$ with $\varepsilon(t)\xrightarrow{t\to\infty}0$ such that $L(t)$ is supported in an $\varepsilon(t)$-neighborhood of $A$ for all $t\geq 1$, and we define $\Loc(X,A;D)\coloneqq \Loc(X;D)/\Loc(A\subset X;D)$.
\end{defn}

An obvious variant of \cite[Lemma~6.3.6]{WillettYuHigherIndexTheory}, which can be proven with the same methods, says that the inclusion $A\subset X$ induces via the theory of covering isometries an isomorphism $\K_*(\Loc(A;D))\cong\K_*(\Loc(A\subset X;D))$. Therefore we obtain a long exact sequence
\begin{align}
\cdots&\to\K_*(\Loc(A;D))\to \K_*(\Loc(X;D))\to \K_*(\Loc(X,A;D))\nonumber
\\&\to \K_{*-1}(\Loc(A;D))\to\cdots.\label{eq:localizationLES}
\end{align}

The relation between localization algebras and our definition of $\K$-homology is the following. As $\Kom(\HilbertMod)\cong D\otimes\Kom(H)$, it is straightforward to verify (cf.\ \cite[Corollary~4.2]{QiaoRoe}) that
\begin{align*}
\delta\colon\ \Loc(X;D)\maxtensor\Cz(X) &\to\frac{\Cb([1,\infty),D\otimes \Kom(H))}{\Cz([1,\infty),D\otimes\Kom(H))},
\\
L\otimes f&\mapsto [t\mapsto L(t)\circ (\rho(f)\otimes\id_D)]
\end{align*}
defines an $\Gamma$-equivariant asymptotic morphism. It vanishes on the ideal $\Loc(A\subset X;D)\maxtensor\Cz(X\setminus A)$ and hence we also obtain an induced $\Gamma$-equivariant asymptotic morphism
\[
\delta\colon\ \Loc(X,A;D)\maxtensor\Cz(X\setminus A) \to\frac{\Cb([1,\infty),D\otimes \Kom(H))}{\Cz([1,\infty),D\otimes\Kom(H))}.
\]
Both of them define canonical $\EE^\Gamma$-theory elements in $\EE^\Gamma_0(\Loc(X;D)\maxtensor\Cz(X),D)$ and \linebreak $\EE^\Gamma_0(\Loc(X,A;D)\maxtensor\Cz(X\setminus A),D)$, respectively, which we shall also denote by the letter~$\delta$.\footnote{Here we need $\EE^\Gamma$-theory also for non-separable \textCstar-algebras. It can be defined in an ad hoc way as in \cite[Section~1]{WulffTwisted}, or even simpler by letting $\EE^\Gamma_*(A,B)$ for possibly non-separable $A$, $B$ be obtained from the groups $\EE^\Gamma_*(A',B')$ by taking the direct limit over all separable $B'\subset B$ and the inverse limit over all separable $A'\subset A$. The relevant maps considered in the remainder of this section will be the same for both versions.
Note that these generalizations retain the composition and tensor products of $\EE^\Gamma_*$-theory, but be aware that they are simply too naive to preserve the long exact sequences. }
We obtain the composed group homomorphism
\begin{align*}
\Delta\colon\ \K_*(\Loc(X;D))&\cong\EE_*(\C,\Loc(X;D))
\\&\xrightarrow{\blank\otimes\id_{\Cz(X)}}\EE^\Gamma_*(\Cz(X),\Loc(X;D)\maxtensor\Cz(X))
\\&\xrightarrow{\delta\circ\blank}\EE^\Gamma_*(\Cz(X),D)=\K^\Gamma_*(X;D)
\end{align*}
and similarily $\Delta\colon \K_*(\Loc(X,A;D))\to\K^\Gamma_*(X,A;D)$ in the relative case.
By exploiting the homological properties of domains and targets of the maps $\Delta$ (in particular that the domains are functorial under uniformly continuous coarse $\Gamma$-maps via the theory of covering isometries, that the maps $\Delta$ are natural under such maps and that the connecting homomorphisms are in both cases constructed via mapping cones) we see that the maps $\Delta$ map \eqref{eq:localizationLES} to the long exact sequence of $\K^\Gamma_*(\blank,D)$.

\begin{thm}\label{thm:Khomologydefinitions}
The maps $\Delta$ are isomorphisms in the absolute case for every proper metric space~$X$ equipped with a proper isometric $\Gamma$-action and hence by a five lemma argument also in the relative case.\footnote{Instead of Theorem~\ref{thm:Khomologydefinitions}, preprint versions of this article contained the ``Assumption 6.8'' that $\Gamma$ and $D$ have been chosen for which the statement of this theorem holds. Thanks to a referee's hint, we can now prove it in full generality.}
\end{thm}

Proofs of some special cases of the theorem have been known for a long time.
In the non-equivariant case ($\Gamma=1$) without coefficients ($D=\C$), one proof is sketched in \cite[Construction~6.7.5, Exercises~B.3.2 and~B.3.3]{WillettYuHigherIndexTheory}, albeit for a different version of the localization algebra. Note, however, that this proof relies solely on homological properties and a comparison on one-point spaces. Therefore, it goes through for original version of the localization algebras, too, and even with arbitrary coefficient \textCstar-algebras $D$.

Another proof for the case $\Gamma=1$ and $D=\C$ can be found in \cite[Proposition~4.3]{QiaoRoe} and relies on Paschke duality. The author is indebted to a referee for pointing out a version of Paschke duality for the equivariant $\KK$-theory groups $\KK^\Gamma_*(\Cz(X),D)$ which allows to prove the theorem in full generality.
Nevertheless, as localization algebras and $\EE$-theory are defined in a very similiar manner, it would be interesting to know if the theorem can be proven directly without taking the detour via Paschke duality.

\begin{proof}
It suffices to prove the theorem for separable $D$, because the general case follows by taking direct limits over all separable sub-\textCstar-algebras.
In this case, we can make use of the Paschke duality for equivariant $\KK$-theory that was discussed in \cite[Appendix~B]{GueWillYu_DyncomplKtheory}.

An adjointable operator $T$ on $\HilbertMod$ is called pseudolocal if the commutators $\rho(f)T-T\rho(f)$ are compact for all $f\in \Cz(X)$.
Let $\PseudoLoc_\Gamma(X;D)$ denote the sub-\textCstar-algebra of $\Lin(\HilbertMod)$ generated by the $\Gamma$-equivariant pseudolocal operators that have finite propagation and let $\PseudoLocLoc(X;D)$ denote the sub-\textCstar-algebra of $\Cb([1,\infty),\PseudoLoc_\Gamma(X;D))$ of all bounded uniformly continuous functions $L\colon[1,\infty)\to\PseudoLoc(X;D)$ such that the propagation $L(t)$ is finite for all $t\geq 1$ and converges to zero as $t\to\infty$.\looseness=-1

We are free to choose a very specific sufficiently large $\HilbertMod=H\otimes D$, namely $\HilbertMod=\HilbertMod_D\coloneqq H'\otimes \ell^2(\N)\otimes\ell^2(\Gamma)\otimes D$ for a non-degenerate $\Gamma$-$X$-module $(H',\rho',u')$.
It will be equipped with $\Gamma$-action $\varepsilon$ which acts diagonally on the first, third and fourth factor and the representation $\pi$ of $\Cz(X)$ which acts on the first factor.
For this choice, \cite[Appendix B]{GueWillYu_DyncomplKtheory} contains the proof of isomorphisms
\begin{gather*}
\KK_p^\Gamma(\Cz(X),D)\cong\K_{p+1}\biggl(\frac{\PseudoLoc_\Gamma(X;D)}{\Roe_\Gamma(X;D)}\biggr) \cong\K_{p+1}\biggl(\frac{\PseudoLocLoc(X;D)}{\Loc(X;D)}\biggr)\cong\K_p(\Loc(X;D))
\end{gather*}
for separable $\Gamma$-\textCstar-algebras $D$ and proper metric spaces $X$ equipped with a proper cocompact isometric $\Gamma$-action.

The cocompactness assumption can be removed by performing slight modifications to the aforementioned proof: Instead of $\Gamma$-propagation we only use metric propagation (cf.\ \cite[items~(ii) and~(iii) on p.~74]{GueWillYu_DyncomplKtheory}; in the co-compact case, finite $\Gamma$-propagation of a $\Gamma$-equivariant operator is equivalent to finite metric propagation anyway).
Furthermore, the proof utilises a cut-off function~$c$, i.e., a compactly supported continuous function $X\to [0,1]$ such that $\sum_{\gamma\in \Gamma}c(x\cdot\gamma)^2=1$ for all $x\in X$. If the action is not co-compact, then we have to use the weaker condition on the support of $c$ that it intersect only finitely many $\Gamma$-translates of each compact subset. The existence of such a cut-off function follows easily from the properness of the action.
With these modifications, the calculations in the proof of \cite[Lemma~B.7]{GueWillYu_DyncomplKtheory} go through. Just at the end of that proof one has to perform another truncation a la \cite[Sublemma~12.3.3]{HigRoe} by an equivariant partition of unity to turn the operator with finite $\Gamma$-propagation into one with finite metric propagation.
As \cite[Lemmas~B.14, B.15 and B.16]{GueWillYu_DyncomplKtheory} generalize to the non-cocompact case without any problems, the abovementioned isomorphisms follow.

Let us briefly recall the definition of the abovementioned isomorphisms in degree $p=0$.
According to the analogue of \cite[Lemma B.7]{GueWillYu_DyncomplKtheory} in even degree\footnote{The cited lemma only considers the odd case, but the even case can be proven completely analogously.}, every class in $\KK_*^\Gamma(\Cz(X),D)$ can be represented by a cycle of the form $\left(\HilbertMod_D\oplus\HilbertMod_D,\left(\begin{smallmatrix}0&U^*\\U&0\end{smallmatrix}\right)\!,\epsilon\oplus\epsilon,\pi\oplus\pi\right)$ with $U\in \PseudoLoc_\Gamma(X;D)$.
Then $U$ is a unitary modulo $\Roe_\Gamma(X;D)$ and the Paschke duality isomorphism is defined by mapping this cycle to the class represented by $U$. The second isomorphism then takes it to the class represented by a function $t\mapsto U_t$ which represents a unitary in $\PseudoLocLoc(X;D)/\Loc(X;D)$ and such that $U_0$ agrees with $U$ modulo $\Roe_\Gamma(X;D)$. Such a function can be constructed from the constant function $U$ by applying the procedure from the proof of \cite[Lemma~B.15]{GueWillYu_DyncomplKtheory}.
The third isomorphism then maps it to the class represented by the difference of the projection valued functions $t\mapsto P_t$ and $t\mapsto Q$ defined by
\[
P_t\coloneqq \begin{pmatrix}U_tU_t^*&U_t(\id-U_t^*U_t)^{\frac12} \\U_t^*(\id-U_tU_t^*)^{\frac12}&\id-U_t^*U_t\end{pmatrix}\qquad\text{and}\qquad Q=\begin{pmatrix}\id&0\\0&0\end{pmatrix}\!,
\]
cf.\ \cite[Proposition~4.8.10]{HigRoe}.

In order to prove the claim in even degrees, it now suffices to show that $\Delta$ composed with these isomorphisms is the canonical natural isomorphism $\KK_0^\Gamma(\Cz(X),D)\xrightarrow{\cong}\EE_0^\Gamma(\Cz(X),D)$.
The latter is defined via \cite[Theorem~2.2]{Thomsen_univKK} and hence we have to verify two properties: That the composition is natural in $D$ and that it maps $1_{\Cz(X)}\in\KK_0^\Gamma(\Cz(X),\Cz(X))$ to $1_{\Cz(X)}\in\EE_0^\Gamma(\Cz(X),\Cz(X))$. Naturality is clear from the description of the three isomorphisms above and the definition of~$\Delta$.

For the verification of the other property, recall that $1_{\Cz(X)}\in\KK_0^\Gamma(\Cz(X),\Cz(X))$ is represented by the cycle $(\Cz(X)\oplus 0,0,\varepsilon'\oplus 0,\pi'\oplus 0)$, where $\varepsilon'$ and $\pi'$ denote the canonical representations of $\Gamma$ and $\Cz(X)$. After adding the degenerate cycle $\left(\HilbertMod_{\Cz(X)}\oplus\HilbertMod_{\Cz(X)}, \left(\begin{smallmatrix}0&\id\\\id&0\end{smallmatrix}\right)\!,\epsilon\oplus\epsilon,\pi\oplus\pi\right)$, we see that it is also represented by the cycle $\left(\HilbertMod_{\Cz(X)}\oplus\HilbertMod_{\Cz(X)}, \left(\begin{smallmatrix}0&U^*\\U&0\end{smallmatrix}\right)\!,\epsilon\oplus\epsilon,\pi\oplus\pi\right)$ for a~projection $U\colon \HilbertMod_{\Cz(X)}\cong\Cz(X)\oplus\HilbertMod_{\Cz(X)}\to\HilbertMod_{\Cz(X)}$ onto the second summand. Since $U$ is a module homomorphism, it has propagation zero and $U\in \PseudoLocLoc(X;\Cz(X))$. Furthermore we have $UU^*=\id$ and $p\coloneqq \id-U^*U$ is a projection onto an isomorphically embedded copy of $\Cz(X)$ in $\HilbertMod_{\Cz(X)}$.
For the calculation of the image of this class in $\EE_0^\Gamma(\Cz(X),D)$ we can choose the constant function $t\mapsto U_t\coloneqq U$. We end up with the element that is represented by the $*$-homomorphism $\Cz(X)\to \Kom(\HilbertMod_{\Cz(X)})$, $f\mapsto \pi(f)p$, and this is exactly $1_{\Cz(X)}$.

The isomorphism in odd degrees now follows from the one in even degrees by replacing $D$ with the suspension $D\otimes \Cz(\R)$.
\end{proof}

The reason why we do not just choose the $\K$-theory of the localization algebras as definition of $\Gamma$-equivariant $\K$-homology with coefficients is that we prefer to have a definition that works also for non-proper group actions, but the proofs of the homological properties of $\K_*(\Loc(\blank;\C))$ in~\cite{WillettYuHigherIndexTheory}, in particular the existence of the all important covering isometries underlying functoriality, rely heavily on properness.
Note also that the equivariant Paschke duality that we used in the proof of Theorem~\ref{thm:Khomologydefinitions} works only for proper $\Gamma$-actions, so the two definitions might truely be different in the non-proper case.
Unfortunately, this makes the proof of Theorem~\ref{thm:transgressionfactorsthroughassembly} below much more complicated than the argument using Paschke duality which was presented in \cite[Remark~12.3.8]{HigRoe} for the non-equivariant case without coefficients.

\begin{defn}
Let $(X,A)$ be a pair of proper metric spaces equipped with a proper isometric $\Gamma$-action.
The(\emph{uncoarsified$)$ coarse assembly maps}
\begin{gather*}
\begin{split}
&\mu\colon\ \K^\Gamma_*(X;D)\cong \K_*(\Loc(X;D))\xrightarrow{(\ev1)_*}\K_*(\Roe_\Gamma(X;D)),
\\
&\mu\colon\ \K^\Gamma_*(X,A;D)\cong \K_*(\Loc(X,A;D))\xrightarrow{(\ev1)_*}\K_*(\Roe_\Gamma(X,A;D))
\end{split}
\end{gather*}
are induced by the evaluation at one.
\end{defn}

The coarsified coarse assembly map will be obtained by taking the direct limit over all the uncoarsified assembly maps of the Rips complexes $(P_n(X'),P_n(A'))$ of a discretization $(X',A')\subset (X,A)$.
As we have chosen a description of the assembly maps which requires proper metrics, we can't content ourselves with the canonical topological coarse structures on the Rips complexes, but we have to metrize them.
Most authors equip the Rips complexes $P_n(X')$ with the spherical metric from \cite{Roe_Hyperbolic}, but they have the drawback that the inclusions $X'\subset P_n(X')$ are in general not coarse equivalences if the original space $X$ was not a path-metric space.
Therefore, we use the $\Gamma$-invariant metrics from the following lemma, such that all of the inclusions $X\supset X'\subset P_m(X')\subset P_n(X')$ for $m\leq n$ are even isometric coarse equivalences.
For a similar approach to metrizing the Rips complexes, see \cite[Section~7.2]{WillettYuHigherIndexTheory}.

\begin{lem}
We can assign to each discrete proper metric isometric $\Gamma$-space $X'$ a $\Gamma$-invariant metric $d_{\cP(X')}$ on $\cP(X')$ such that
\begin{enumerate}\itemsep=0pt
\item[$\bullet$] its restriction to $P_0(X')=X'$ is the original metric $d_{X'}$,
\item[$\bullet$] its restriction to each $P_n(X')$ is a $\Gamma$-invariant \emph{proper} metric inducing the already existing topology on the complex,
\item[$\bullet$] for all $m\leq n$, $P_m(X')$ is coarsely dense in $P_n(X')$,
\item[$\bullet$] if $f\colon X'\to Y'$ is a coarse $\Gamma$-map between discrete proper metric isometric $\Gamma$-spaces and $X'$ is $\delta$-separated for some $\delta>0$, i.e., the distance between each two distinct points of $X'$ is at least $\delta$, then all of the restrictions
\[
\cP(f)|_{P_n(X')}\colon\ P_n(X')\to \cP(Y')
\]
are low-scale Lipschitz coarse maps. Here, by low-scale Lipschitz we mean that there are constants $L,\delta'>0$ such that the Lipschitz inequality
\[
d_{\cP(Y')}(f(x),f(y))\leq L\cdot d_{\cP(X')}(x,y)
\]
holds for all $x,y\in P_n(X')$ with $d_{\cP(X')}(x,y)<\delta'$.
\end{enumerate}
\end{lem}

The last property is very useful when working with the localization algebras, because the direct description of the functoriality of their $\K$-theory using families of covering isometries (in~the concrete version of \cite[Definition 4.29 and~Proposition 4.30]{EngelWulffZeidler}) is only available for uniformly continuous coarse maps.
Note that $\delta$-separatedness is never a problem, because all discretizations either already have this property by construction (cf.\ the proof of Lemma~\ref{lem:topbbg} applied to the entourage $E_\delta$) or can be thinned out such that they satisfy this property.
\begin{proof}
For $R>0$, the $R$-spherical metric on a $k$-simplex is the path metric obtained by identifying it with the subset
\[
S^k_+(R)\coloneqq\big\{(x_0,\dots,x_k)\mid x_0,\dots,x_k\geq 0 \wedge x_0^2+\dots+x_k^2=R^2\big\}\subset\R^{k+1}.
\]
On each $P_n(X')$ with $n\geq 1$ we first define the auxiliary metric $d^1_{P_n(X')}$ as the largest metric such that its restriction to each simplex is at most the $2n$-spherical metric.
This metric is clearly a~proper path metric and it induces the given topology, but it may take the value $\infty$ and is not yet the restriction of the metric $d_{\cP(X')}$ that we are looking for. Also, note that $X'$ is $\pi n$-dense in~$P_n(X')$ with respect to $d^1_{P_n(X')}$.
On $P_0(X')=X'$ we simply define $d^1_{P_0(X')}=d_{X'}$ instead.

Thanks to the choice $R=2n$, similar arguments to those in the proof of \cite[Proposition~3.4(iv)]{Roe_Hyperbolic} (note also \cite{Roe_Hyperbolic_Erratum}) show that the inclusion $X'\subset P_n(X')$ is distance non-decreasing with respect to $d_{X'}$ and $d^1_{P_n(X')}$ for each $n\in\N$. Hence there is a largest metric $d^2_{P_n(X')}$ on $P_n(X')$ which is less or equal to $d^1_{P_n(X')}$ and restricts to $d_{X'}$ on $X'$. These new metrics have the property that the inclusions $P_m(X')\subset P_n(X')$ are also distance non-decreasing with respect to $d^2_{P_m(X')}$ and $d^2_{P_n(X')}$ for all $m\leq n$.

Using all of these properties, it is now readily verified that there is a largest metric on $d_{\cP(X')}$ on $\cP(X')$ such that all of the restrictions to $P_n(X')$ are less or equal to $d^2_{P_m(X')}$ and this metric clearly satisfies the first three properties on the list and is $\Gamma$-invariant.

Now, if $f\colon X'\to Y'$ and $\delta>0$ are as in the forth item on the list, then for each $n\in\N$ there is an $S_n\in\N$ such that $d_{X'}(x,y)\leq n$ implies $d_{Y'}(f(x),f(y))\leq S_n$. Note that $\cP(f)$ restricts to a $\frac{S_n}{n}$-Lipschitz map from $\big(P_n(X'),d^1_{P_n(X')}\big)$ to $\big(P_{S_n}(Y'),d^1_{P_{S_n}(Y')}\big)$ for every $n\geq 1$ and, therefore, if $x,y\in P_n(X')$ with $d_{\cP(X')}(x,y)<\frac{\delta}{2}$, then $d_{\cP(Y')}(f(x),f(y))\leq L\cdot d_{\cP(X')}(x,y)$ with $L\coloneqq\max\bigl(S_1,\frac{S_2}{2},\dots,\frac{S_n}{n}\bigr)$.
\end{proof}

\begin{defn}
Taking direct limits over the uncoarsified assembly maps of the Rips complexes, we obtain the (\emph{coarsified$)$ coarse assembly maps},
\begin{gather*}
\mu\colon\ \KX^\Gamma_*(X;D)=\varinjlim_{n\in\N}\K^\Gamma_*(P_n(X');D)
\to\varinjlim_{n\in\N}\K_*(\Roe_\Gamma(P_n(X');D))
\cong \K_*(\Roe_\Gamma(X;D)),
\\
\mu\colon\ \KX^\Gamma_*(X,A;D)=\varinjlim_{n\in\N}\K^\Gamma_*(P_n(X'),P_n(A');D)
\\ \hphantom{\mu\colon\ \KX^\Gamma_*(X,A;D)}
{}\to\varinjlim_{n\in\N}\K_*(\Roe_\Gamma(P_n(X'),P_n(A');D))
\cong \K_*(\Roe_\Gamma(X,A;D)).
\end{gather*}
\end{defn}

\begin{thm}\label{thm:transgressionfactorsthroughassembly}
Let $X$ be a proper metric space $X$ equipped with a proper isometric $\Gamma$-action and let $\partial X$ be a componentwise Higson dominated corona as defined in Definition~$\ref{defn:reducedtransgression}$. Then the reduced transgression map factors through the assembly maps as follows:
\[
\KX^\Gamma_*(X;D)\xrightarrow{\mu}\K_*(\Roe_\Gamma(X;D))\xrightarrow{\nu}\tilde\K^\Gamma_{*-1}(\partial X;D).
\]
Similarily, if $A$ is a subspace of $X$ and $(\partial X,\partial A)$ a pair of Higson-dominated coronas,
then the relative transgression map factors through the relative assembly maps:
\[
\KX^\Gamma_*(X,A;D)\xrightarrow{\mu}\K_*(\Roe_\Gamma(X,A;D))\xrightarrow{\nu}\K^\Gamma_{*-1}(\partial X,\partial A;D).
\]
In particular, if the transgression maps are isomorphisms, then the assembly maps are injective.
\end{thm}

\begin{proof}We first prove the reduced version in detail. Write $\partial X\coloneqq\coprod_{K\in\coarsecomp(X)}\partial K$ and let $\overline{X}\coloneqq \coprod_{K\in\coarsecomp(X)}\overline{K}$ be the decomposed locally compact $\Gamma$-space obtained by taking the Higson dominated compactification $\overline{K}$ corresponding to $\partial K$ of each coarse component $K\in\coarsecomp(X)$.
The map $\nu$ is then defined as follows.
The restrictions of functions in $\Cz\big(\overline{X}\big)$ to $X$ have vanishing variation on each coarse component and hence also on the whole space, so they commute with operators in $\Roe_\Gamma(X;D)$ up to compact operators as in the proof of \cite[Proposition 5.18]{RoeCoarseCohomIndexTheory}.
Thus, there is a $\Gamma$-equivariant $*$-homomorphism
\[
\Roe_\Gamma(X;D)\otimes\Cz(\overline{X})\to\frac{\multiplier(D\otimes\Kom(H))}{D\otimes\Kom(H)},\qquad T\otimes f\mapsto [T\circ(\rho(f)\otimes\id_D)]
\]
which vanishes on $\Roe_\Gamma(X;D)\otimes\Cz(X)$ and hence factors through $\Roe_\Gamma(X;D)\otimes\Cz(\partial X)$.
Note that its restriction to $\Roe_\Gamma(X;D)\otimes\Cz(\coarsecomp(X))$ lifts to an $\Gamma$-equivariant $*$-homomorphism into $\multiplier(D\otimes\Kom(H))$ when $\Cz(\coarsecomp(X))$ is considered as the sub-$\Gamma$-\textCstar-algebra of $\Cz\big(\overline{X}\big)$ via pull-back under $\overline{X}\to\coarsecomp(X)$.
Identifying $\Cz(\cO^+(\partial X))$ with $\Cone(\Cz(\coarsecomp(X))\to\Cz(\partial X))$ and applying the $*$-homomorphism above slice-wise, we obtain a $\Gamma$-equivariant $*$-homomorphism
\[
\delta'\colon\ \Roe_\Gamma(X;D)\otimes\Cz(\cO^+(\partial X))\to\Cone\biggl(\multiplier(D\otimes\Kom(H))\to \frac{\multiplier(D\otimes\Kom(H))}{D\otimes\Kom(H)}\biggr).
\]
The inclusion of $D\otimes \Kom(H)$ as a canonical ideal in the mapping cone to the right induces an isomorphism on $\EE^\Gamma$-theory and hence we can define the map $\nu$ as the composition
\begin{align*}
\K_*(\Roe_\Gamma(X;D))&\cong\EE_*(\C,\Roe_\Gamma(X;D))
\\
&\xrightarrow{\blank\otimes\id_{\Cz(\cO^+(\partial X))}}\EE^\Gamma_*(\Cz(\cO^+(\partial X)),\Roe_\Gamma(X;D)\maxtensor\Cz(\cO^+(\partial X)))
\\
&\xrightarrow{\delta'\circ\blank}\EE^\Gamma_*\biggl(\Cz(\cO^+(\partial X)),\Cone\biggl(\multiplier(D\otimes\Kom(H))\to \frac{\multiplier(D\otimes\Kom(H))}{D\otimes\Kom(H)}\biggr)\biggr)
\\
&\cong\EE^\Gamma_{*-1}(\Cz(\cO^+(\partial X)),D\otimes\Kom(H))
=\widetilde\K^\Gamma_*(\partial X;D).
\end{align*}

The definition of $\nu$ is functorial under $\Gamma$-equivariant coarse maps and thus it identifies with the corresponding maps $\nu_n$ associated to the locally compact $\Gamma$-spaces $\overline{X}_n$ obtained from $X_n\coloneqq P_n(X')$ by gluing $\partial X$ to it. We have to show that the diagram
\[
\xymatrix@C+3ex{
\K_*(\Loc(X_n;D))\ar[r]^{(\ev{1})_*}\ar[d]^{\cong}_{\Delta}&\K_*(\Roe_\Gamma(X_n;D))\ar[d]^{\nu_n}
\\\K^\Gamma_*(X_n;D)\ar[r]^{\partial}&\widetilde\K^\Gamma_{*-1}(\partial X;D)
}
\]
commutes and then the claim follows by taking the direct limit for $n\to\infty$.
This will be done by considering for each $x\in \K_*(\Loc(X_n;D))$ the following diagram in the $\EE^\Gamma$-category, in which all solid arrows come from $*$-homomorphisms and all dashed arrows from asymptotic morphisms or $\EE^\Gamma$-theory elements:
\[
\xymatrix{
{\begin{array}{l}
\Cz(0,1)\\\otimes\Cz(X_n)
\end{array}}\ar[r]^-{\simeq}\ar@{-->}[d]^{x\otimes\id}
&\Cone {\begin{pmatrix}
\Cz\big(\cO^+\big(\overline{X}_n\big)\big)\\\downarrow\\\Cz(\cO^+(\partial X))
\end{pmatrix}}\ar@{-->}[d]^{x\otimes\id}
&{\begin{array}{l}
\Cz(0,1)\\\otimes\Cz(\cO^+(\partial X))
\end{array}}\ar[l]\ar@{-->}[d]^{x\otimes\id}
\\{\begin{array}{l}
\Loc(X_n;D)\\\otimes\Cz(0,1)\\\otimes\Cz(X_n)
\end{array}}\ar[r]^-{\simeq}\ar@{-->}[d]^{\delta\otimes\id}
&\Loc(X_n;D)\otimes\Cone {\begin{pmatrix}
\Cz\big(\cO^+\big(\overline{X}_n\big)\big)\\\downarrow\\\Cz(\cO^+(\partial X))
\end{pmatrix}}\ar@{-->}[d]^{\delta_1}
&{\begin{array}{l}
\Loc(X_n;D)\\\otimes\Cz(0,1)\\\otimes\Cz(\cO^+(\partial X))
\end{array}}\ar[l]\ar@{-->}[d]^{\delta_2\otimes\id}
\\{\begin{array}{l}
\Cz(0,1)\otimes\\ D\otimes\Kom(H)
\end{array}}\ar[r]^-{\simeq}
&\Cone {\begin{pmatrix}
\Cz(0,1]\otimes\multiplier(D\otimes\Kom(H))\\\downarrow\\{\scriptstyle\Cone\left(\multiplier(D\otimes\Kom(H))\to\frac{\multiplier(D\otimes\Kom(H))}{D\otimes\Kom(H)}\right)}
\end{pmatrix}}
&{\begin{array}{l}
\Cz(0,1)\otimes \\\Cone{\begin{pmatrix}{\scriptstyle\multiplier(D\otimes\Kom(H))}\\\downarrow\\\frac{\multiplier(D\otimes\Kom(H))}{D\otimes\Kom(H)}\end{pmatrix}}
\end{array}}\ar[l]_-{\simeq}
\\{\begin{array}{l}
\Cz(0,1)\otimes \\D\otimes\Kom(H)
\end{array}}\ar[r]^-{\simeq}\ar@{=}[u]
&\Cone {\begin{pmatrix}
\Cz(0,1]\otimes\Kom(H)\otimes D\\\downarrow\\D\otimes\Kom(H)
\end{pmatrix}}\ar[u]_-{\simeq}
&{\begin{array}{l}
\Cz(0,1)\otimes\\ D\otimes\Kom(H)
\end{array}}\ar[l]_-{\simeq}\ar[u]_-{\simeq}
}
\]

The rows of the diagram are all of the form
\[
J\to \Cone(A\twoheadrightarrow A/J)\leftarrow \Cz(0,1)\otimes A/J
\]
for different \textCstar-algebras $A$ with ideals $J$. The arrows going upwards are also induced by $*$-homomorphisms of the form $J\to \Cone(A\twoheadrightarrow A/J)$. All the arrows marked with ``$\simeq$'' are isomorphisms in $\EE^\Gamma$-theory, because they are inclusions of ideals with contractible quotient \textCstar-algebras. Furthermore, the mapping cone in the middle of the bottom row is itself isomorphic to $\Cz(0,1)\otimes D\otimes\Kom(H)$ and the maps leading into it then come from the inclusion of $(0,1)$ into itself as the left or right half, up to orientation. Hence they represent the same element in $\EE^\Gamma$-theory up to a sign.

Using \cite[Lemma 6.1.2]{WillettYuHigherIndexTheory}, which clearly also holds true for operators on Hilbert modules, and the fact that functions on $\overline X_n$ restrict to functions of vanishing variation on $X_n$, we see that
\begin{align*}
\Loc(X;D)\maxtensor\Ct\big(\overline X_n\big) &\to\frac{\Cb([1,\infty),\multiplier(D\otimes\Kom(H)))}{\Cz([1,\infty),\multiplier(D\otimes\Kom(H)))},
\\
L\otimes f&\mapsto [t\mapsto L(t)\circ (\rho(f)\otimes\id_D)]
\end{align*}
is an equivariant asymptotic morphism.
It clearly restricts to $\delta$ on the ideal $\Loc(X;D)\otimes\Cz(X_n)$ and also passes to an equivariant asymptotic morphism between the quotients
\[
\Loc(X;D)\otimes\Ct(\partial X)\to\frac{\Cb([1,\infty),\multiplier(D\otimes\Kom(H))/D\otimes\Kom(H))}{\Cz([1,\infty), \multiplier(D\otimes\Kom(H))/D\otimes\Kom(H))}.
\]
Applying these constructions slicewise yields the asymptotic morphisms $\delta_1$ and $\delta_2$.
In particular, we note that $\delta_2$ is in fact given by the family of $*$-homomorphisms $(\delta'\circ\ev{t})_{t\in[1,\infty)}$ and as such it is $1$-homotopic to the constant family $(\delta'\circ\ev{1})_{t\in[1,\infty)}$. Hence $\delta_2$ and $\delta'$ represent the same $\EE^\Gamma$-theory class.

It is clear that the diagram commutes.
The left column now is $\Delta(x)\in\K^\Gamma_{*+1}((0,1)\times X_n;D)\cong \K^\Gamma_*(X_n;D)$ and the right column is $\nu_n\circ(\ev{1})_*(x)\in \K^\Gamma_{*}(\cO^+(\partial X);D)\cong \tilde\K^\Gamma_{*-1}(\partial X;D)$. The middle column represents an element in $\K^\Gamma_{*+1}\big(\cO^+\big(\overline{X}_n\big)\times\{1\}\cup \cO(\partial X)\times(0,1];D\big)$ which is mapped to both $\Delta(x)$ and $\nu_n\circ(\ev{1})_*(x)$ in the obvious way. By definition of the connecting homomorphisms~$\partial$ associated to the pair $\big(\cO^+\big(\overline X_n\big),\cO^+(\partial X)\big)$ via the mapping cone construction, this shows exactly that $\partial(\Delta(x))=\nu_n\circ(\ev{1})_*(x)$.

The proof of the relative case works completely analogously but a lot simpler by showing that for each $x\in \K_*(\Loc(X_n,A_n;D))$ a diagram
\[
\xymatrix{
\Cz(X_n\setminus A_n)
\ar[r]^-{\simeq}\ar@{-->}[d]^-{x\otimes\id}
&\Cone {\begin{pmatrix}
\Cz\big(\overline{X}_n\setminus\overline{A_n}\big)\\\downarrow\\\Cz(\partial X\setminus\partial A))
\end{pmatrix}}
\ar@{-->}[d]^-{x\otimes\id}
&{\begin{array}{l}
\otimes\Cz(0,1)\\\otimes\Cz(\partial X\setminus\partial A)
\end{array}}
\ar[l]\ar@{-->}[d]^-{x\otimes\id}
\\{\begin{array}{l}
\Loc(X_n,A_n;D)\\\otimes\Cz(X_n\setminus A_n)
\end{array}}\ar[r]^-{\simeq}\ar@{-->}[d]^{\delta}
&{\begin{array}{l}
\Loc(X_n,A_n;D)\\\otimes\Cone {\begin{pmatrix}
\Cz\big(\overline{X}_n\setminus\overline{A_n}\big)\\\downarrow\\\Cz(\partial X\setminus\partial A))
\end{pmatrix}}\end{array}}\ar@{-->}[d]
&{\begin{array}{l}
\Loc(X_n,A_n;D)\\\otimes\Cz(0,1)\\\otimes\Cz(\partial X\setminus\partial A)
\end{array}}\ar[l]\ar@{-->}[d]
\\D\otimes\Kom(H)
\ar[r]^-{\simeq}
&\Cone{\begin{pmatrix}\multiplier(D\otimes\Kom(H))\\\downarrow\\\frac{\multiplier(D\otimes\Kom(H))}{D\otimes\Kom(H)}\end{pmatrix}}
&\Cz(0,1)\otimes \frac{\multiplier(D\otimes\Kom(H))}{D\otimes\Kom(H)}
\ar[l]_-{\simeq}
}
\]
commutes in the $\EE^\Gamma$-category, where $\Delta(x)$ is the left column and $\nu((\ev{1})_*(x))$ is the composition of the right column with the bottom row.
\end{proof}

\subsection*{Acknowledgements}

The author would like to thank Alexander Engel for insightful conversations.
Furthermore, the author is also very grateful to the anonymous referees for their meticulous inspection of the manuscript, for their long list of comments and for pointing out numerous inaccuracies.
Their suggestions, critizism and partially also strong opinions had a significant positive influence on this paper.

The research was supported by the DFG through the Priority Programme ``Geometry at Infinity'' (SPP 2026; individual project ``Duality and the coarse assembly map'', WU 869/1-1, WU 869/1-2).

\LastPageEnding

\end{document}